\documentclass{amsart}
 \usepackage{amsmath,amssymb}
 \usepackage{times}
  \usepackage{mathrsfs}

 \usepackage[all, cmtip]{xy}

 \usepackage[plainpages=false,colorlinks,hyperindex,pdfpagemode=None,bookmarksopen,linkcolor=red,citecolor=blue,urlcolor=blue]{hyperref}
\usepackage{pdflscape}
\usepackage{multirow}

 \usepackage{graphicx}
 \usepackage{leftindex}
 \swapnumbers

\newtheorem{theorem}[subsection]{Theorem}
\newtheorem{cor}[subsection]{Corollary}
\newtheorem{lemma}[subsection]{Lemma}

\newtheorem{prop}[subsection]{Proposition}
 \newtheorem{assumption}[subsection]{Assumption}
\theoremstyle{definition}

\newtheorem{definition}[subsection]{Definition}
\newtheorem{remark}[subsection]{Remark}
\newtheorem*{remark*}{Remark}

\newtheorem*{theorem*}{Theorem}
\newtheorem*{defn*}{Definition}
\newtheorem*{lemma*}{Lemma}
\newtheorem*{corollary*}{Corollary}
\newtheorem{question*}{Question}
\newtheorem*{conjecture*}{Conjecture}
 \newtheorem*{prop*}{Proposition}
 \newtheorem*{example*}{Example}
 \newtheorem*{assumption*}{Assumption}

\usepackage{tikz,tikz-3dplot}
\usetikzlibrary{matrix}
\usetikzlibrary{external}
\usetikzlibrary{decorations.pathreplacing}

\newcommand{\cC}{\mathcal{C}}
\newcommand{\cD}{\mathcal{D}}
\newcommand{\cE}{\mathcal{E}}

\newcommand{\cM}{\mathcal{M}}

\newcommand{\cQ}{\mathcal{Q}}
\newcommand{\cR}{\mathcal{R}}
\newcommand{\cS}{\mathcal{S}}
\newcommand{\cT}{\mathcal{T}}
\newcommand{\cU}{\mathcal{U}}
\newcommand{\cV}{\mathcal{V}}
\newcommand{\cX}{\mathcal{X}}

\newcommand{\bA}{\mathbf{A}}
\newcommand{\bC}{\mathbf{C}}

\newcommand{\bF}{\mathbf{F}}
\newcommand{\bI}{\mathbf{I}}

\newcommand{\bL}{\mathbf{L}}

\newcommand{\bR}{\mathbf{R}}
\newcommand{\bS}{\mathbf{S}}
\newcommand{\bZ}{\mathbf{Z}}

\newcommand{\LCH}{\mathrm{LCHS}}
\newcommand{\Mod}{\mathrm{Mod}}
\newcommand{\Fun}{\mathrm{Fun}}
\newcommand{\Hom}{\mathrm{Hom}}

\newcommand{\LMod}{\mathrm{LMod}}

\newcommand{\GL}{\mathrm{GL}}

\newcommand{\End}{\mathrm{End}}

\newcommand{\Spec}{\mathrm{Spec}}
\newcommand{\Pic}{\mathrm{Pic}}

\newcommand{\op}{\mathrm{op}}

\newcommand{\St}{\mathrm{St}}
\newcommand{\Cat}{\mathrm{Cat}}

\newcommand{\sfF}{\mathsf{F}}

\renewcommand{\SS}{\mathit{SS}}

\newcommand{\Brane}{\mathrm{Brane}}
\newcommand{\sm}{\mathit{sm}}

\numberwithin{equation}{subsection}

\newcommand{\pt}{\mathit{pt}}

\newcommand{\mumon}{\mu\mathrm{mon}}

\newcommand{\XZ}{\pi_{\mathrm{f}}}
\newcommand{\XY}{\pi_{\mathrm{lag}}}
\newcommand{\Sh}{\mathrm{Sh}}
\newcommand{\MSh}{\mathrm{MSh}}
\newcommand{\Loc}{\mathrm{Loc}}
\newcommand{\bk}{\mathbf{k}}

\newcommand{\shHom}{\underline{\Hom}}

\newcommand{\LagGr}{\mathrm{LagGr}}

\newcommand{\leftparen}{(}    
\newcommand{\rightparen}{)}  
\newcommand{\leftbracket}{[}
\newcommand{\rightbracket}{]}

\newcommand{\deltazero}{\leftparen -\delta,0\rightbracket}
\newcommand{\Czero}{\leftparen - C,0\rightbracket}
\newcommand{\inftyzero}{\leftparen -\infty,0\rightbracket}

\newcommand{\inftydelta}{\leftparen -\infty,-\delta\rightbracket}

\newcommand{\Shall}{\Sh_{\mathit{all}}}

\newcommand{\Shv}{\mathrm{Shv}}

\newcommand{\Cone}{\mathrm{Cone}}
\newcommand{\proj}{\mathrm{proj}}

\newcommand{\cL}{\mathcal{L}}

\newcommand{\bq}{\mathbf{q}}
\newcommand{\bp}{\mathbf{p}}
\newcommand{\PrL}{\mathcal{P}\mathrm{r}^{\mathrm{L}}}
\newcommand{\PrR}{\mathcal{P}\mathrm{r}^{\mathrm{R}}}

\newcommand{\Open}{\mathrm{Open}}

\newcommand{\Cylb}{\mathrm{Cylb}}
\newcommand{\Nb}{\mathrm{Nb}}
\newcommand{\QCoh}{\mathrm{QCoh}}
\newcommand{\fiber}{\mathrm{fiber}}
\newcommand{\unit}{\mathrm{unit}}
\newcommand{\counit}{\mathrm{co}\text{-}\mathrm{unit}}

\newcommand{\Ttotleft}{\overleftarrow{T}_{\mathit{tot}}}
\newcommand{\Ttotright}{\overrightarrow{T}_{\mathit{tot}}}
\newcommand{\Ttot}{T_{\mathit{tot}}}
\newcommand{\Supp}{\mathrm{Supp}}
\newcommand{\ep}{\epsilon}

\begin{document}

\title{Brane structures in microlocal sheaf theory}
\author{Xin Jin and David Treumann}

\maketitle

\begin{abstract}
Let $L$ be an exact Lagrangian submanifold of a cotangent bundle $T^* M$, asymptotic to a Legendrian submanifold $\Lambda \subset T^{\infty} M$.  We study a locally constant sheaf of $\infty$-categories on $L$, called the sheaf of brane structures or $\Brane_L$.  Its fiber is the $\infty$-category of spectra, and we construct a Hamiltonian invariant, fully faithful functor from $\Gamma(L,\Brane_L)$ to the $\infty$-category of sheaves of spectra on $M$ with singular support in $\Lambda$.
\end{abstract}

\tableofcontents

\section{Introduction}
This paper investigates the following corollary of the main result of \cite{NZ}:

\begin{theorem*}[Nadler-Zaslow]
Let $L$ be an exact Lagrangian submanifold of a cotangent bundle $T^* M$, asymptotic to a Legendrian submanifold $\Lambda \subset T^{\infty} M$.  If the brane obstructions on $L$ vanish, there is a fully faithful functor
\begin{equation}
\label{eq:one}
\Loc(L) \to \Sh_{\Lambda}(M)
\end{equation}
from local systems on $L$, to constructible sheaves on $M$ with singular support over $\Lambda$.
\end{theorem*}

Nadler and Zaslow deduce this theorem using Floer theory --- more precisely, they embed the category of local systems on $L$ into a Fukaya category of $T^* M$, and produce a full embedding  of this Fukaya category into the derived category of sheaves on $M$.  Our aim is to give a purely sheaf-theoretic construction of \eqref{eq:one}, one that is ``soft'' enough to apply to sheaves of spectra.

\tableofcontents

\subsection{Exact Lagrangians and wavefronts}
\label{intro:eLaw}
Let $\alpha$ denote the one-form
\begin{equation}
\label{intro:alpha}
\alpha = \xi_1 dx_1 + \cdots + \xi_n dx_n
\end{equation}
so that $-d\alpha$ is the standard symplectic form on $\bR^{2n}$.  A Lagrangian submanifold $L \subset \bR^{2n}$ is called exact if there is an $f:L \to \bR$ with $df = \alpha\vert_L$.  Exact Lagrangians are intensively studied, partly to avoid analytic issues that arise in Floer theory of more general Lagrangians.  A more intrinsic reason to study exact Lagrangians is a variant of the Arnold conjecture, which asserts that exact Lagrangians have a \emph{topological} nature: taken up to Hamiltonian isotopy, they have no moduli.

The theory of the wavefront projection shows one aspect of the topological nature of exact Lagrangians.  If $L$ is connected, the function $f$ is unique up to an additive constant.  The wavefront projection is the (immersed, singular) hypersurface in $\bR^{n+1}$ parametrized by $(x_1,\ldots,x_n,f)$.  One recovers $L$ from those $n+1$ coordinates as $\xi_i = \partial f/\partial x_i$.  In fact if $L$ is in general position, it can be uniquely recovered from just the image $\bF \subset \bR^{n+1}$ of the wavefront immersion $L \to \bR^{n+1}$.

\begin{example*}
There are no embedded exact Lagrangians in $\bR^{2n}$ that are compact --- we discuss this further in \S\ref{intro:compact-noncompact}.  But the wavefront map makes sense for immersed Lagrangians $L \to \bR^{2n}$, here is an example in $\bR^2$:
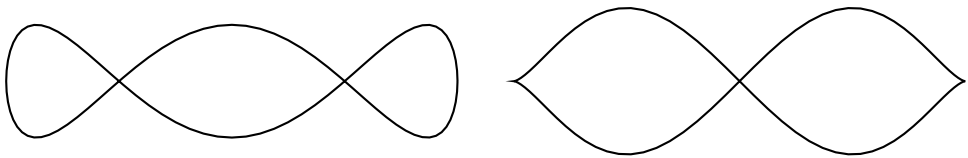
\begin{center}
\begin{tikzpicture}[scale = .75]
\draw [thick]
(4.00000000000000,0.000000000000000)--(3.99210691371309,0.187381314585725)--(3.96845880525791,0.368124552684678)--(3.92914900291475,0.535826794978997)--(3.87433264451452,0.684547105928689)--(3.80422606518061,0.809016994374947)--(3.71910594355301,0.904827052466020)--(3.61930820986408,0.968583161128631)--(3.50522672017545,0.998026728428272)--(3.37731170200806,0.992114701314478)--(3.23606797749979,0.951056516295154)--(3.08205297110316,0.876306680043863)--(2.91587450968565,0.770513242775789)--(2.73818842371475,0.637423989748690)--(2.54969595899476,0.481753674101716)--(2.35114100916989,0.309016994374947)--(2.14330717991599,0.125333233564305)--(1.92701469640686,-0.0627905195293133)--(1.70311716626029,-0.248689887164855)--(1.47249821073871,-0.425779291565072)--(1.23606797749979,-0.587785252292473)--(0.994759548659420,-0.728968627421411)--(0.749525258342899,-0.844327925502015)--(0.501332934257217,-0.929776485888251)--(0.251162078117254,-0.982287250728689)--(0.000000000000000,-1.00000000000000)--(-0.251162078117254,-0.982287250728689)--(-0.501332934257217,-0.929776485888251)--(-0.749525258342899,-0.844327925502016)--(-0.994759548659419,-0.728968627421412)--(-1.23606797749979,-0.587785252292473)--(-1.47249821073871,-0.425779291565072)--(-1.70311716626029,-0.248689887164855)--(-1.92701469640686,-0.0627905195293133)--(-2.14330717991599,0.125333233564304)--(-2.35114100916989,0.309016994374947)--(-2.54969595899476,0.481753674101716)--(-2.73818842371475,0.637423989748690)--(-2.91587450968565,0.770513242775789)--(-3.08205297110316,0.876306680043863)--(-3.23606797749979,0.951056516295154)--(-3.37731170200806,0.992114701314478)--(-3.50522672017545,0.998026728428272)--(-3.61930820986408,0.968583161128631)--(-3.71910594355301,0.904827052466019)--(-3.80422606518061,0.809016994374947)--(-3.87433264451452,0.684547105928690)--(-3.92914900291475,0.535826794978997)--(-3.96845880525791,0.368124552684678)--(-3.99210691371309,0.187381314585726)--(-4.00000000000000,0.000000000000000)--(-3.99210691371309,-0.187381314585725)--(-3.96845880525791,-0.368124552684677)--(-3.92914900291475,-0.535826794978997)--(-3.87433264451452,-0.684547105928689)--(-3.80422606518061,-0.809016994374947)--(-3.71910594355300,-0.904827052466019)--(-3.61930820986408,-0.968583161128631)--(-3.50522672017545,-0.998026728428272)--(-3.37731170200806,-0.992114701314478)--(-3.23606797749979,-0.951056516295154)--(-3.08205297110316,-0.876306680043863)--(-2.91587450968565,-0.770513242775789)--(-2.73818842371476,-0.637423989748691)--(-2.54969595899476,-0.481753674101716)--(-2.35114100916989,-0.309016994374947)--(-2.14330717991599,-0.125333233564304)--(-1.92701469640686,0.0627905195293108)--(-1.70311716626029,0.248689887164855)--(-1.47249821073871,0.425779291565072)--(-1.23606797749979,0.587785252292473)--(-0.994759548659421,0.728968627421410)--(-0.749525258342899,0.844327925502016)--(-0.501332934257218,0.929776485888251)--(-0.251162078117253,0.982287250728689)--(0.000000000000000,1.00000000000000)--(0.251162078117251,0.982287250728689)--(0.501332934257217,0.929776485888252)--(0.749525258342897,0.844327925502016)--(0.994759548659420,0.728968627421411)--(1.23606797749979,0.587785252292473)--(1.47249821073871,0.425779291565073)--(1.70311716626029,0.248689887164856)--(1.92701469640686,0.0627905195293119)--(2.14330717991598,-0.125333233564303)--(2.35114100916989,-0.309016994374947)--(2.54969595899476,-0.481753674101715)--(2.73818842371475,-0.637423989748689)--(2.91587450968564,-0.770513242775790)--(3.08205297110316,-0.876306680043864)--(3.23606797749979,-0.951056516295154)--(3.37731170200806,-0.992114701314478)--(3.50522672017545,-0.998026728428272)--(3.61930820986408,-0.968583161128631)--(3.71910594355300,-0.904827052466020)--(3.80422606518061,-0.809016994374947)--(3.87433264451452,-0.684547105928689)--(3.92914900291475,-0.535826794978998)--(3.96845880525791,-0.368124552684680)--(3.99210691371309,-0.187381314585724)--(4.00000000000000,0.000000000000000);
\draw [thick]
(13.0000000000000,0.000000000000000)--(12.9921069137131,-0.000988289981876861)--(12.9684588052579,-0.00781305011399713)--(12.9291490029148,-0.0258509997203336)--(12.8743326445145,-0.0595897113507078)--(12.8042260651806,-0.112256994144896)--(12.7191059435530,-0.185533741714553)--(12.6193082098641,-0.279369617411445)--(12.5052267201755,-0.391914399269005)--(12.3773117020081,-0.519570431078125)--(12.2360679774998,-0.657163890148917)--(12.0820529711032,-0.798224974386350)--(11.9158745096856,-0.935360111646119)--(11.7381884237148,-1.06069334521042)--(11.5496959589948,-1.16634952707103)--(11.3511410091699,-1.24494914244139)--(11.1433071799160,-1.29008367385391)--(10.9270146964069,-1.29674145173502)--(10.7031171662603,-1.26165686814013)--(10.4724982107387,-1.18356047014282)--(10.2360679774998,-1.06331351044005)--(9.99475954865942,-0.903917636852723)--(9.74952525834290,-0.710398105649023)--(9.50133293425722,-0.489566724215712)--(9.25116207811725,-0.249678177146732)--(9.00000000000000,0.000000000000000)--(8.74883792188275,0.249678177146732)--(8.49866706574278,0.489566724215713)--(8.25047474165710,0.710398105649023)--(8.00524045134058,0.903917636852722)--(7.76393202250021,1.06331351044005)--(7.52750178926129,1.18356047014282)--(7.29688283373971,1.26165686814013)--(7.07298530359314,1.29674145173502)--(6.85669282008401,1.29008367385391)--(6.64885899083011,1.24494914244139)--(6.45030404100524,1.16634952707103)--(6.26181157628525,1.06069334521042)--(6.08412549031435,0.935360111646120)--(5.91794702889684,0.798224974386350)--(5.76393202250021,0.657163890148917)--(5.62268829799194,0.519570431078126)--(5.49477327982455,0.391914399269006)--(5.38069179013592,0.279369617411445)--(5.28089405644699,0.185533741714553)--(5.19577393481939,0.112256994144896)--(5.12566735548548,0.0595897113507078)--(5.07085099708525,0.0258509997203338)--(5.03154119474209,0.00781305011399719)--(5.00789308628691,0.000988289981876847)--(5.00000000000000,0.000000000000000)--(5.00789308628691,-0.000988289981876847)--(5.03154119474209,-0.00781305011399716)--(5.07085099708525,-0.0258509997203337)--(5.12566735548548,-0.0595897113507079)--(5.19577393481939,-0.112256994144896)--(5.28089405644699,-0.185533741714553)--(5.38069179013592,-0.279369617411444)--(5.49477327982455,-0.391914399269005)--(5.62268829799194,-0.519570431078124)--(5.76393202250021,-0.657163890148917)--(5.91794702889684,-0.798224974386350)--(6.08412549031435,-0.935360111646119)--(6.26181157628524,-1.06069334521042)--(6.45030404100524,-1.16634952707103)--(6.64885899083011,-1.24494914244139)--(6.85669282008401,-1.29008367385391)--(7.07298530359314,-1.29674145173502)--(7.29688283373971,-1.26165686814013)--(7.52750178926129,-1.18356047014283)--(7.76393202250021,-1.06331351044005)--(8.00524045134058,-0.903917636852724)--(8.25047474165710,-0.710398105649022)--(8.49866706574278,-0.489566724215714)--(8.74883792188275,-0.249678177146731)--(9.00000000000000,0.000000000000000)--(9.25116207811725,0.249678177146730)--(9.50133293425722,0.489566724215712)--(9.74952525834290,0.710398105649021)--(9.99475954865942,0.903917636852723)--(10.2360679774998,1.06331351044005)--(10.4724982107387,1.18356047014283)--(10.7031171662603,1.26165686814013)--(10.9270146964069,1.29674145173502)--(11.1433071799160,1.29008367385391)--(11.3511410091699,1.24494914244139)--(11.5496959589948,1.16634952707103)--(11.7381884237148,1.06069334521042)--(11.9158745096856,0.935360111646120)--(12.0820529711032,0.798224974386349)--(12.2360679774998,0.657163890148917)--(12.3773117020081,0.519570431078124)--(12.5052267201755,0.391914399269005)--(12.6193082098641,0.279369617411444)--(12.7191059435530,0.185533741714554)--(12.8042260651806,0.112256994144896)--(12.8743326445145,0.0595897113507082)--(12.9291490029148,0.0258509997203337)--(12.9684588052579,0.00781305011399722)--(12.9921069137131,0.000988289981876833)--(13.0000000000000,0.000000000000000);
\end{tikzpicture}
\end{center}

The left-hand figure displays $(x,\xi):L \to \bR^2$, which is exact so long as the area of the central bigon is equal to the sum of the areas enclosed by the two outer lobes.  The right-hand figure is the wavefront hypersurface $\bF$ --- note the two singularities where the map $(x,f): L \to \bR^2$ fails to be an immersion.
\end{example*}

\subsection{Isotopies}
\label{intro:isotopies}
Part of the topological nature of exact Lagrangians is illustrated by the following Proposition:

\begin{prop*}
Suppose that $L_1$ and $L_2$ are two embedded exact Lagrangians with associated wavefront hypersurfaces $\bF_1$ and $\bF_2$.  If there is a smooth isotopy of $\bR^{n+1}$ carrying $\bF_1$ to $\bF_2$, then $L_1$ and $L_2$ are Hamiltonian isotopic.
\end{prop*}

The converse of this Proposition is false.  For example, the following two fragments of wavefronts are associated to a pair of Hamiltonian isotopic Lagrangians:

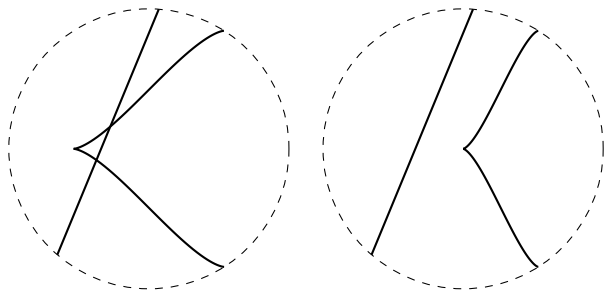
\begin{center}
\begin{tikzpicture}
\draw [thick] (1.00000000000000,-1.57079632679490)--
(0.998026728428272,-1.57071370865007)--
(0.992114701314478,-1.57013694551431)--
(0.982287250728689,-1.56857968535837)--
(0.968583161128631,-1.56557103917673)--
(0.951056516295154,-1.56066300718853)--
(0.929776485888251,-1.55343754406168)--
(0.904827052466019,-1.54351315173756)--
(0.876306680043864,-1.53055089588322)--
(0.844327925502015,-1.51425975108832)--
(0.809016994374947,-1.49440119050971)--
(0.770513242775789,-1.47079294758219)--
(0.728968627421412,-1.44331189047119)--
(0.684547105928689,-1.41189596393529)--
(0.637423989748690,-1.37654516797450)--
(0.587785252292473,-1.33732155783022)--
(0.535826794978997,-1.29434826533703)--
(0.481753674101715,-1.24780755706014)--
(0.425779291565073,-1.19793795984268)--
(0.368124552684678,-1.14503049909501)--
(0.309016994374947,-1.08942410915006)--
(0.248689887164855,-1.03150028806647)--
(0.187381314585725,-0.971677081176312)--
(0.125333233564304,-0.910402488260458)--
(0.0627905195293135,-0.848147398324422)--
(0.000000000000000,-0.785398163397448)--
(-0.0627905195293134,-0.722648928470474)--
(-0.125333233564304,-0.660393838534439)--
(-0.187381314585725,-0.599119245618585)--
(-0.248689887164855,-0.539296038728428)--
(-0.309016994374947,-0.481372217644840)--
(-0.368124552684678,-0.425765827699888)--
(-0.425779291565073,-0.372858366952215)--
(-0.481753674101715,-0.322988769734761)--
(-0.535826794978997,-0.276448061457862)--
(-0.587785252292473,-0.233474768964681)--
(-0.637423989748690,-0.194251158820399)--
(-0.684547105928689,-0.158900362859605)--
(-0.728968627421411,-0.127484436323707)--
(-0.770513242775789,-0.100003379212705)--
(-0.809016994374947,-0.0763951362851909)--
(-0.844327925502015,-0.0565365757065764)--
(-0.876306680043863,-0.0402454309116796)--
(-0.904827052466019,-0.0272831750573382)--
(-0.929776485888251,-0.0173587827332153)--
(-0.951056516295154,-0.0101333196063714)--
(-0.968583161128631,-0.00522528761816290)--
(-0.982287250728689,-0.00221664143652411)--
(-0.992114701314478,-0.000659381280582030)--
(-0.998026728428272,-0.0000826181448217722)--
(-1.00000000000000,0.000000000000000)--
(-0.998026728428272,0.0000826181448218902)--
(-0.992114701314478,0.000659381280582148)--
(-0.982287250728689,0.00221664143652424)--
(-0.968583161128631,0.00522528761816282)--
(-0.951056516295154,0.0101333196063714)--
(-0.929776485888251,0.0173587827332153)--
(-0.904827052466020,0.0272831750573384)--
(-0.876306680043864,0.0402454309116798)--
(-0.844327925502015,0.0565365757065766)--
(-0.809016994374947,0.0763951362851909)--
(-0.770513242775789,0.100003379212705)--
(-0.728968627421412,0.127484436323707)--
(-0.684547105928689,0.158900362859605)--
(-0.637423989748690,0.194251158820399)--
(-0.587785252292473,0.233474768964681)--
(-0.535826794978996,0.276448061457862)--
(-0.481753674101715,0.322988769734761)--
(-0.425779291565072,0.372858366952216)--
(-0.368124552684679,0.425765827699888)--
(-0.309016994374947,0.481372217644840)--
(-0.248689887164855,0.539296038728427)--
(-0.187381314585725,0.599119245618585)--
(-0.125333233564305,0.660393838534439)--
(-0.0627905195293132,0.722648928470474)--
(0.000000000000000,0.785398163397448)--
(0.0627905195293128,0.848147398324422)--
(0.125333233564304,0.910402488260458)--
(0.187381314585724,0.971677081176312)--
(0.248689887164855,1.03150028806647)--
(0.309016994374947,1.08942410915006)--
(0.368124552684678,1.14503049909501)--
(0.425779291565072,1.19793795984268)--
(0.481753674101715,1.24780755706014)--
(0.535826794978996,1.29434826533703)--
(0.587785252292473,1.33732155783022)--
(0.637423989748689,1.37654516797450)--
(0.684547105928689,1.41189596393529)--
(0.728968627421411,1.44331189047119)--
(0.770513242775789,1.47079294758219)--
(0.809016994374947,1.49440119050971)--
(0.844327925502015,1.51425975108832)--
(0.876306680043864,1.53055089588322)--
(0.904827052466020,1.54351315173756)--
(0.929776485888251,1.55343754406168)--
(0.951056516295154,1.56066300718853)--
(0.968583161128631,1.56557103917673)--
(0.982287250728689,1.56857968535837)--
(0.992114701314478,1.57013694551431)--
(0.998026728428272,1.57071370865007);
\draw [dashed] (0,0) circle (1.86209588911859cm);
\draw [thick] (0.131719452433182, 1.85743131397181) -- (-1.21714709935872,
-1.40923881538758);
\end{tikzpicture}
\quad
\begin{tikzpicture}
\draw [thick] (1.00000000000000,-1.57079632679490)--
(0.999013364214136,-1.57071370865007)--
(0.996057350657239,-1.57013694551431)--
(0.991143625364344,-1.56857968535837)--
(0.984291580564316,-1.56557103917673)--
(0.975528258147577,-1.56066300718853)--
(0.964888242944126,-1.55343754406168)--
(0.952413526233010,-1.54351315173756)--
(0.938153340021932,-1.53055089588322)--
(0.922163962751008,-1.51425975108832)--
(0.904508497187474,-1.49440119050971)--
(0.885256621387895,-1.47079294758219)--
(0.864484313710706,-1.44331189047119)--
(0.842273552964344,-1.41189596393529)--
(0.818711994874345,-1.37654516797450)--
(0.793892626146237,-1.33732155783022)--
(0.767913397489498,-1.29434826533703)--
(0.740876837050858,-1.24780755706014)--
(0.712889645782536,-1.19793795984268)--
(0.684062276342339,-1.14503049909501)--
(0.654508497187474,-1.08942410915006)--
(0.624344943582428,-1.03150028806647)--
(0.593690657292862,-0.971677081176312)--
(0.562666616782152,-0.910402488260458)--
(0.531395259764657,-0.848147398324422)--
(0.500000000000000,-0.785398163397448)--
(0.468604740235343,-0.722648928470474)--
(0.437333383217848,-0.660393838534439)--
(0.406309342707138,-0.599119245618585)--
(0.375655056417573,-0.539296038728428)--
(0.345491502812526,-0.481372217644840)--
(0.315937723657661,-0.425765827699888)--
(0.287110354217464,-0.372858366952215)--
(0.259123162949142,-0.322988769734761)--
(0.232086602510502,-0.276448061457862)--
(0.206107373853763,-0.233474768964681)--
(0.181288005125655,-0.194251158820399)--
(0.157726447035656,-0.158900362859605)--
(0.135515686289294,-0.127484436323707)--
(0.114743378612105,-0.100003379212705)--
(0.0954915028125263,-0.0763951362851909)--
(0.0778360372489926,-0.0565365757065764)--
(0.0618466599780683,-0.0402454309116796)--
(0.0475864737669903,-0.0272831750573382)--
(0.0351117570558743,-0.0173587827332153)--
(0.0244717418524232,-0.0101333196063714)--
(0.0157084194356845,-0.00522528761816290)--
(0.00885637463565569,-0.00221664143652411)--
(0.00394264934276112,-0.000659381280582030)--
(0.000986635785864221,-0.0000826181448217722)--
(0.000000000000000,0.000000000000000)--
(0.000986635785864221,0.0000826181448218902)--
(0.00394264934276106,0.000659381280582148)--
(0.00885637463565569,0.00221664143652424)--
(0.0157084194356845,0.00522528761816282)--
(0.0244717418524232,0.0101333196063714)--
(0.0351117570558744,0.0173587827332153)--
(0.0475864737669902,0.0272831750573384)--
(0.0618466599780682,0.0402454309116798)--
(0.0778360372489924,0.0565365757065766)--
(0.0954915028125263,0.0763951362851909)--
(0.114743378612105,0.100003379212705)--
(0.135515686289294,0.127484436323707)--
(0.157726447035656,0.158900362859605)--
(0.181288005125655,0.194251158820399)--
(0.206107373853763,0.233474768964681)--
(0.232086602510502,0.276448061457862)--
(0.259123162949142,0.322988769734761)--
(0.287110354217464,0.372858366952216)--
(0.315937723657661,0.425765827699888)--
(0.345491502812526,0.481372217644840)--
(0.375655056417572,0.539296038728427)--
(0.406309342707138,0.599119245618585)--
(0.437333383217848,0.660393838534439)--
(0.468604740235343,0.722648928470474)--
(0.500000000000000,0.785398163397448)--
(0.531395259764656,0.848147398324422)--
(0.562666616782152,0.910402488260458)--
(0.593690657292862,0.971677081176312)--
(0.624344943582427,1.03150028806647)--
(0.654508497187474,1.08942410915006)--
(0.684062276342339,1.14503049909501)--
(0.712889645782536,1.19793795984268)--
(0.740876837050857,1.24780755706014)--
(0.767913397489498,1.29434826533703)--
(0.793892626146237,1.33732155783022)--
(0.818711994874345,1.37654516797450)--
(0.842273552964344,1.41189596393529)--
(0.864484313710706,1.44331189047119)--
(0.885256621387895,1.47079294758219)--
(0.904508497187474,1.49440119050971)--
(0.922163962751008,1.51425975108832)--
(0.938153340021932,1.53055089588322)--
(0.952413526233010,1.54351315173756)--
(0.964888242944126,1.55343754406168)--
(0.975528258147577,1.56066300718853)--
(0.984291580564316,1.56557103917673)--
(0.991143625364344,1.56857968535837)--
(0.996057350657239,1.57013694551431)--
(0.999013364214136,1.57071370865007);
\draw [dashed] (0,0) circle (1.86209588911859cm);
\draw [thick] (0.131719452433182, 1.85743131397181) -- (-1.21714709935872,
-1.40923881538758);
\end{tikzpicture}
\end{center}

For small $n$, there is a finite list of such ``Reidemeister moves'' that generate the equivalence relation on exact Lagrangians given by Hamiltonian isotopy.  We warn that this is not so for $n \geq 7$.

\subsection{Compact and noncompact Lagrangians}
\label{intro:compact-noncompact}
The discussion of \S\ref{intro:eLaw} applies more generally with $\bR^{2n}$ replaced by the cotangent bundle of a manifold $M$ --- if the $x_i$ are local coordinates on $M$, then \eqref{intro:alpha} is independent of those coordinates.  The wavefront map of an exact Lagrangian $L \subset T^* M$ takes values in $M \times \bR$, and the Proposition of \S\ref{intro:isotopies} holds in this setting.  The Arnold conjecture predicts that the only compact exact Lagrangian in $T^* M$, up to Hamiltonian isotopy, is the zero section $M$.  Floer theory \cite{Ab-Nearby,Kr}, and microlocal sheaf theory \cite{Guillermou}, have been used to tightly circumscribe any potential counterexamples to this conjecture --- a strong result along these lines is that any compact exact Lagrangian in $T^* M$ must be homotopy equivalent to $M$.  The theory that we develop here is inspired by this work, but it could not be used to deduce any new results about compact $L$.  Our aim is different.

Noncompact exact Lagrangians are abundant.  In dimension 4, they are associated with a rich combinatorics of cluster algebras \cite{STWZ}.  In some complex varieties, they are associated with perverse sheaves, especially with tilting sheaves \cite{J-Perverse}, \cite{J-Tilting}.  These examples have been studied by us elsewhere using Nadler-Zaslow's \eqref{eq:one}.  Our results here remove the dependence of Floer theory, in the same spirit of some of Tamarkin's work giving sheaf-theoretic alternatives for the classic applications of Floer theory.

\subsection{Coefficients}
\label{intro:coefficients}
The Floer-theoretic constructions of \cite{NZ} give a functor \eqref{eq:one} from local systems of $\bk$-modules to constructible sheaves of $\bk$-modules, where $\bk$ is any commutative ring.  Our version is sufficiently soft that we may replace $\bk$ by a ring spectrum --- but this requires us to revisit the notion of a brane obstruction.  We discuss this briefly here, and more in \S\ref{intro:mumon}.

(We remark that our discussion only requires $\bk$ to be $\mathrm{E}_2$-commutative, in other words commutative enough that we may speak of ``$\bk$-linear'' stable $\infty$-categories.\footnote{Even this assumption can be relaxed: if $\bk'$ is an associative $\bk$-algebra, then \eqref{eq:one} induces a functor from left $\bk'$-module objects in $\Loc(L;\bk)$ to left $\bk'$-module objects in $\Sh_{\Lambda}(M;\bk)$.}  But in this paper we do not explore this in any significant way.)

\begin{quote}
{\bf Proposition/Definition.}  Let $\bk$ be an $\mathrm{E}_2$-commutative ring spectrum, and let $L \subset T^* M$ be an exact Lagrangian.  Then $L$ carries a canonical locally constant sheaf of $\bk$-linear categories $\Brane_L$, whose fiber is equivalent to $\Mod(\bk)$.  We say that the \emph{brane obstruction of $L$ vanishes} with coefficients in $\bk$ if this locally constant sheaf is constant, i.e. is equivalent to a constant sheaf.
\end{quote}

What we will construct canonically is a functor $\Gamma(L,\Brane_L) \to \Sh_{\Lambda}(M,\bk)$.  A trivialization of $\Brane_L$ gives an equivalence $\Gamma(L,\Brane_L) \cong \Loc(L;\bk)$.  In general, $\Gamma(L,\Brane_L)$ can be identified with the category of modules over a famous associative $\bk$-algebra spectrum --- the Thom spectrum of the map $L \to B \Pic(\bk)$ that classifies $\Brane_L$.

The sheaf of categories, and something more general, is an essentially standard construction in microlocal sheaf theory.  (In \cite{TZ} one of us has called it the ``Kashiwara-Schapira sheaf;'' in \cite{Guillermou} it is called the ``Kashiwara-Schapira stack.'')  The foundational text \cite{KS} of the theory is significantly older than the foundational text of $\infty$-category theory \cite{higher-topoi}, which we require to make sense of the Proposition above.  In \S\ref{sec:som} we give some details about how to adapt microlocal sheaf theory to treat sheaves of spectra.

\subsection{Singular support and front projections}
\label{intro:ssafp}
To have a good theory of noncompact exact Lagrangians \S\ref{intro:compact-noncompact}, we impose boundary conditions: we require that $L$ should be asymptotic to a Legendrian submanifold $\Lambda \subset T^{\infty} M$.  Here $T^{\infty} M$ is a cosphere bundle of infinite radius, the boundary of $T^* M$ with its natural contact structure.  We will denote the projection of a Legendrian $\Lambda \subset T^{\infty}M$ by $\Phi \subset M$, and often make the following hypothesis:
\begin{equation}
\label{eq:finite-to-one}
\text{$\Lambda \to \Phi$ is finite-to-one and over $\Phi^{sm}$ (the smooth locus of $\Phi$) it is one-to-one}.
\end{equation}
When \eqref{eq:finite-to-one} holds, $\Lambda$ induces a coorientation on the smooth parts of $\Phi$, and $\Lambda$ can be recovered from $\Phi$ as its conormal lift --- in that case we call $\Phi$ the ``front projection'' of $\Lambda$.  The wavefront projection of an exact Lagrangian (the situation of \S\ref{intro:eLaw}) can be treated as a special case --- see \S\ref{subsec:conic-bL}.

(Throughout the paper, we use ``front projection" for Legendrians and ``wavefront projection" for exact Lagrangians, though this is not at all standard or historical.)

Suppose $F$ is a sheaf on $M$ with singular support in $\Lambda$ --- or rather, in the conic Lagrangian subset of $T^* M$ obtained by taking the union of the zero section $\zeta_M$ with the cone over $\Lambda$.  The theory of singular supports was introduced and developed by Kashiwara and Schapira \cite{KS}.  One gets a sense of what it means to have singular support in $\Lambda$ by studying the structure of $F$ near the smooth points, multiple points, and singular points of the front $\Phi$. 

\subsubsection{Smooth part}
\label{intro:smooth-part}
The front projection breaks $M$ into chambers, the connected components of $M - \Phi$, on which $F$ is locally constant.  At a smooth point $x \in \Phi$ a small ball around $x$ will be broken into two contractible chambers, $R_1$ and $R_2$, so that $F$ is constant along $R_1$ and $R_2$.  One can use the restriction maps of $F$ to give a map $F(R_1) \to F(R_2)$ against the co-orientation of $\Phi$.
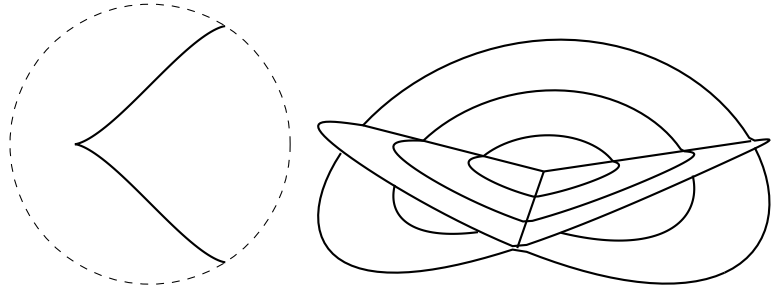
\begin{equation}
\label{eq:smooth-part}
\begin{tikzpicture}
\draw [dashed] (0,0) circle (1cm);
\draw [thick, rounded corners] (-1,0)--(-1/4,1/10)--(1/4,1/10)--(1,0);
\draw [thick, gray] (-1/2,1/20)--(-.47,-.09);
\draw [thick, gray] (1/2,1/20)--(.47,-.09);
\draw [thick, gray] (0,1/10)--(0,-.065);
\node at (0,-.5) {$R_1$};
\node at (0,.5) {$R_2$};
\node at (4,0) {$F(R_1) \to F(R_2)$};
\end{tikzpicture}
\end{equation}
The singular support condition implies that these restriction maps (which we often imagine as drawn on $M$, as in the diagram \eqref{eq:damp}) completely determine $F$ near a smooth point.  They vary in a locally constant fashion along the smooth part of $\Phi$.  

\subsubsection{Double and multiple points}
\label{intro:damp}
Near a double point of $\Phi$ (i.e., a point where $\Lambda \to \Phi$ is two-to-one, but still an immersion), we get a square of restriction maps
\begin{equation}
\label{eq:damp}
\begin{tikzpicture}
\draw [dashed] (0,0) circle (3cm);
\node at (0,0) {$\xymatrix{& F(R_{22}) \\ F(R_{12}) \ar[ur] & & F(R_{21}) \ar[ul] \\ & F(R_{11}) \ar[ur] \ar[ul]}$};
\draw [thick,  domain=-100:100, samples=40] 
 plot ({\x/38}, {1.5*sin(\x)} );
 \foreach \p in {1,2,3}
 \draw[thick, gray]  ({\p*20/38}, {1.5*sin(\p*20)})--({\p*20/38+3*(5/(\p+3))*(cos(\p*20)*1.5)/sqrt((cos(\p*20)*38*1.5)^2+1)}, {1.5*sin(\p*20)-3*(5/(\p+3))*1.5/sqrt((cos(\p*20)*38*1.5)^2+1)});
  \foreach \p in {-1,-2,-3}
 \draw[thick, gray]  ({\p*20/38}, {1.5*sin(\p*20)})--({\p*20/38+3*(5/(-\p+3))*(cos(\p*20)*1.5)/sqrt((cos(\p*20)*38*1.5)^2+1)}, {1.5*sin(\p*20)-3*(5/(-\p+3))*1.5/sqrt((cos(\p*20)*38*1.5)^2+1)});
  \foreach \p in {1,2,3}
 \draw[thick, gray]  ({\p*20/38}, {-1.5*sin(\p*20)})--({\p*20/38-3*(5/(\p+3))*(cos(\p*20)*1.5)/sqrt((cos(\p*20)*38*1.5)^2+1)}, {-1.5*sin(\p*20)-3*(5/(\p+3))*1.5/sqrt((cos(\p*20)*38*1.5)^2+1)});
  \foreach \p in {-1,-2,-3}
 \draw[thick, gray]  ({\p*20/38}, {-1.5*sin(\p*20)})--({\p*20/38-3*(5/(-\p+3))*(cos(\p*20)*1.5)/sqrt((cos(\p*20)*38*1.5)^2+1)}, {-1.5*sin(\p*20)-3*(5/(-\p+3))*1.5/sqrt((cos(\p*20)*38*1.5)^2+1)});
\draw [thick,  domain=-100:100, samples=40] 
 plot ({\x/38}, {-1.5*sin(\x)} );
 \end{tikzpicture}
\end{equation}
The sheaf structure of $F$ makes this a commutative square (even in the $\infty$-categorical sense: a functor from $\Delta^1 \times \Delta^1$ to $\Mod(\bk)$); the singular support condition implies that it is \emph{exact}, i.e. it is both a pushout and a pullback square in the stable $\infty$-category $\Mod(\bk)$.  Near a multiple point of degree $n$ (i.e. near a point where $n$ branches intersecting transversely, which is only possible when $\dim M\geq n$) the picture is similar: there, one has an exact $n$-cube, i.e. a commutative $n$-cube all of whose square faces are exact.

\subsubsection{Singular points of a front}
\label{intro:spoaf}
If $\Lambda \to \Phi \to M$ fails to be an immersion at $x \in \Lambda$, the image of $x$ in $M$ is called a singular point of $\Phi$.  Here are two examples, where $\Lambda$ has dimension $1$ and $2$:
\begin{center}
\begin{tikzpicture}
\draw [thick] (1.00000000000000,-1.57079632679490)--
(0.998026728428272,-1.57071370865007)--
(0.992114701314478,-1.57013694551431)--
(0.982287250728689,-1.56857968535837)--
(0.968583161128631,-1.56557103917673)--
(0.951056516295154,-1.56066300718853)--
(0.929776485888251,-1.55343754406168)--
(0.904827052466019,-1.54351315173756)--
(0.876306680043864,-1.53055089588322)--
(0.844327925502015,-1.51425975108832)--
(0.809016994374947,-1.49440119050971)--
(0.770513242775789,-1.47079294758219)--
(0.728968627421412,-1.44331189047119)--
(0.684547105928689,-1.41189596393529)--
(0.637423989748690,-1.37654516797450)--
(0.587785252292473,-1.33732155783022)--
(0.535826794978997,-1.29434826533703)--
(0.481753674101715,-1.24780755706014)--
(0.425779291565073,-1.19793795984268)--
(0.368124552684678,-1.14503049909501)--
(0.309016994374947,-1.08942410915006)--
(0.248689887164855,-1.03150028806647)--
(0.187381314585725,-0.971677081176312)--
(0.125333233564304,-0.910402488260458)--
(0.0627905195293135,-0.848147398324422)--
(0.000000000000000,-0.785398163397448)--
(-0.0627905195293134,-0.722648928470474)--
(-0.125333233564304,-0.660393838534439)--
(-0.187381314585725,-0.599119245618585)--
(-0.248689887164855,-0.539296038728428)--
(-0.309016994374947,-0.481372217644840)--
(-0.368124552684678,-0.425765827699888)--
(-0.425779291565073,-0.372858366952215)--
(-0.481753674101715,-0.322988769734761)--
(-0.535826794978997,-0.276448061457862)--
(-0.587785252292473,-0.233474768964681)--
(-0.637423989748690,-0.194251158820399)--
(-0.684547105928689,-0.158900362859605)--
(-0.728968627421411,-0.127484436323707)--
(-0.770513242775789,-0.100003379212705)--
(-0.809016994374947,-0.0763951362851909)--
(-0.844327925502015,-0.0565365757065764)--
(-0.876306680043863,-0.0402454309116796)--
(-0.904827052466019,-0.0272831750573382)--
(-0.929776485888251,-0.0173587827332153)--
(-0.951056516295154,-0.0101333196063714)--
(-0.968583161128631,-0.00522528761816290)--
(-0.982287250728689,-0.00221664143652411)--
(-0.992114701314478,-0.000659381280582030)--
(-0.998026728428272,-0.0000826181448217722)--
(-1.00000000000000,0.000000000000000)--
(-0.998026728428272,0.0000826181448218902)--
(-0.992114701314478,0.000659381280582148)--
(-0.982287250728689,0.00221664143652424)--
(-0.968583161128631,0.00522528761816282)--
(-0.951056516295154,0.0101333196063714)--
(-0.929776485888251,0.0173587827332153)--
(-0.904827052466020,0.0272831750573384)--
(-0.876306680043864,0.0402454309116798)--
(-0.844327925502015,0.0565365757065766)--
(-0.809016994374947,0.0763951362851909)--
(-0.770513242775789,0.100003379212705)--
(-0.728968627421412,0.127484436323707)--
(-0.684547105928689,0.158900362859605)--
(-0.637423989748690,0.194251158820399)--
(-0.587785252292473,0.233474768964681)--
(-0.535826794978996,0.276448061457862)--
(-0.481753674101715,0.322988769734761)--
(-0.425779291565072,0.372858366952216)--
(-0.368124552684679,0.425765827699888)--
(-0.309016994374947,0.481372217644840)--
(-0.248689887164855,0.539296038728427)--
(-0.187381314585725,0.599119245618585)--
(-0.125333233564305,0.660393838534439)--
(-0.0627905195293132,0.722648928470474)--
(0.000000000000000,0.785398163397448)--
(0.0627905195293128,0.848147398324422)--
(0.125333233564304,0.910402488260458)--
(0.187381314585724,0.971677081176312)--
(0.248689887164855,1.03150028806647)--
(0.309016994374947,1.08942410915006)--
(0.368124552684678,1.14503049909501)--
(0.425779291565072,1.19793795984268)--
(0.481753674101715,1.24780755706014)--
(0.535826794978996,1.29434826533703)--
(0.587785252292473,1.33732155783022)--
(0.637423989748689,1.37654516797450)--
(0.684547105928689,1.41189596393529)--
(0.728968627421411,1.44331189047119)--
(0.770513242775789,1.47079294758219)--
(0.809016994374947,1.49440119050971)--
(0.844327925502015,1.51425975108832)--
(0.876306680043864,1.53055089588322)--
(0.904827052466020,1.54351315173756)--
(0.929776485888251,1.55343754406168)--
(0.951056516295154,1.56066300718853)--
(0.968583161128631,1.56557103917673)--
(0.982287250728689,1.56857968535837)--
(0.992114701314478,1.57013694551431)--
(0.998026728428272,1.57071370865007);
\draw [dashed] (0,0) circle (1.86209588911859cm);
\end{tikzpicture}
\tdplotsetmaincoords{70}{180}
\begin{tikzpicture}[tdplot_main_coords][scale=0.75]
\draw [thick]
(-2.40343084664080, 1.79541643231187, 2.77942286340599) --
(-2.51142335543284, 1.64096091659322, 2.77596375714903) --
(-2.60950442360117, 1.48002927782051, 2.76561714165016) --
(-2.69728696997873, 1.31325663964930, 2.74847485420102) --
(-2.77442455695897, 1.14130117749129, 2.72468905095858) --
(-2.84061275772692, 0.964841520996624, 2.69447085639632) --
(-2.89559035769264, 0.784574075812750, 2.65808848934842) --
(-2.93914038538594, 0.601210275189442, 2.61586488228046) --
(-2.97109096874365, 0.415473772276811, 2.56817481491793) --
(-2.99131601341007, 0.228097584196913, 2.51544158767492) --
(-2.99973570037386, 0.0398211991600644, 2.45813326440961) --
(-2.99631680097715, -0.148612341957294, 2.39675851785640) --
(-2.98107280805393, -0.336459378055468, 2.33186211461063) --
(-2.95406388267993, -0.522978562702124, 2.26402007974178) --
(-2.91539661674441, -0.707433789887975, 2.19383458395414) --
(-2.86522361228062, -0.889097099100892, 2.12192859867704) --
(-2.80374287921531, -1.06725154825349, 2.04894036652629) --
(-2.73119705391401, -1.24119404312598, 2.02448226378243) --
(-2.64787244160617, -1.41023811215787, 2.09768758778462) --
(-2.55409788646915, -1.57371661563758, 2.17002582984196) --
(-2.45024347383058, -1.73098437859811, 2.24085491059684) --
(-2.33671906961055, -1.88142073702784, 2.30954614611421) --
(-2.21397270276817, -2.02443198734791, 2.37548982812472) --
(-2.08248879713588, -2.15945372948891, 2.43810063583842) --
(-1.94278625961993, -2.28595309432018, 2.49682283129360) --
(-1.79541643231187, -2.40343084664080, 2.55113519212621) --
(-1.64096091659322, -2.51142335543284, 2.60055563797654) --
(-1.48002927782051, -2.60950442360117, 2.64464550946855) --
(-1.31325663964930, -2.69728696997873, 2.68301346178159) --
(-1.14130117749129, -2.77442455695897, 2.71531893825471) --
(-0.964841520996625, -2.84061275772692, 2.74127519319170) --
(-0.784574075812751, -2.89559035769264, 2.76065183703620) --
(-0.601210275189442, -2.93914038538594, 2.77327688132571) --
(-0.415473772276808, -2.97109096874365, 2.77903826527334) --
(-0.228097584196912, -2.99131601341007, 2.77788485042728) --
(-0.0398211991600645, -2.99973570037386, 2.76982687457938) --
(0.148612341957293, -2.99631680097715, 2.75493586089371) --
(0.336459378055468, -2.98107280805393, 2.73334398306194) --
(0.522978562702123, -2.95406388267993, 2.70524289212027) --
(0.707433789887975, -2.91539661674441, 2.67088201534129) --
(0.889097099100892, -2.86522361228062, 2.63056634229983) --
(1.06725154825348, -2.80374287921531, 2.58465371776391) --
(1.24119404312597, -2.73119705391402, 2.53355166543923) --
(1.41023811215787, -2.64787244160617, 2.47771377075993) --
(1.57371661563758, -2.55409788646915, 2.41763565483219) --
(1.73098437859811, -2.45024347383058, 2.35385057526612) --
(1.88142073702784, -2.33671906961055, 2.28692469294354) --
(2.02443198734791, -2.21397270276817, 2.21745204673390) --
(2.15945372948891, -2.08248879713588, 2.14604928076318) --
(2.28595309432018, -1.94278625961993, 2.07335017103703) --
(2.40343084664080, -1.79541643231187, 2.00000000000000) --
(2.51142335543284, -1.64096091659322, 2.07335017103703) --
(2.60950442360117, -1.48002927782051, 2.14604928076318) --
(2.69728696997873, -1.31325663964930, 2.21745204673390) --
(2.77442455695897, -1.14130117749129, 2.28692469294354) --
(2.84061275772692, -0.964841520996625, 2.35385057526612) --
(2.89559035769264, -0.784574075812749, 2.41763565483219) --
(2.93914038538594, -0.601210275189445, 2.47771377075993) --
(2.97109096874365, -0.415473772276811, 2.53355166543923) --
(2.99131601341007, -0.228097584196912, 2.58465371776391) --
(2.99973570037386, -0.0398211991600647, 2.63056634229983) --
(2.99631680097715, 0.148612341957292, 2.67088201534129) --
(2.98107280805393, 0.336459378055467, 2.70524289212027) --
(2.95406388267993, 0.522978562702123, 2.73334398306194) --
(2.91539661674441, 0.707433789887975, 2.75493586089371) --
(2.86522361228062, 0.889097099100891, 2.76982687457938) --
(2.80374287921531, 1.06725154825349, 2.77788485042728) --
(2.73119705391401, 1.24119404312598, 2.77903826527334) --
(2.64787244160617, 1.41023811215787, 2.77327688132571) --
(2.55409788646915, 1.57371661563758, 2.76065183703620) --
(2.45024347383058, 1.73098437859811, 2.74127519319170) --
(2.33671906961055, 1.88142073702784, 2.71531893825471) --
(2.21397270276817, 2.02443198734791, 2.68301346178159) --
(2.08248879713588, 2.15945372948891, 2.64464550946855) --
(1.94278625961993, 2.28595309432018, 2.60055563797654) --
(1.79541643231187, 2.40343084664080, 2.55113519212621) --
(1.64096091659322, 2.51142335543284, 2.49682283129360) --
(1.48002927782051, 2.60950442360117, 2.43810063583842) --
(1.31325663964930, 2.69728696997873, 2.37548982812472) --
(1.14130117749129, 2.77442455695897, 2.30954614611421) --
(0.964841520996625, 2.84061275772692, 2.24085491059684) --
(0.784574075812749, 2.89559035769264, 2.17002582984196) --
(0.601210275189445, 2.93914038538594, 2.09768758778462) --
(0.415473772276811, 2.97109096874365, 2.02448226378243) --
(0.228097584196915, 2.99131601341007, 2.04894036652629) --
(0.0398211991600649, 2.99973570037386, 2.12192859867704) --
(-0.148612341957292, 2.99631680097715, 2.19383458395414) --
(-0.336459378055467, 2.98107280805393, 2.26402007974178) --
(-0.522978562702126, 2.95406388267993, 2.33186211461063) --
(-0.707433789887975, 2.91539661674441, 2.39675851785640) --
(-0.889097099100891, 2.86522361228062, 2.45813326440961) --
(-1.06725154825348, 2.80374287921531, 2.51544158767492) --
(-1.24119404312597, 2.73119705391402, 2.56817481491793) --
(-1.41023811215787, 2.64787244160617, 2.61586488228046) --
(-1.57371661563758, 2.55409788646916, 2.65808848934842) --
(-1.73098437859811, 2.45024347383058, 2.69447085639632) --
(-1.88142073702784, 2.33671906961055, 2.72468905095858) --
(-2.02443198734791, 2.21397270276817, 2.74847485420102) --
(-2.15945372948891, 2.08248879713588, 2.76561714165016) --
(-2.28595309432018, 1.94278625961993, 2.77596375714903) --
(-2.40343084664080, 1.79541643231187, 2.77942286340599);

\draw[thick]
(-1.60228723109387, 1.19694428820791, 2.42426406871193) --
(-1.67428223695523, 1.09397394439548, 2.42238116975734) --
(-1.73966961573412, 0.986686185213674, 2.41674918563801) --
(-1.79819131331915, 0.875504426432869, 2.40741810624367) --
(-1.84961637130598, 0.760867451660858, 2.39447075489563) --
(-1.89374183848461, 0.643227680664416, 2.37802205320103) --
(-1.93039357179509, 0.523049383875166, 2.35821800100059) --
(-1.95942692359063, 0.400806850126295, 2.33523438046363) --
(-1.98072731249577, 0.276982514851207, 2.30927519583316) --
(-1.99421067560672, 0.152065056131276, 2.28057086266969) --
(-1.99982380024924, 0.0265474661067096, 2.24937616266647) --
(-1.99754453398477, -0.0990748946381960, 2.21596798218913) --
(-1.98738187203595, -0.224306252036978, 2.18064285461268) --
(-1.96937592178662, -0.348652375134749, 2.14371432827028) --
(-1.94359774449627, -0.471622526591983, 2.10551018337607) --
(-1.91014907485374, -0.592731399400595, 2.06636952262474) --
(-1.86916191947687, -0.711501032168991, 2.02663976129204) --
(-1.82079803594268, -0.827462695417318, 2.01332645644782) --
(-1.76524829440411, -0.940158741438578, 2.05317438761681) --
(-1.70273192431277, -1.04914441042505, 2.09255033915691) --
(-1.63349564922038, -1.15398958573207, 2.13110480733465) --
(-1.55781271307370, -1.25428049135190, 2.16849557996107) --
(-1.47598180184545, -1.34962132489861, 2.20439077389131) --
(-1.38832586475725, -1.43963581965927, 2.23847178084289) --
(-1.29519083974662, -1.52396872954679, 2.27043609538537) --
(-1.19694428820791, -1.60228723109387, 2.30000000000000) --
(-1.09397394439548, -1.67428223695523, 2.32690108337648) --
(-0.986686185213674, -1.73966961573412, 2.35090056959432) --
(-0.875504426432867, -1.79819131331915, 2.37178543751485) --
(-0.760867451660858, -1.84961637130598, 2.38937031157188) --
(-0.643227680664416, -1.89374183848461, 2.40349910717837) --
(-0.523049383875167, -1.93039357179509, 2.41404641614430) --
(-0.400806850126295, -1.95942692359063, 2.42091861980860) --
(-0.276982514851206, -1.98072731249577, 2.42405472000503) --
(-0.152065056131275, -1.99421067560672, 2.42342688048623) --
(-0.0265474661067097, -1.99982380024924, 2.41904067400023) --
(0.0990748946381950, -1.99754453398477, 2.41093503282629) --
(0.224306252036978, -1.98738187203595, 2.39918190320933) --
(0.348652375134749, -1.96937592178662, 2.38388560675986) --
(0.471622526591983, -1.94359774449627, 2.36518191448805) --
(0.592731399400594, -1.91014907485374, 2.34323684169061) --
(0.711501032168990, -1.86916191947687, 2.31824517438728) --
(0.827462695417316, -1.82079803594268, 2.29042874038628) --
(0.940158741438578, -1.76524829440411, 2.26003444032505) --
(1.04914441042505, -1.70273192431277, 2.22733205616266) --
(1.15398958573207, -1.63349564922038, 2.19261185657606) --
(1.25428049135190, -1.55781271307370, 2.15618202051476) --
(1.34962132489860, -1.47598180184545, 2.11836590178264) --
(1.43963581965927, -1.38832586475725, 2.07949915892673) --
(1.52396872954679, -1.29519083974662, 2.03992677590813) --
(1.60228723109387, -1.19694428820791, 2.00000000000000) --
(1.67428223695523, -1.09397394439548, 2.03992677590813) --
(1.73966961573412, -0.986686185213674, 2.07949915892673) --
(1.79819131331915, -0.875504426432867, 2.11836590178264) --
(1.84961637130598, -0.760867451660858, 2.15618202051476) --
(1.89374183848461, -0.643227680664416, 2.19261185657606) --
(1.93039357179509, -0.523049383875166, 2.22733205616266) --
(1.95942692359063, -0.400806850126297, 2.26003444032505) --
(1.98072731249577, -0.276982514851208, 2.29042874038628) --
(1.99421067560672, -0.152065056131275, 2.31824517438728) --
(1.99982380024924, -0.0265474661067098, 2.34323684169061) --
(1.99754453398477, 0.0990748946381949, 2.36518191448805) --
(1.98738187203595, 0.224306252036978, 2.38388560675986) --
(1.96937592178662, 0.348652375134749, 2.39918190320933) --
(1.94359774449627, 0.471622526591983, 2.41093503282629) --
(1.91014907485374, 0.592731399400594, 2.41904067400023) --
(1.86916191947687, 0.711501032168991, 2.42342688048623) --
(1.82079803594268, 0.827462695417317, 2.42405472000503) --
(1.76524829440411, 0.940158741438580, 2.42091861980860) --
(1.70273192431277, 1.04914441042505, 2.41404641614430) --
(1.63349564922038, 1.15398958573207, 2.40349910717837) --
(1.55781271307370, 1.25428049135190, 2.38937031157188) --
(1.47598180184545, 1.34962132489860, 2.37178543751485) --
(1.38832586475725, 1.43963581965927, 2.35090056959432) --
(1.29519083974662, 1.52396872954679, 2.32690108337648) --
(1.19694428820791, 1.60228723109387, 2.30000000000000) --
(1.09397394439548, 1.67428223695522, 2.27043609538537) --
(0.986686185213674, 1.73966961573412, 2.23847178084289) --
(0.875504426432869, 1.79819131331915, 2.20439077389131) --
(0.760867451660858, 1.84961637130598, 2.16849557996107) --
(0.643227680664417, 1.89374183848461, 2.13110480733465) --
(0.523049383875166, 1.93039357179509, 2.09255033915691) --
(0.400806850126297, 1.95942692359063, 2.05317438761681) --
(0.276982514851208, 1.98072731249577, 2.01332645644782) --
(0.152065056131277, 1.99421067560672, 2.02663976129204) --
(0.0265474661067099, 1.99982380024924, 2.06636952262474) --
(-0.0990748946381947, 1.99754453398477, 2.10551018337607) --
(-0.224306252036978, 1.98738187203595, 2.14371432827028) --
(-0.348652375134750, 1.96937592178662, 2.18064285461268) --
(-0.471622526591983, 1.94359774449627, 2.21596798218913) --
(-0.592731399400594, 1.91014907485374, 2.24937616266647) --
(-0.711501032168989, 1.86916191947687, 2.28057086266969) --
(-0.827462695417316, 1.82079803594268, 2.30927519583316) --
(-0.940158741438579, 1.76524829440411, 2.33523438046363) --
(-1.04914441042505, 1.70273192431277, 2.35821800100059) --
(-1.15398958573207, 1.63349564922038, 2.37802205320103) --
(-1.25428049135190, 1.55781271307370, 2.39447075489563) --
(-1.34962132489861, 1.47598180184545, 2.40741810624367) --
(-1.43963581965927, 1.38832586475725, 2.41674918563801) --
(-1.52396872954679, 1.29519083974662, 2.42238116975734) --
(-1.60228723109387, 1.19694428820791, 2.42426406871193);

\draw [thick]
(-0.801143615546934, 0.598472144103957, 2.15000000000000) --
(-0.837141118477613, 0.546986972197740, 2.14933429469046) --
(-0.869834807867058, 0.493343092606837, 2.14734308760930) --
(-0.899095656659576, 0.437752213216434, 2.14404405285154) --
(-0.924808185652991, 0.380433725830429, 2.13946647288324) --
(-0.946870919242306, 0.321613840332208, 2.13365097862826) --
(-0.965196785897545, 0.261524691937583, 2.12664918882530) --
(-0.979713461795315, 0.200403425063147, 2.11852325185635) --
(-0.990363656247883, 0.138491257425604, 2.10934529411321) --
(-0.997105337803358, 0.0760325280656378, 2.09919677979855) --
(-0.999911900124620, 0.0132737330533548, 2.08816778784387) --
(-0.998772266992385, -0.0495374473190980, 2.07635621236256) --
(-0.993690936017976, -0.112153126018489, 2.06386689373476) --
(-0.984687960893310, -0.174326187567375, 2.05081068803679) --
(-0.971798872248136, -0.235811263295992, 2.03730348307473) --
(-0.955074537426872, -0.296365699700297, 2.02346516975603) --
(-0.934580959738436, -0.355750516084495, 2.00941857792940) --
(-0.910399017971338, -0.413731347708659, 2.00471161386172) --
(-0.882624147202056, -0.470079370719289, 2.01879998503465) --
(-0.851365962156384, -0.524572205212527, 2.03272148620948) --
(-0.816747824610192, -0.576994792866035, 2.04635254915624) --
(-0.778906356536851, -0.627140245675948, 2.05957218359522) --
(-0.737990900922724, -0.674810662449303, 2.07226305111526) --
(-0.694162932378627, -0.719817909829636, 2.08431250667782) --
(-0.647595419873309, -0.761984364773394, 2.09561359846230) --
(-0.598472144103957, -0.801143615546934, 2.10606601717798) --
(-0.546986972197741, -0.837141118477613, 2.11557698641637) --
(-0.493343092606837, -0.869834807867058, 2.12406208614118) --
(-0.437752213216434, -0.899095656659576, 2.13144600200658) --
(-0.380433725830429, -0.924808185652992, 2.13766319385260) --
(-0.321613840332208, -0.946870919242306, 2.14265847744427) --
(-0.261524691937584, -0.965196785897545, 2.14638751429081) --
(-0.200403425063147, -0.979713461795315, 2.14881720519717) --
(-0.138491257425603, -0.990363656247883, 2.14992598405486) --
(-0.0760325280656374, -0.997105337803358, 2.14970400926424) --
(-0.0132737330533548, -0.999911900124620, 2.14815325108927) --
(0.0495374473190975, -0.998772266992385, 2.14528747416929) --
(0.112153126018489, -0.993690936017976, 2.14113211534313) --
(0.174326187567374, -0.984687960893310, 2.13572405786990) --
(0.235811263295992, -0.971798872248136, 2.12911130405059) --
(0.296365699700297, -0.955074537426872, 2.12135254915624) --
(0.355750516084495, -0.934580959738436, 2.11251666044457) --
(0.413731347708658, -0.910399017971338, 2.10268206588930) --
(0.470079370719289, -0.882624147202056, 2.09193605804795) --
(0.524572205212527, -0.851365962156384, 2.08037401924685) --
(0.576994792866035, -0.816747824610192, 2.06809857496093) --
(0.627140245675948, -0.778906356536851, 2.05521868290270) --
(0.674810662449302, -0.737990900922724, 2.04184866590588) --
(0.719817909829636, -0.694162932378627, 2.02810719718786) --
(0.761984364773394, -0.647595419873309, 2.01411624699778) --
(0.801143615546934, -0.598472144103957, 2.00000000000000) --
(0.837141118477613, -0.546986972197741, 2.01411624699778) --
(0.869834807867058, -0.493343092606837, 2.02810719718786) --
(0.899095656659576, -0.437752213216434, 2.04184866590588) --
(0.924808185652991, -0.380433725830429, 2.05521868290270) --
(0.946870919242306, -0.321613840332208, 2.06809857496093) --
(0.965196785897546, -0.261524691937583, 2.08037401924685) --
(0.979713461795314, -0.200403425063148, 2.09193605804795) --
(0.990363656247883, -0.138491257425604, 2.10268206588930) --
(0.997105337803358, -0.0760325280656375, 2.11251666044457) --
(0.999911900124620, -0.0132737330533549, 2.12135254915624) --
(0.998772266992385, 0.0495374473190974, 2.12911130405059) --
(0.993690936017976, 0.112153126018489, 2.13572405786990) --
(0.984687960893310, 0.174326187567374, 2.14113211534313) --
(0.971798872248136, 0.235811263295992, 2.14528747416929) --
(0.955074537426872, 0.296365699700297, 2.14815325108927) --
(0.934580959738435, 0.355750516084496, 2.14970400926424) --
(0.910399017971338, 0.413731347708659, 2.14992598405486) --
(0.882624147202056, 0.470079370719290, 2.14881720519717) --
(0.851365962156385, 0.524572205212527, 2.14638751429081) --
(0.816747824610192, 0.576994792866035, 2.14265847744427) --
(0.778906356536851, 0.627140245675948, 2.13766319385260) --
(0.737990900922724, 0.674810662449302, 2.13144600200658) --
(0.694162932378627, 0.719817909829636, 2.12406208614118) --
(0.647595419873309, 0.761984364773394, 2.11557698641637) --
(0.598472144103957, 0.801143615546934, 2.10606601717798) --
(0.546986972197741, 0.837141118477612, 2.09561359846230) --
(0.493343092606837, 0.869834807867058, 2.08431250667782) --
(0.437752213216435, 0.899095656659575, 2.07226305111526) --
(0.380433725830429, 0.924808185652991, 2.05957218359522) --
(0.321613840332208, 0.946870919242306, 2.04635254915624) --
(0.261524691937583, 0.965196785897546, 2.03272148620948) --
(0.200403425063148, 0.979713461795314, 2.01879998503465) --
(0.138491257425604, 0.990363656247883, 2.00471161386172) --
(0.0760325280656384, 0.997105337803358, 2.00941857792940) --
(0.0132737330533550, 0.999911900124620, 2.02346516975603) --
(-0.0495374473190974, 0.998772266992385, 2.03730348307473) --
(-0.112153126018489, 0.993690936017976, 2.05081068803679) --
(-0.174326187567375, 0.984687960893310, 2.06386689373476) --
(-0.235811263295992, 0.971798872248136, 2.07635621236256) --
(-0.296365699700297, 0.955074537426872, 2.08816778784387) --
(-0.355750516084495, 0.934580959738436, 2.09919677979855) --
(-0.413731347708658, 0.910399017971339, 2.10934529411321) --
(-0.470079370719290, 0.882624147202056, 2.11852325185635) --
(-0.524572205212526, 0.851365962156385, 2.12664918882530) --
(-0.576994792866035, 0.816747824610192, 2.13365097862826) --
(-0.627140245675948, 0.778906356536851, 2.13946647288324) --
(-0.674810662449303, 0.737990900922723, 2.14404405285154) --
(-0.719817909829636, 0.694162932378627, 2.14734308760930) --
(-0.761984364773394, 0.647595419873309, 2.14933429469046) --
(-0.801143615546934, 0.598472144103957, 2.15000000000000);

\draw [thick]
(-2.40343084664080, 1.79541643231187, 1.22057713659401) --
(-2.51142335543284, 1.64096091659322, 1.22403624285097) --
(-2.60950442360117, 1.48002927782051, 1.23438285834984) --
(-2.69728696997873, 1.31325663964930, 1.25152514579898) --
(-2.77442455695897, 1.14130117749129, 1.27531094904142) --
(-2.84061275772692, 0.964841520996624, 1.30552914360368) --
(-2.89559035769264, 0.784574075812750, 1.34191151065158)--
(-2.93914038538594, 0.601210275189442, 1.38413511771954) --
(-2.97109096874365, 0.415473772276811, 1.43182518508207) --
(-2.99131601341007, 0.228097584196913, 1.48455841232508) --
(-2.99973570037386, 0.0398211991600644, 1.54186673559039) --
(-2.99631680097715, -0.148612341957294, 1.60324148214360) --
(-2.98107280805393, -0.336459378055468, 1.66813788538938) --
(-2.95406388267993, -0.522978562702124, 1.73597992025822) --
(-2.91539661674441, -0.707433789887975, 1.80616541604586) --
(-2.86522361228062, -0.889097099100892, 1.87807140132296)--
(-2.80374287921531, -1.06725154825349, 1.95105963347371);

\draw [thick]
(2.69728696997873, -1.31325663964930, 1.78254795326610) --
(2.77442455695897, -1.14130117749129, 1.71307530705646) --
(2.84061275772692, -0.964841520996625, 1.64614942473388) --
(2.89559035769264, -0.784574075812749, 1.58236434516781)--
(2.93914038538594, -0.601210275189445, 1.52228622924007) --
(2.97109096874365, -0.415473772276811, 1.46644833456077) --
(2.99131601341007, -0.228097584196912, 1.41534628223609) --
(2.99973570037386, -0.0398211991600647, 1.36943365770017) --
(2.99631680097715, 0.148612341957292, 1.32911798465871) --
(2.98107280805393, 0.336459378055467, 1.29475710787973) --
(2.95406388267993, 0.522978562702123, 1.26665601693806) --
(2.91539661674441, 0.707433789887975, 1.24506413910629) --
(2.86522361228062, 0.889097099100891, 1.23017312542062) --
(2.80374287921531, 1.06725154825349, 1.22211514957272) --
(2.73119705391401, 1.24119404312598, 1.22096173472666) --
(2.64787244160617, 1.41023811215787, 1.22672311867429) --
(2.55409788646915, 1.57371661563758, 1.23934816296380) --
(2.45024347383058, 1.73098437859811, 1.25872480680830) --
(2.33671906961055, 1.88142073702784, 1.28468106174529) --
(2.21397270276817, 2.02443198734791, 1.31698653821841) --
(2.08248879713588, 2.15945372948891, 1.35535449053145) --
(1.94278625961993, 2.28595309432018, 1.39944436202346) --
(1.79541643231187, 2.40343084664080, 1.44886480787378) --
(1.64096091659322, 2.51142335543284, 1.50317716870640) --
(1.48002927782051, 2.60950442360117, 1.56189936416158) --
(1.31325663964930, 2.69728696997873, 1.62451017187528) --
(1.14130117749129, 2.77442455695897, 1.69045385388579) --
(0.964841520996625, 2.84061275772692, 1.75914508940316) --
(0.784574075812749, 2.89559035769264, 1.82997417015804) --
(0.601210275189445, 2.93914038538594, 1.90231241221538) --
(0.415473772276811, 2.97109096874365, 1.97551773621757) --
(0.228097584196915, 2.99131601341007, 1.95105963347371) --
(0.0398211991600649, 2.99973570037386, 1.87807140132296) --
(-0.148612341957292, 2.99631680097715, 1.80616541604586) --
(-0.336459378055467, 2.98107280805393, 1.73597992025822) --
(-0.522978562702126, 2.95406388267993, 1.66813788538937) --
(-0.707433789887975, 2.91539661674441, 1.60324148214360) --
(-0.889097099100891, 2.86522361228062, 1.54186673559039) --
(-1.06725154825348, 2.80374287921531, 1.48455841232508) --
(-1.24119404312597, 2.73119705391402, 1.43182518508207) --
(-1.41023811215787, 2.64787244160617, 1.38413511771954) --
(-1.57371661563758, 2.55409788646916, 1.34191151065158) --
(-1.73098437859811, 2.45024347383058, 1.30552914360368) --
(-1.88142073702784, 2.33671906961055, 1.27531094904142) --
(-2.02443198734791, 2.21397270276817, 1.25152514579898) --
(-2.15945372948891, 2.08248879713588, 1.23438285834984) --
(-2.28595309432018, 1.94278625961993, 1.22403624285097) --
(-2.40343084664080, 1.79541643231187, 1.22057713659401);

\draw[thick]
(1.98072731249577, -0.276982514851208, 1.70957125961372) --
(1.98845017417795, -0.214629692288108, 1.69532481506196) --
(1.99421067560672, -0.152065056131275, 1.68175482561272) --
(1.99800313186263, -0.0893503501230789, 1.66889142003613) --
(1.99982380024924, -0.0265474661067098, 1.65676315830939) --
(1.99967088398666, 0.0362816170535921, 1.64539696820652) --
(1.99754453398477, 0.0990748946381949, 1.63481808551195) --
(1.99344684869431, 0.161770397263247, 1.62504999799081) --
(1.98738187203595, 0.224306252036978, 1.61611439324014) --
(1.97935558940941, 0.286620743620770, 1.60803111053709) --
(1.96937592178662, 0.348652375134749, 1.60081809679067) --
(1.95745271789466, 0.410339928847783, 1.59449136669522) --
(1.94359774449627, 0.471622526591983, 1.58906496717371) --
(1.92782467477746, 0.532439689842114, 1.58455094619008) --
(1.91014907485374, 0.592731399400594, 1.58095932599977) --
(1.89058838840824, 0.652438154629195, 1.57829808089779) --
(1.86916191947687, 0.711501032168991, 1.57657311951377) --
(1.84589081339764, 0.769861744090596, 1.57578827169335) --
(1.82079803594268, 0.827462695417317, 1.57594527999497) --
(1.79390835065385, 0.884247040964444, 1.57704379582094) --
(1.76524829440411, 0.940158741438580, 1.57908138019140) --
(1.73484615120880, 0.995142618741654, 1.58205350915951) --
(1.70273192431277, 1.04914441042505, 1.58595358385570) --
(1.66893730658075, 1.10211082324010, 1.59077294513882) --
(1.63349564922038, 1.15398958573207, 1.59650089282163) --
(1.59644192886855, 1.20472949982577, 1.60312470942783) --
(1.55781271307370, 1.25428049135190, 1.61062968842812) --
(1.51764612420805, 1.30259365946418, 1.61899916689229) --
(1.47598180184545, 1.34962132489860, 1.62821456248515) --
(1.43286086364186, 1.39531707702711, 1.63825541472393) --
(1.38832586475725, 1.43963581965927, 1.64909943040568) --
(1.34242075585874, 1.48253381554676, 1.66072253310376) --
(1.29519083974662, 1.52396872954679, 1.67309891662352) --
(1.24668272664595, 1.56389967040172, 1.68620110229852) --
(1.19694428820791, 1.60228723109387, 1.70000000000000) --
(1.14602461026625, 1.63909352773540, 1.71446497272430) --
(1.09397394439548, 1.67428223695522, 1.72956390461463) --
(1.04084365831867, 1.70781863174572, 1.74526327226631) --
(0.986686185213674, 1.73966961573412, 1.76152821915711) --
(0.931554971967940, 1.76980375584464, 1.77832263303747) --
(0.875504426432869, 1.79819131331915, 1.79560922610869);

\draw[thick]
(-0.224306252036978, 1.98738187203595, 1.85628567172972) --
(-0.286620743620769, 1.97935558940941, 1.83764116995614) --
(-0.348652375134750, 1.96937592178662, 1.81935714538732) --
(-0.410339928847780, 1.95745271789466, 1.80147419313019)--
(-0.471622526591983, 1.94359774449627, 1.78403201781087) --
(-0.532439689842114, 1.92782467477746, 1.76706934542042) --
(-0.592731399400594, 1.91014907485374, 1.75062383733353) --
(-0.652438154629195, 1.89058838840824, 1.73473200669076) --
(-0.711501032168989, 1.86916191947687, 1.71942913733031) --
(-0.769861744090597, 1.84589081339764, 1.70474920544905) --
(-0.827462695417316, 1.82079803594268, 1.69072480416684) --
(-0.884247040964446, 1.79390835065385, 1.67738707116171) --
(-0.940158741438579, 1.76524829440411, 1.66476561953637) --
(-0.995142618741654, 1.73484615120880, 1.65288847206974) --
(-1.04914441042505, 1.70273192431277, 1.64178199899941) --
(-1.10211082324011, 1.66893730658075, 1.63147085947305) --
(-1.15398958573207, 1.63349564922038, 1.62197794679897) --
(-1.20472949982577, 1.59644192886855, 1.61332433761726) --
(-1.25428049135190, 1.55781271307370, 1.60552924510437) --
(-1.30259365946418, 1.51764612420805, 1.59860997631504) --
(-1.34962132489861, 1.47598180184545, 1.59258189375633) --
(-1.39531707702711, 1.43286086364187, 1.58745838127899) --
(-1.43963581965927, 1.38832586475725, 1.58325081436199) --
(-1.48253381554676, 1.34242075585874, 1.57996853485611) --
(-1.52396872954679, 1.29519083974662, 1.57761883024266) --
(-1.56389967040172, 1.24668272664595, 1.57620691745349) --
(-1.60228723109387, 1.19694428820791, 1.57573593128807) --
(-1.63909352773540, 1.14602461026625, 1.57620691745349) --
(-1.67428223695522, 1.09397394439548, 1.57761883024266) --
(-1.70781863174572, 1.04084365831867, 1.57996853485611) --
(-1.73966961573412, 0.986686185213676, 1.58325081436199) --
(-1.76980375584464, 0.931554971967942, 1.58745838127899) --
(-1.79819131331915, 0.875504426432871, 1.59258189375633) --
(-1.82480427306568, 0.818589863729837, 1.59860997631504) --
(-1.84961637130598, 0.760867451660858, 1.60552924510437) --
(-1.87260312149470, 0.702394155277730, 1.61332433761726) --
(-1.89374183848461, 0.643227680664417, 1.62197794679897) --
(-1.91301166091416, 0.583426417988107, 1.63147085947305) --
(-1.93039357179509, 0.523049383875164, 1.64178199899941) --
(-1.94587041727990, 0.462156163168861, 1.65288847206974) --
(-1.95942692359063, 0.400806850126293, 1.66476561953637) --
(-1.97104971209226, 0.339061989112655, 1.67738707116171) --
(-1.98072731249577, 0.276982514851206, 1.69072480416684) --
(-1.98845017417795, 0.214629692288110, 1.70474920544905) --
(-1.99421067560672, 0.152065056131275, 1.71942913733031) --
(-1.99800313186263, 0.0893503501230791, 1.73473200669076) --
(-1.99982380024924, 0.0265474661067101, 1.75062383733353) --
(-1.99967088398666, -0.0362816170535919, 1.76706934542042) --
(-1.99754453398477, -0.0990748946381946, 1.78403201781087) --
(-1.99344684869431, -0.161770397263250, 1.80147419313019) --
(-1.98738187203595, -0.224306252036980, 1.81935714538732) --
(-1.97935558940941, -0.286620743620771, 1.83764116995614);

\draw [thick]
(-2.75659166407450, -1.18372395327414, 2.00000000000000)--(0,0,2);

\draw [thick]
(0.353160817433702, 2.97914038558601, 2.00000000000000)--(0,0,2);

\draw [thick]
(2.40343084664080, -1.79541643231187, 2.00000000)--(0,0,2);
\end{tikzpicture}
\end{center}
The figure on the left is the ``fold'' or ``cusp'' singularity.  The figure on the right is the critical front of the $\mathrm{D}_4^{-}$-bifurcation of fronts \cite[Ch. 22]{Arnold-singularities}, which is also of interest as the wavefront projection of a holomorphic Lagrangian (the graph of $\pm \sqrt{z} dz$).

Near a singular point, $\Phi$ breaks $M$ into chambers and lower-dimensional strata which can fit together in a complicated way.  Still it is sometimes possible to describe the structure of $F$ near a singular point concretely.  A description for the examples pictured above is fairly simple: at a cusp the two restriction maps must compose (in one direction) to the identity \cite[Ex. 5.3.4]{KS}, and at a critical $\mathrm{D}_4^{-}$ front we get \cite[\S 4.1.2]{TZ2} a triple of exact commutative squares glued at their source and sink:
\begin{equation}
\label{eq:fold-and}
\begin{tikzpicture}
\node at (3,0) {$p\circ i = 1_A;$};
\node at (0,0) {$\xymatrix{A \\ & B \ar[lu]_(.42){p} \\ A \ar[ru]_(.55){i}}$};
\draw [thick] (1.00000000000000,-1.57079632679490)--
(0.998026728428272,-1.57071370865007)--
(0.992114701314478,-1.57013694551431)--
(0.982287250728689,-1.56857968535837)--
(0.968583161128631,-1.56557103917673)--
(0.951056516295154,-1.56066300718853)--
(0.929776485888251,-1.55343754406168)--
(0.904827052466019,-1.54351315173756)--
(0.876306680043864,-1.53055089588322)--
(0.844327925502015,-1.51425975108832)--
(0.809016994374947,-1.49440119050971)--
(0.770513242775789,-1.47079294758219)--
(0.728968627421412,-1.44331189047119)--
(0.684547105928689,-1.41189596393529)--
(0.637423989748690,-1.37654516797450)--
(0.587785252292473,-1.33732155783022)--
(0.535826794978997,-1.29434826533703)--
(0.481753674101715,-1.24780755706014)--
(0.425779291565073,-1.19793795984268)--
(0.368124552684678,-1.14503049909501)--
(0.309016994374947,-1.08942410915006)--
(0.248689887164855,-1.03150028806647)--
(0.187381314585725,-0.971677081176312)--
(0.125333233564304,-0.910402488260458)--
(0.0627905195293135,-0.848147398324422)--
(0.000000000000000,-0.785398163397448)--
(-0.0627905195293134,-0.722648928470474)--
(-0.125333233564304,-0.660393838534439)--
(-0.187381314585725,-0.599119245618585)--
(-0.248689887164855,-0.539296038728428)--
(-0.309016994374947,-0.481372217644840)--
(-0.368124552684678,-0.425765827699888)--
(-0.425779291565073,-0.372858366952215)--
(-0.481753674101715,-0.322988769734761)--
(-0.535826794978997,-0.276448061457862)--
(-0.587785252292473,-0.233474768964681)--
(-0.637423989748690,-0.194251158820399)--
(-0.684547105928689,-0.158900362859605)--
(-0.728968627421411,-0.127484436323707)--
(-0.770513242775789,-0.100003379212705)--
(-0.809016994374947,-0.0763951362851909)--
(-0.844327925502015,-0.0565365757065764)--
(-0.876306680043863,-0.0402454309116796)--
(-0.904827052466019,-0.0272831750573382)--
(-0.929776485888251,-0.0173587827332153)--
(-0.951056516295154,-0.0101333196063714)--
(-0.968583161128631,-0.00522528761816290)--
(-0.982287250728689,-0.00221664143652411)--
(-0.992114701314478,-0.000659381280582030)--
(-0.998026728428272,-0.0000826181448217722)--
(-1.00000000000000,0.000000000000000)--
(-0.998026728428272,0.0000826181448218902)--
(-0.992114701314478,0.000659381280582148)--
(-0.982287250728689,0.00221664143652424)--
(-0.968583161128631,0.00522528761816282)--
(-0.951056516295154,0.0101333196063714)--
(-0.929776485888251,0.0173587827332153)--
(-0.904827052466020,0.0272831750573384)--
(-0.876306680043864,0.0402454309116798)--
(-0.844327925502015,0.0565365757065766)--
(-0.809016994374947,0.0763951362851909)--
(-0.770513242775789,0.100003379212705)--
(-0.728968627421412,0.127484436323707)--
(-0.684547105928689,0.158900362859605)--
(-0.637423989748690,0.194251158820399)--
(-0.587785252292473,0.233474768964681)--
(-0.535826794978996,0.276448061457862)--
(-0.481753674101715,0.322988769734761)--
(-0.425779291565072,0.372858366952216)--
(-0.368124552684679,0.425765827699888)--
(-0.309016994374947,0.481372217644840)--
(-0.248689887164855,0.539296038728427)--
(-0.187381314585725,0.599119245618585)--
(-0.125333233564305,0.660393838534439)--
(-0.0627905195293132,0.722648928470474)--
(0.000000000000000,0.785398163397448)--
(0.0627905195293128,0.848147398324422)--
(0.125333233564304,0.910402488260458)--
(0.187381314585724,0.971677081176312)--
(0.248689887164855,1.03150028806647)--
(0.309016994374947,1.08942410915006)--
(0.368124552684678,1.14503049909501)--
(0.425779291565072,1.19793795984268)--
(0.481753674101715,1.24780755706014)--
(0.535826794978996,1.29434826533703)--
(0.587785252292473,1.33732155783022)--
(0.637423989748689,1.37654516797450)--
(0.684547105928689,1.41189596393529)--
(0.728968627421411,1.44331189047119)--
(0.770513242775789,1.47079294758219)--
(0.809016994374947,1.49440119050971)--
(0.844327925502015,1.51425975108832)--
(0.876306680043864,1.53055089588322)--
(0.904827052466020,1.54351315173756)--
(0.929776485888251,1.55343754406168)--
(0.951056516295154,1.56066300718853)--
(0.968583161128631,1.56557103917673)--
(0.982287250728689,1.56857968535837)--
(0.992114701314478,1.57013694551431)--
(0.998026728428272,1.57071370865007);
\draw [dashed] (0,0) circle (1.86209588911859cm);
\node at (7,0) {$\xymatrix{
& C \\
B_1 \ar[ur] & B_2 \ar[u] & B_3 \ar[ul] \\
& A \ar[u] \ar[ul] \ar[ur]
}$};
\end{tikzpicture}
\end{equation}

\subsection{Microlocal monodromy and brane obstruction}
\label{intro:mumon}
Let $F$ be a sheaf with singular support in $\Lambda$, where $\Lambda$ and $\Phi$ are as in \S\ref{intro:ssafp}.  Let us introduce a notation: we write $\Lambda^{\mathit{sm}} \subset \Lambda^{\mathit{imm}} \subset \Lambda$ for the set where $\Lambda \to M$ is an embedding resp. an immersion.  At each point $P \in \Lambda^\mathit{sm}$, we define $\mu_P(F)$ to be the cone on the map \eqref{eq:smooth-part}.  We refer to $\mu_P(F)$ as a ``microlocal stalk'' of $F$ at $P$.  The spectra $\mu_P$ are the stalks of a local system on $\Lambda^{\mathit{imm}}$: they vary in a locally constant fashion over the smooth part of $\Phi$, and the exactness condition of \S\ref{intro:damp} allows for a definition of $\mu_P(F)$ at the preimage of a multiple point of $\Phi$ as well.  For example, on one of the two strands of the diagram \eqref{eq:damp}, we have the microlocal stalks
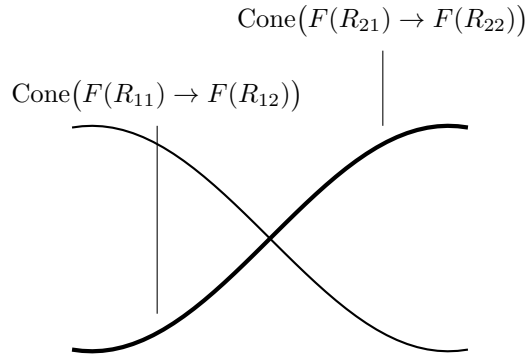
\begin{center}
\begin{tikzpicture}
\draw [ultra thick,  domain=-100:100, samples=40] 
 plot ({\x/38}, {1.5*sin(\x)} );
 \draw [thick,  domain=-100:100, samples=40] 
 plot ({\x/38}, {-1.5*sin(\x)} );

\draw (-1.5,1.5)--(-1.5,-1);
\draw (1.5,1.5)--(1.5,2.5);
\node at (-1.5,1.9) {$\mathrm{Cone}\big(F(R_{11}) \to F(R_{12})\big)$};
\node at (1.5,2.9) {$\mathrm{Cone}\big(F(R_{21}) \to F(R_{22})\big)$};
\end{tikzpicture}
\end{center}
which are canonically identified by the exactness of the square \eqref{eq:damp}.  The cones on the other pair of parallel arrows are identified in a similar way.

Now we discuss the nature of microlocal monodromy near and \emph{through} a singular point of a front.  For example, if $P$ and $Q$ are on opposite sides of a fold, and $F$ is the sheaf pictured in \eqref{eq:fold-and} then $\mu_P = \mathrm{Cone}(i)$ and $\mu_Q = \mathrm{Cone}(p)$ are related but not isomorphic: the homotopy between $p \circ i$ and $1_A$ gives an isomorphism between $\mathrm{Cone}(i)$ and $\Sigma^{-1} \mathrm{Cone}(p)$.  As in a local system we think of this isomorphism as being induced by a path from $P$ to $Q$ passing through the fold, but it means that a closed loop that passes through several folds does not exactly induce an automorphism of the microlocal stalk at the base point.  Instead, it induces a map
\begin{equation}
\label{eq:bI-twist}
\mu_P \cong \bI \otimes_{\bk} \mu_P
\end{equation}
where $\bI$ is an invertible $\bk$-module spectrum depending only on the based loop (but not on the sheaf $F$).

The assignment from based loops to invertible $\bk$-modules is similar to the data of a $\GL_1(\bk)$-valued cocycle $\eta$ on the fundamental group of $\Lambda$, and the data \eqref{eq:bI-twist} is similar to the data of a module over the $\eta$-twisted group algebra.  This resemblance is made precise in the theory of algebra structures on Thom spectra \cite{ABGHR}, by keeping track of higher homotopies.

The critical $\mathrm{D}_4^{-}$ front, pictured on the right in \S\ref{intro:spoaf}, is one place such higher homotopies intervene.  There, we have $\Lambda \cong \bR^2$ and $\Lambda^{\mathit{imm}} \cong \bR^2 -  \{(0,0)\}$.  A closed loop around the singular point (with an appropriate base point) induces the composite
\begin{equation}
\label{eq:ABiC}
B_1/A \stackrel\sim\to C/B_2 \stackrel\sim\leftarrow B_3/A \stackrel\sim\to C/B_1 \stackrel\sim\leftarrow B_2/A \stackrel\sim\to C/B_3 \stackrel\sim\leftarrow B_1/A
\end{equation}
where the spectra are as in \eqref{eq:fold-and}, and we write $X/Y$ as a shorthand for $\mathrm{Cone}(Y \to X)$.  The map \eqref{eq:ABiC} is always homotopic to $-1$. 

(Let us prove this. After taking the cone of the natural map from the constant sheaf with stalk $A$ to the given sheaf, we may assume $A = 0$. Then the data are the maps $B_i \to C$, $i = 1,2,3$, that induce isomorphisms $C \cong B_i \oplus B_j$ when $i \neq j$. Again without loss of generality, set $C = B_1 \oplus B_2$. Then $B_3 \to B_1 \oplus B_2$ is isomorphic to the inclusion of the graph of an isomorphism $f:B_1 \cong B_2$, and the sequence of maps in \eqref{eq:ABiC} is 
\[
B_1 \xrightarrow{\mathrm{id}} B_1 \xleftarrow{\mathrm{id}} B_1 \xrightarrow{f} B_2 \xleftarrow{\mathrm{id}} B_2 \xrightarrow{i_2} (B_1 \oplus B_2)/\Gamma_f \xleftarrow{i_1}B_1
\]
The first three maps compose to $f:B_1 \to B_2$ and the last three maps compose to $-f^{-1}:B_2 \to B_1$.)


Another formulation of the ``cocycle'' $\eta$ is that it is the data of a locally constant sheaf of $\infty$-categories over $\Lambda$, the sheaf of Brane structures \S\ref{intro:coefficients}.  Its global sections are in this sense $\eta$-twisted locals system --- the $\infty$-category of them is equivalent to the $\infty$-category of the twisted group algebra.  Then microlocal monodromy is a functor $\Sh_{\Lambda}(M)$ to $\Gamma(\Lambda,\Brane_{\Lambda})$, that we denote by $\mumon$.  In this paper, we will not try to make this path-theoretic description of $\mumon$ precise.  Instead, we construct $\Brane_{\Lambda}$ by \v Cech theory, and prove that it is locally constant using Kashiwara and Schapira's theory of contact transformations.

\subsection{The Nadler-Zaslow functor}
\label{intro:NZ}
An exact Lagrangian $L$ in $T^* M$ determines a conic Lagrangian in $T^*(M \times \bR)$, or equivalently a Legendrian in $(T^* M) \times \bR$, which we denote by $\bL$ (cf. (\ref{eq:cone-bL})). Here we are thinking of $(T^* M) \times \bR$ as living in the cosphere bundle at infinity of $T^*(M\times\bR)$, so then the Legendrian boundary of $\bL$ is canonically identified with $L$. The category of sheaves with singular support in $\bL$ can be described as in \S\ref{intro:ssafp}, with $\Lambda$ replaced by $L$ and $\Phi$ replaced by the wavefront projection of $L$ (which we sometimes denote by $\bF$).  Then $L$ carries a locally constant sheaf of categories $\Brane_L$, and \S\ref{intro:mumon} supplies a functor $\Sh_{\bL}(M \times \bR) \to \Gamma(L,\Brane_L)$.  Guillermou has proved (in \cite{Guillermou}, with some superficial differences in language) that when $L$ is compact and $L \to T^* M$ is an embedding, this functor restricts to an equivalence on the full subcategory $\Sh_{\bL}^0(M \times \bR) \subset \Sh_{\bL}(M \times \bR)$ spanned by sheaves that vanish on $M \times (C,\infty)$ for $C \gg 0$.  

We adapt this proof to the case where $L$ is noncompact (and ``lower exact'' \S\ref{subsec:exact-lags}), and to treat sheaves of spectra.  Then our version of \eqref{eq:one} is given by the composite
\begin{equation}
\label{eq:two}
\Gamma(L,\Brane_L) \xleftarrow{\mumon} \Sh_{\bL}^0(M \times \bR) \xrightarrow{\proj_{1,*}} \Sh_{\Lambda}(M)
\end{equation}

\subsection{The $J$-homomorphism}
There is a universal locally constant sheaf of categories with fiber $\Mod(\bk)$, living above a space denoted $B\Pic(\bk)$.  When $\bk = \bS$ is the sphere spectrum $B\Pic(\bS)$ is an infinite loop space with 
\[
\pi_1(B\Pic(\bS)) = \bZ \qquad \pi_2(B\Pic(\bS)) = \{\pm 1\} \qquad \pi_{i}(B\Pic(\bS)) = \pi_{i-2}(\bS) \quad \text{for $i \geq 3$}
\]
Like any locally constant sheaf, the sheaf of brane structures is the pullback of this universal one along a classifying map $L \to B\Pic(\bS)$.  In subsequent papers \cite{J-BO, J-Jhom}, by studying a bundle over $L$ whose fiber is a stabilization of the space of contact transformations, the first named author shows that this map factors as
\begin{equation}
\label{intro:factor}
L \to \LagGr(\infty) \to B\Pic(\bS)
\end{equation}
where the first map is the stable Gauss map\footnote{e.g. \cite[\S 2]{AbouzaidKragh}, though in our case there is an involution of $\LagGr(\infty)$ that one needs to take.} and the second map (whose domain is $\mathrm{U}/\mathrm{O}$, one of the Bott spaces whose based loop space is naturally homotopy equivalent to $\bZ \times B \mathrm{O}$) is the delooping of the real $J$-homomorphism $\bZ \times B\mathrm{O} \to \Pic(\bS)$.

Since $\bS$ is initial among ring spectra, there is a canonical map $B\Pic(\bS) \to B\Pic(R)$ for any commutative ring spectrum $R$, and we define an $R$-linear brane structure to be a trivialization of the composite map $L \to B\Pic(R)$.  For some special values of $R$, the theory of genera and orientations \cite{ABGHR2} shows that the composite $\LagGr(\infty) \to B\Pic(\bS) \to B\Pic(R)$ factors in a canonical fashion through a quotient of the domain that has only finitely many homotopy groups.  In these cases the vanishing of the Brane obstruction is implied by the vanishing of finitely many characteristic classes (see \cite[\S\!\S  7-8]{ABG}, or \cite{ABS} for an older perspective for items (2) and (3)):
\begin{enumerate}
\item For example, if $R$ is discrete, then $J:\bZ \times B\mathrm{O} \to \Pic(R)$ factors through $\bZ \times B(\mathrm{O}/\mathrm{SO}) \cong \bZ \times B(\bZ/2)$.  Thus the Brane obstruction can be trivialized by giving a null-homotopy of a map $L \to B\bZ \times B^2(\bZ/2)$ ---  this is the usual problem of gradings and relative pin structures.

\item If $R = \mathbf{KU}$, then $J:\bZ \times B\mathrm{O} \to \Pic(R)$ factors through a loop space we can denote by $\bZ/2 \ltimes B(\mathrm{O}/\mathrm{Spin}^c)$, whose homotopy groups are $\bZ/2, \bZ/2, 0,$ and $\bZ$ in degrees 0--3 and vanish otherwise.  (There is actually a splitting $\bZ/2 \ltimes B(\mathrm{O}/\mathrm{Spin}^c) \cong \bZ/2 \times B(\mathrm{O}/\mathrm{Spin}^c)$ preserving the loop space structure, but it is not canonical.)

\item If $R = \mathbf{KO}$, then $J:\bZ \times B\mathrm{O} \to \Pic(R)$ factors through a loop space we can denote by $\bZ/8 \ltimes B(\mathrm{O}/\mathrm{Spin})$, whose homotopy groups are $\bZ/8,\bZ/2,$ and $\bZ/2$ in degrees 0--2 and vanish otherwise.
\item If $R = \mathbf{tmf}$, $\bZ \times B\mathrm{O} \to \Pic(R)$ factors through $\bZ \times B(\mathrm{O}/\mathrm{String})$, whose homotopy groups are $\bZ, \bZ/2, \bZ/2,0,$ and $\bZ$ in degrees 0--4 and vanish otherwise.
\end{enumerate}
In the nondiscrete cases to identify $\Gamma(L,\Brane_L)$ with a category of untwisted local systems one might require less of the Maslov class (e.g. for $\mathbf{KU}$ it need only be even) but there are new classes that must be trivialized --- in $H^4(L,\bZ)$, for $\mathbf{KU}$, in $H^3(L,\mathbf{Z}/2)$ for $\mathbf{KO}$, and in $H^3(L,\bZ/2)$ and $H^5(L,\bZ)$ for $\mathbf{tmf}$.

\subsection{More general microlocal categories}

We briefly indicate some work in progress in this section.
Lurie's preprint \cite{Lurie-circle} on circle actions and algebraic $K$-theory has a brief discussion explaining the topological obstructions to defining an $\bS$-linear Fukaya category of a general symplectic manifold.  In particular, he suggests that an $\bS$-linear stable $\infty$-category can be associated to any symplectic $2n$-manifold for which the composite
\begin{equation}
\label{eq:lurie-fukaya}
M \to B\mathrm{U}(n) \to B\mathrm{U} \xrightarrow{\mathrm{Bott}} B^2(\bZ \times B\mathrm{U}) \xrightarrow{J_{\bC}} B^2 \Pic(\bS)
\end{equation}
is null-homotopic.  The compatibility between the real and complex $J$-homomorphisms, and between the real and complex Bott periodicities, shows that when \eqref{eq:lurie-fukaya} is null-homotopic, each Lagrangian in $M$ carries a locally constant sheaf of categories with fiber $\Mod(\bS)$, which we may as well go on denoting $\Brane_L$.  When $M$ is a Weinstein manifold, the data \eqref{eq:lurie-fukaya} determines a sheaf of categories over a skeleton of $M$, which is locally isomorphic to a category of sheaves --- following Tamarkin, we may call it the ``microlocal category'' of $M$.  The techniques of this paper can be extended to define a functor from $\Gamma(L,\Brane_L)$ to this microlocal category. After the first version of our work appeared in 2017, Nadler-Shende \cite{Nadler-Shende} defined such a sheaf of categories on any Weinstein manifold in 2020, given the null-homotopy of an obstruction $M\rightarrow B^2\Pic(\bS)$. Using the result of \cite{J-Jhom}, one can identify the obstruction with (\ref{eq:lurie-fukaya}).

\subsection{Precedents}

Microlocal sheaf theory was created by Kashiwara and Schapira \cite{KS}, and \eqref{eq:two} is directly inspired by the work of Nadler and Zaslow \cite{NZ}.  Our main result generalizes \cite{Guillermou} on compact Lagrangians and our proof uses some of the same techniques.  The most basic of these techniques, probing the symplectic geometry of $T^* M$ by studying sheaves on $M \times \bR$, was pioneered by Tamarkin \cite{Tamarkin1} and has many other recent applications.  Here we indicate some influences and precedents that are less direct.

The algebraic-topological work required to assign gradings to Floer chain groups was first explained by Seidel \cite{Seidel-graded}.  The appearance of spectra in this paper is part of a tradition, maybe starting with Cohen-Jones-Segal \cite{CJS}, of pursuing a ``Floer homotopy type'' underlying Floer homology.  The role in this story of locally constant sheaves of categories, whose fiber is the category of spectra, is anticipated in \cite{Doug}, and appear explicitly as Thom spectra in the Floer-theoretic works \cite{Kr,Kr2} and the non-Floer-theoretic work \cite{AbouzaidKragh}.  Lurie's preprint \cite[\S 1.3]{Lurie-circle} defines a topological invariant of an almost symplectic manifold, which he expects is the topological obstruction to defining a Fukaya category over the sphere spectrum --- the fact that it vanishes for a cotangent bundle has inspired us.

\subsection{Questions}

\subsubsection{Singular Lagrangians}

Suppose $L$ is a singular Lagrangian (for simplicity, in $\bR^{2n}$) whose Lagrangian singularity type is locally constant along a smooth locus $L_0 \subset L$.  Then one has a locally constant sheaf of categories along $L_0$ --- the Kashiwara-Schapira sheaf $\MSh_L\vert_{L_0}$.  Does this sheaf of categories have any familiar description in algebraic topology, along the lines of the $J$-homomorphism in the case of a smooth $L$?

\subsubsection{Immersed Lagrangians}
If $L \to T^* M$ is an exact Lagrangian immersion, we can still make sense of the diagram \eqref{eq:two}, but the left-hand map is no longer an equivalence.  It does have a right adjoint, and one could in this way compose $\proj_{1,*}$ with this adjoint to obtain a functor from twisted local systems on $L$ to $\Sh(M)$.  More intriguingly, one could regard \eqref{eq:two} as a correspondence or multiple-valued functor --- in the case where $L$ and $M$ are circles, this correspondence is studied in \cite[\S 6]{STZ}, it is shown to be closely related to the HOMFLY homology of the Legendrian lift of $L \subset (T^* S^1) \times \bR \subset S^3$.  

Neither recipe matches the standard Floer-theoretic treatments of immersed Lagrangians in any obvious way.  What's going on?

\subsubsection{Holomorphic Lagrangians}
Holomorphic Lagrangians are never lower exact, so one cannot apply \eqref{eq:two} to them directly.
Still, let us indicate an interesting feature of the $J$-homomorphism in the holomorphic symplectic setting.  Suppose (as in \cite{J-Perverse}) that $L$ is a holomorphic Lagrangian in the cotangent bundle of a complex manifold $M$.  Then the Gauss map $L \to \mathrm{U}/\mathrm{O}$ factors through $\mathrm{Sp}/\mathrm{U}$.  
After \cite{J-Perverse} it is tempting to call a sheaf of spectra on $M$ that arises from a holomorphic Lagrangian ``perverse,'' although such ``perverse sheaves of spectra'' cannot live in the heart of any $t$-structure on $\Sh(M,\bS)$.

It is interesting to speculate if there is a deeper implication.  The target of the holomorphic Gauss map is a further delooping of $\mathrm{U}/\mathrm{O}$, and we can compose it with a further delooping of the $J$-homomorphism:
\[
L_{\bC} \to \mathrm{Sp}/\mathrm{U} \to B^2 \Pic(\bS)
\]
This suggests that $L_{\bC}$ carries a sheaf of $2$-categories --- is there anything to that?  Can our recipe be adapted to produce some kind of $2$-categorical sheaf-like object on $M$, e.g. a perverse schober?  The ``categorified nature'' of holomorphic symplectic geometry compared to real symplectic geometry is a well-known phenomenon in quantum field theory, see especially \cite{KapustinRozansky}. Since the completion of the first version of this paper in 2017, there has been some developments in this line, e.g. see \cite{DoanRezchikov} which uses holomorphic Floer theory. 

Note this line of speculation can be continued up Bott's tower --- e.g. as the quadruple delooping of $\bZ \times B \mathrm{O}$ is the stable quaternionic Grassmannian, each almost quaternionic manifold carries a locally constant sheaf of $4$-categories.

\subsubsection{Lagrangian cobordisms}
\newcommand{\Lag}{\mathrm{Lag}}
Nadler and Tanaka \cite{NadlerTanaka} consider $\infty$-categories of exact Lagrangians and cobordisms between them.  These $\infty$-categories are stable (for nontrivial reasons), and in fact are linear over a symmetric monoidal ``coefficient'' category called $\Lag_{\mathit{pt}}(\mathit{pt})$.  \cite{NadlerTanaka} leaves open the problem of identifying this coefficient category.
It is defined as a colimit of unstable categories
\[
\Lag_{\mathit{pt}}(\mathit{pt}) := \varinjlim_n \Lag_{\bR^n}^{\diamond 0}(T^* \bR^n)
\]
whose objects (resp. morphisms) are exact Lagrangian submanifolds of $T^* \bR^n$ (resp. $T^* \bR^n \times T^* \bR$).  The morphisms are subject to a condition in the $T^* \bR$-factor, called ``noncharacteristic'' or ``$\bR^n$-avoiding'', that makes them irreversible.

As part of the definition, all of these Lagrangians (objects, cobordisms, and higher morphisms) are equipped with a trivialization of the composite of the Gauss map with $\LagGr \to S^1 \times B^2(\bZ/2) \cong B\Pic(\bZ)$ --- but not with a local system.  Still, we recognize the map to be trivialized as $\Brane_L$ with $\bZ$-coefficients, thus (if we stick to lower exact Lagrangians) our results associate
\begin{enumerate}
\item functors $\Loc(L;\bZ) \to \Sh(\bR^n;\bZ)$ to objects of $\Lag_{\bR^n}^{\diamond 0}(T^* \bR^n)$
\item functors $\Loc(C;\bZ) \to \Sh(\bR^n \times \bR;\bZ)$ to morphisms of $\Lag_{\bR^n}^{\diamond 0}(T^* \bR^n)$.  (The noncharacteristicness condition forces the image of this functor to lie in a certain localization of $\Sh(\bR^n \times \bR)$.)
\end{enumerate}
and so on for higher morphisms.  What is the right way to organize this structure?

\subsection{Acknowledgments}
We have benefited from answers, advice, and conversations on this topic from Mohammed Abouzaid, Elden Elmento, Marc Hoyois,  Jacob Lurie, Akhil Mathew, Haynes Miller, David Nadler, Dmitri Pavlov, Dima Tamarkin, and Dylan Wilson. We thank Wenyuan Li and Marco Volpe for useful comments on a previous version of the paper. We thank Peter Haine for teaching us many aspects about $\infty$-topoi. We are grateful to the anonymous referee for many useful comments and suggestions. 
XJ was supported by NSF-DMS-1854232. DT was supported by NSF-DMS-1510444 and a Sloan Research Fellowship.

\section{Sheaves of spectra}
\label{sec:som}

In this section we review the microlocal theory of sheaves on manifolds, noting what requires care in an $\infty$-categorical setting.  Even when working over a discrete ring $\bk$, we depart somewhat from \cite{KS}, in that we do not impose boundedness conditions.  In $\infty$-categorical jargon all of our $\infty$-categories of sheaves are ``presentable.''

It is an old observation of Neeman that working systematically with unbounded categories can simplify certain arguments in triangulated categories, but for many years it was not possible to take advantage of this observation in the kind of sheaf theory we discuss here --- most strikingly, we believe that in 1990 when \cite{KS} was being written it was an open problem whether the proper base-change theorem holds for unbounded complexes, even for a pair of maps between finite-dimensional manifolds.  This was settled by Lurie  in the affirmative in \cite[\S 7.3.1]{higher-topoi} (announced earlier in \cite{on-infty-topoi}).

Our interest is in Lagrangian submanifolds in cotangent bundles in $T^* M$, which are related to constructible sheaves on $M$ and  (as we begin to discuss in \S\ref{sec:wavefront-notation}) $M \times \bR$.  In \cite{KS} it is shown that systematic study of non-constructible sheaves can simplify the study of constructible sheaves, in particular Kashiwara and Schapira do not require the use of stratified Morse theory.  But there are crucial tools in the approach that are difficult to import to the presentable setting.  The problem is that specialization to the normal cone, and functors derived from it, such as $\mu\mathrm{hom}$, commute with neither infinite colimits nor infinite limits.  So in our microlocal analysis of sheaves of spectra, we restrict to sheaves that are constructible on a Whitney stratification and use stratified Morse theory.

Nevertheless, the microlocal theory of nonconstructible sheaves may be important in symplectic geometry, as it is the basis of Tamarkin's microlocal category of a compact symplectic manifold \cite{Tamarkin2}.  It will be necessary to develop a theory for sheaves of spectra.  We take some first steps in \S\ref{subsec:MTONCS}--\S\ref{subsec:TSOMS}; there is also recent work of Robalo and Schapira along these lines \cite{RS}.

\subsection{Coefficients and stable $\infty$-categories} 
\label{subsec:2.1}
 We will use $\Sigma$ for the suspension functor in a stable $\infty$-category.
If $\bk$ is an associative algebra spectrum, we write $\LMod(\bk)$ for its $\infty$-category of left module spectra --- it is a compactly generated presentable stable $\infty$-category (see \cite[\S 7.2.5]{higher-algebra}). When $\bk$ is a discrete ring, $\LMod(\bk)$ is an $\infty$-categorical enrichment of the usual unbounded derived category of $\bk$-modules.

If $\bk$ has the structure of an $\mathrm{E}_2$-algebra, then $\LMod(\bk)$ has a monoidal structure $(- \otimes_{\bk} -)$ that preserves colimits in both variables (see \cite[\S 7.1.2]{higher-algebra}).  In this case we will shorten our notation for this monoidal stable $\infty$-category to $\Mod(\bk)$.

We write $\St$ for the symmetric monoidal $\infty$-category of presentable stable $\infty$-categories, and continuous (colimit-preserving) functors.  We write $\St_{\bk}$ for the $\infty$-category of left $\Mod(\bk)$-module objects in $\St$.  A morphism in $\St_{\bk}$ is a ``$\bk$-linear continuous functor.'' We write $\bS$ for the sphere spectrum, which has canonically $\St_{\bS} \cong \St$.

We write $B\Pic(\bk)^{\not\simeq}$ for the full subcategory of $\St_{\bk}$ of objects that are equivalent to $\Mod(\bk)$, and $B\Pic(\bk) \subset B\Pic(\bk)^{\not\simeq}$ for the subcategory obtained by discarding non-invertible morphisms in $B\Pic(\bk)^{\not\simeq}$.  Thus $B\Pic(\bk)$ is a connected $\infty$-groupoid, with a distinguished object $\Mod(\bk)$. We will usually regard $B\Pic(\bk)$ as a pointed space, i.e. abuse notation and not distinguish between the $\infty$-groupoid and its nerve. The meaning of the notation is that there is a canonical homotopy equivalence between the space of based loops in $B\Pic(\bk)$ (i.e. the space of self-equivalences of $\Mod(\bk) \in \St_{\bk}$) and the space $\Pic(\bk)$ of $\otimes_{\bk}$-invertible objects of $\Mod(\bk)$. (The canonical homotopy equivalence is $F \mapsto F(\bk)$; this is a very special case of the Morita theory of \cite[\S 4.8]{higher-algebra})

\subsection{Sheaves and sheaf operations}
We write $\Shall(X,\bk) \in \St_{\bk}$ for the $\infty$-category of sheaves of $\bk$-module spectra on a finite-dimensional locally compact Hausdorff space $X$.  Here by finite-dimensional, we mean the space $X$ is paracompact and has finite covering dimension (cf. \cite[Definition 7.2.3.1]{higher-topoi}).  Formally, if $\Shv(X)$ denotes the $\infty$-topos associated to $X$ \cite[\S 6.5.4]{higher-topoi}, then $\Shall(X,\bk)$ is the $\infty$-category of contravariant functors $\Shv(X)^{\op} \to \Mod(\bk)$ that convert small colimits into small limits.  (In \cite[Notation 6.3.5.16]{higher-topoi}, this is $\Shv_{\Mod(\bk)}(X)$.) 

\subsubsection{Covers}
\label{subsec:covers}
One may also obtain $\Shall(X,\bk)$ as a localization \S\ref{subsec:localization} of the category of $\Mod(\bk)$-valued presheaves on $X$.  As we are assuming $X$ is finite-dimensional and locally compact, it is the full subcategory of presheaves $P$ that obey any of the following equivalent conditions:
\begin{enumerate}
\item Let $U$ be an open subset of $X$ and let $\{U_i\}_{i \in I}$ be an open cover of $U$.  Let $\cU$ denote the poset of open subsets of $U$ that are contained in at least one of the $U_i$.  Then the natural map
\[
P(U) \to \varprojlim_{V \in \cU} P(V)
\]
assembled from the restriction maps $P(U) \to P(V)$ is an isomorphism\footnote{Throughout the paper, we follow the convention, as in \cite{higher-topoi}, that by an isomorphism, we mean an invertible morphism in an $(\infty, 1)$-category, i.e. a morphism admitting a two-sided inverse up to homotopy.}.
\item Let $U$ be an open subset of $X$, and let $\{U_i\}_{i \in I}$ be an open cover of $U$ that is closed under finite intersections --- that is, suppose that for each $i,j$, there is a $k$ such that $U_i \cap U_j = U_k$.  (We will follow \cite[Def 4.5]{DanDan} and call such a covering a ``\v Cech cover.'')  Regarding $I$ as a poset and $i \mapsto U_i \mapsto P(U_i)$ as a functor on (the opposite of) this poset, the natural map
\[
P(U) \to  \varprojlim_{i \in I} P(U_i)
\]
assembled from the restriction maps $P(U) \to P(U_i)$ is an isomorphism.
\item Let $U$ be an open subset of $X$, and let $\{U_i\}_{i \in I}$ be an open cover of $U$ that has the following hypercovering property: every finite intersection $U_{i_1} \cap \cdots \cap U_{i_n}$ can be covered by open subsets from $\{U_i\}_{i \in I}$.  (We will follow \cite[Def. 4.5]{DanDan} and call a covering of $U$ with this property a ``complete cover.'').  Then regarding $I$ as a poset, the natural map
\[
P(U) \to \varprojlim_{i \in I} P(U_i)
\]
is an isomorphism.
\end{enumerate}

\begin{remark*}
Here (1)$\Leftrightarrow$(2) holds for arbitrary topological spaces, since for any open cover $\{U_i\}_{i\in I}$ of $U$, the associated \v Cech cover as a subposet in $\cU$ from (1) is cofinal. The equivalence (1)$\Leftrightarrow$(3) (for sheaves of spaces) is equivalent to the condition that $\Shv(X)$ is hypercomplete (cf. \cite[Theorem 6.5.3.13, \S 6.5.4]{higher-topoi}), which is saying that every sheaf in $\Shv(X)$ satisfies the stronger condition of hyperdescent than usual descent. This does \emph{not} hold for every locally compact Hausdorff space $X$. For a thorough discussion about the differences between descent and hyperdescent, we refer the reader to  \cite[\S 6.5.4]{higher-topoi}, especially Counterexample 6.5.4.8 in \emph{loc. cit.}. However, for a finite-dimensional locally compact Hausdorff space $X$, it is a paracompact Hausdorff space of finite covering dimension, hence by \cite[Theorem 7.2.3.6, Corollary 7.2.1.12]{higher-topoi}, $\Shv(X)$ is hypercomplete. Then $\Shall(X,\bk)\simeq \Shv(X)\otimes\Mod(\bk)$ \cite[\S 1.3.1]{SAG} satisfies hyperdescent as well. 
\end{remark*}

\subsubsection{Locally constant sheaves, operations} 
\label{subsec:lcso}
One defines the constant sheaf with fiber $M \in \Mod(\bk)$ to be the sheafification of the presheaf that takes the constant value $M$, and call a sheaf locally constant if it is isomorphic to such a constant sheaf in some open cover of $M$.  We write $\Loc(X,\bk) \subset \Shall(X,\bk)$ for the full subcategory of locally constant sheaves on $X$.  

When $X$ is locally contractible, a necessary and sufficient condition for $F$ to be locally constant is for each there to exist a complete covering (in the sense of \S\ref{subsec:covers}(3))  by contractible open subsets $U_i \subset X$, such that $F(U_i) \to F(U_j)$ is an isomorphism whenever $U_j \subset U_i$.  In fact if $\{U_i\}_{i \in I}$ is such a covering then $\Loc(X,\bk)$ is equivalent to the full subcategory of $\Fun(I^{\op},\Mod(\bk))$ spanned by functors that carry every arrow in $I^{\op}$ to an equivalence.  (Let us call functors with this property ``locally constant functors.'')

If $f:Y \to X$ is a continuous map, the pullback functor $f^*:\Shall(X) \to \Shall(Y)$ is a continuous functor, with a right adjoint $f_*$ that is not always continuous.  If $f = j$ is an open embedding then $j^*$ (which in the case of an open embedding we also denote by $j^!$) has a left adjoint $j_!$, the extension-by-zero functor.  As it is a left adjoint, it is automatically continuous.

If $f$ is a proper continuous map between locally compact spaces, then $f_*$ (which in the case of a proper map we also denote by $f_!$) is continuous.  For a general map between locally compact spaces, we define $f_! = \overline{f}_* \circ j_!$, where $j,\overline{f}$ is a factorization into an open inclusion $j$ and a proper map $\overline{f}$.

We also have the standard sheaf operations for an inclusion of a subset $\iota: Z\hookrightarrow X$, referred to as restrictions and sections with support respectively:
\begin{align*}
F_Z:=\iota_!\iota^*F,\ \Gamma_ZF:=\iota_*\iota^!F. 
\end{align*}

\subsubsection{Deligne gluing} 
\label{subsec:deligne-gluing}
Let $\LCH$ be the $1$-category of finite-dimensional locally compact Hausdorff spaces and continuous maps.  
It is straightforward to verify that there is a functor $\LCH^\op \to \St_{\bk}$, carrying $X$ to $\Shall(X,\bk)$ and $f$ to $f^*$.  It is more difficult to verify that there is a functor $\LCH \to \St_{\bk}$ carrying $X$ to $\Shall(X;\bk)$ and $f$ to $f_!$ --- this is again straightforward for either the subcategory of $\LCH$ whose morphisms are proper maps, or that whose morphisms are open inclusions, but to glue these requires an $\infty$-categorical update to the machine in \cite{SGA4-Del} --- this is carried out in great generality in \cite[Cor. 0.2]{Yifeng1}, \cite[\S 2.2]{Yifeng2}. See also \cite{Volpe}.

\subsubsection{Cosheaf perspective}
\label{subsec:cosheaf-perspective}
If $X$ is a locally compact Hausdorff space, then for each open $U \subset X$, the functor $F \mapsto \Gamma_c(U,F)$ is continuous, as it is the composite of restriction to $U$ and proper pushforward to a point.  In contrast, the functor $F \mapsto \Gamma(U,F)$ is not usually continuous (for example if $U$ contains a closed, infinite discrete set $Z$, then the constant sheaf on $Z$ is a direct sum, while its image under $\Gamma(U,-)$ is a direct product).  Note that $\Gamma_c(U,F)$ is covariant in the $U$-variable --- it is a cosheaf in the sense that, whenever $\{U_i\}_{i \in I}$ is a family of open subsets that is closed under finite intersections and that covers $U$, the natural map
\[
\varinjlim_{i \in I} \Gamma_c(U_i;F) \to \Gamma_c(U;F)
\]
is an isomorphism.

In \cite{Tamarkin2}, Tamarkin uses such cosheaves systematically --- but he calls them sheaves.  Indeed on a locally compact Hausdorff space, the $\infty$-categories of sheaves of $\bk$-modules and of cosheaves of $\bk$-modules are equivalent, via the assignment $F \mapsto \Gamma_c(-;F)$.  This is one formulation of Verdier duality \cite[\S5.5.5]{higher-algebra}.

\subsubsection{Descent}
\label{subsec:descent}
When regarded as a $\St_{\bk}$-valued contravariant functor on $\LCH$, the $\infty$-categories $\Shall(-;\bk)$ themselves form a sheaf, so that whenever $\{U_i\}_{i \in I}$ is a  covering sieve of $X$, the restriction maps $\Shall(X;\bk) \to \Shall(U_i;\bk)$ assemble to an equivalence
\begin{equation}\label{eq: Sh descent}
\Shall(X;\bk) \stackrel{\sim}{\to} \varprojlim_{i \in I} \Shall(U_i;\bk)
\end{equation}
where the limit is taken in the $\infty$-category $\St_{\bk}$.  The same is true with $\Shall(-,\bk)$ replaced by $\Loc(-,\bk)$, or by the categories $\Sh_{\cS}(-;\bk)$ and $\Sh_{\Lambda}(-;\bk)$ discussed in \S\ref{subsec:constructible-sheaves}. A proof of (\ref{eq: Sh descent}) can be found in Appendix \ref{appendix B, descent}.

\subsection{Generators}
\label{subsec:generators}
In classic sheaf theory texts, many identities between sheaf operations are verified by taking suitable resolutions (injective, flabby, soft,\ldots).  The theory of presentable categories gives an alternative which is more general (as it applies to unbounded complexes, or to sheaves of spectra), and in some ways simpler.

The sheaves of the form $j_{U,!} \bk$, where $j_U:U \hookrightarrow X$ runs though all open subsets of $X$, make a small set of generators for $\Shall(X,\bk)$ --- indeed, $j_{U,!} \bk$ represents the sections functor $F \mapsto \Gamma(U;F)$.  We may use this to verify sheaf operation identities by the following device: if $\phi_1,\phi_2$ are continuous functors $\Shall(X) \to \cC$ in $\St_{\bk}$, and $n:\phi_1 \to \phi_2$ is a natural transformation between them (i.e. $n$ is a morphism in $\Fun(\Shall(X),\cC)$), then $n$ is an isomorphism if and only if $n_U:\phi_1(j_{U,!} \bk) \to \phi_2(j_{U,!}\bk)$ is an isomorphism for all $U$.  

(We warn once again, however, that the objects $j_{U,!} \bk$ are not usually compact --- they are $\aleph_1$-compact in the sense of \cite[\S 5.3.4]{higher-topoi})

\subsection{Proper base change}
We record a consequence of the nonabelian proper base change theorem of \cite[\S 7.3]{higher-topoi}. Given a Cartesian diagram of locally compact topological spaces,
\[
\xymatrix{
X' \ar[r]^{q'}  \ar[d]_{p'} & X \ar[d]^p\\
Y' \ar[r]_q & Y
}
\]
the natural map $q^* p_! \to p'_! q'^*$ is an isomorphism of functors $\Shall(X) \to \Shall(Y')$.

There is an expectation, which has been seen through in other contexts \cite{Nick}, that ``six operations'' formalisms can be encoded in $(\infty,2)$-categorical language --- as a functor from the $(2,2)$-category of spaces, correspondences, and 2-morphisms between correspondences, to the $(\infty,2)$-category of categories.  We don't use it in this paper.

\subsection{Monoidal structure and functors from kernels}
\label{subsec:kernel}
The monoidal structure on $\Mod(\bk)$ induces a monoidal structure on $\Shall(X,\bk)$, which we also denote by $(-\otimes_{\bk} -)$.  Formally, one applies \cite[Lemma 2.2.1.9]{higher-algebra} to the monoidal category of presheaves on $X$, after noting that if $P \to P'$ induces an equivalence after sheafification, then so does $P \otimes_{\bk} Q \to P' \otimes_{\bk} Q$.  The pullback functors $f^*$ are monoidal.

For each $K \in \Shall(X \times Y;\bk)$, we define a continuous functor $K \circ: \Shall(Y) \to \Shall(X)$, by the formula
\[
G \mapsto \proj_{1,!}\left(K \otimes_{\bk} \proj_2^* G\right)
\]
(In the notation of \cite[Def. 3.6.1]{KS}, this is $\Phi_K$).  We also define $\circ K:\Shall(Y) \to \Shall(X)$ (denoted $F \mapsto F \circ K$) by
$
G \mapsto \proj_{1,!}\left(\proj_2^* G \otimes_{\bk} K \right)
$.
Unless $\bk$ is $\mathrm{E}_3$-commutative or better, we do not usually have a natural isomorphism $K \circ F \cong F \circ K$.  

\subsubsection{Projection formula}
\label{subsubsec:projection-formula}
The ``projection formula'' is a pair of natural isomorphisms
\begin{equation}
\label{eq:projection-formula}
f_!G \otimes_{\bk} F \cong f_!(G \otimes_{\bk} f^* F)  \qquad F \otimes_{\bk} f_! G \cong f_!( f^* F \otimes_{\bk} G)
\end{equation}
Kashiwara and Schapira prove this as \cite[Prop. 2.6.6]{KS}. For sheaves of spectra, these isomorphisms can be obtained as base-change maps, constructed from the Cartesian diagrams
\[
\xymatrix{
X \ar[r]^f \ar[d]_{\Gamma_f} & Y \ar[d]^{\Delta_Y} \\
X\times Y \ar[r]_{f \times \mathrm{id}_Y} & Y \times Y
}
\qquad \xymatrix{
X \ar[r]^f \ar[d]_{\Gamma_f^T} & Y \ar[d]^{\Delta_Y} \\
Y\times X \ar[r]_{\mathrm{id}_Y \times f} & Y \times Y
}
\]
applied to the sheaf $F \boxtimes G$ on $X \times Y$ or $G \boxtimes F$ on $Y \times X$. Here $\Gamma_f = (\mathrm{id}_X \times f) \circ \Delta_X$ and $\Gamma_f^T = (f \times \mathrm{id}_X) \circ \Delta_X$.

From \eqref{eq:projection-formula} and the adjunction between $f_!$ and $f^!$, one can construct (exactly as in \cite[Prop. 3.1.11]{KS}) natural transformations
\[
f^! F \otimes_{\bk} f^* G \to f^!(F \otimes_{\bk} G) \qquad f^* F \otimes f^! G \to f^!(F \otimes_{\bk} G)
\]
If $f:Y \to X$ is a fiber bundle whose fiber is a topological manifold, or more generally a topological submersion in the sense of \cite[Def. 3.3.1]{KS}, these become isomorphisms.  In particular, putting $\omega_{Y/X} := f^!(\bk)$, we have canonically 
\begin{equation}
\label{eq:submersion}
f^! F \cong f^* F \otimes_{\bk} \omega_{Y/X} \cong \omega_{Y/X} \otimes_{\bk} f^* F
\end{equation}
when $f$ is a topological submersion --- the proof is the same as \cite[Prop. 3.3.2(ii)]{KS}.  Moreover $\omega_{Y/X}$ is invertible, it is locally isomorphic to a suspension of the constant sheaf on $Y$.  Note that \eqref{eq:submersion} implies that, for a topological submersion, $f^!$ is a continuous functor.  

The equation \eqref{eq:submersion} also holds when $f$ is a fiber bundle whose fiber is a manifold with boundary.  In that case the restriction of $f$ to the complement of the boundary (write it as $Y^{\circ})$ is a topological submersion, and $\omega_{Y/X}$ is the extension by zero of $\omega_{Y^{\circ}/X}$ along the inclusion $Y^{\circ} \to Y$.   

\subsubsection{The kernel $K^{-1}$}
If $\Delta_X:X \to X \times X$ denotes the diagonal embedding, then $(\Delta_{X,*} \bk) \circ$ and $\circ (\Delta_{X,*}\bk)$ are both isomorphic to the identity functor on $\Shall(X)$.  Given a kernel $K \in \Shall(X \times Y)$, there is another kernel $K^{-1} \in \Shall(Y \times X)$ along with a natural map $K^{-1} \circ K \to \Delta_{X,*} \bk$ --- it is given by the formula \cite[Eq. 1.21]{GKS}
\begin{equation}
\label{eq:K-inverse}
K^{-1} = \mathrm{flip}(\shHom(K,\omega_{X\times Y/Y}))
\end{equation}
where $\shHom(K,-)$ is the right adjoint functor to $K \otimes_{\bk} (-)$, $\omega_{X \times Y/Y} \cong \omega_X \boxtimes \bk_Y$ as in \S\ref{subsubsec:projection-formula}, and $\mathrm{flip}$ is the obvious equivalence between $\Shall(X \times Y)$ and $\Shall(Y \times X)$.

\subsection{Localization}
\label{subsec:localization}

If $C$ is a presentable $\infty$-category, the following data are equivalent to each other \cite[Prop. 5.2.7.4]{higher-topoi}:
\begin{enumerate}
\item Another presentable $\infty$-category $LC$ together with a continuous functor $C \to LC$ whose right adjoint is fully faithful.
\item A not necessarily continuous functor $L:C \to C$, together with a natural transformation $\eta_L:1_C \to L$ that becomes an isomorphism after applying $L$ --- i.e. that has $L(\eta_L):L \to L^2$ an isomorphism.
\end{enumerate}
When $C$ is stable, these data are furthermore equivalent to
\begin{enumerate}
\item[(3)] A \emph{localizing subcategory} of $C$, i.e. a full subcategory $C' \subset C$ that is closed under infinite direct sums, and that is also presentable.
\end{enumerate}
(We warn that some authors do not require that a ``localizing subcategory'' is presentable, though there is a strong set-theoretic axiom (Vopenka's principle) which implies that presentability of $C'$ is automatic.)

(1) determines (2) by taking $L:C \to C$ to be the composite of $C \to LC$ with its right adjoint.  (2) determines (1) by taking $LC$ to be the essential image of $L$.  In the stable setting (1) and (2) determine (3) by taking $C'$ to be the kernel of $L$.

If $C' \subset C$ is a localizing subcategory of a presentable stable $\infty$-category, we sometimes write $C/C'$ or $L_{C'} C$ for the right orthogonal of $C'$ in $C$, i.e. for full subcategory of $C$ spanned by objects $c$ with $\Hom(c',c) = 0$ for all $c' \in C'$.  This is how (3) determines (1) --- the fully faithful inclusion $C/C' \to C$ has a left adjoint $C \to C/C'$.  

A $\bk$-linear structure on $C$ that preserves $C'$ induces a $\bk$-linear structure on $C/C'$, and the construction $(C' \subset C) \mapsto C/C'$ is functorial.  In fact it extends to a functor 
\begin{equation}
\label{eq:cone-of-categories}
\Fun(\Delta^1,\St_{\bk}) \to \St_{\bk}
\end{equation}
taking the arrow $C \to C'$ to the colimit of the diagram $0 \leftarrow C' \to C$.

A basic example is the case of restriction to an open subset.  That is, if $C$ is the $\infty$-category of sheaves on a space and $LC$ is the $\infty$-category of sheaves on an open subset, then the restriction functor $j^*:C \to LC$ has a fully faithful right adjoint $j_*$, and $L = j_* j^*$.  

\subsection{Microlocal theory of non-constructible sheaves}

\label{subsec:MTONCS}

In a moment, we will restrict our attention to sheaves that are locally constant on the strata of some Whitney stratification, and study them with the help of stratified Morse theory.  But Kashiwara and Schapira give tools for analyzing more general sheaves microlocally.  We discuss some of these tools in this section from the standpoint of sheaves of spectra.  This material is ``optional'' (it is not used elsewhere in this paper) so we will be somewhat terse.

\subsubsection{Noncharacteristic deformation lemma}
\label{subsec:noncharacteristic-deformation-lemma}

This is the name given to \cite[Prop. 2.7.2]{KS}.  The ``unbounded'' analog is the following:

\begin{prop*}
Let $X$ be a Hausdorff space and $F \in \Shall(X,\bk)$ a sheaf on $X$.  Suppose that $\{U_t\}_{t \in \bR}$ is a family of open subsets of $X$ obeying the following:
\begin{enumerate}
\item $U_t = \bigcup_{s < t} U_s$ for all $t \in \bR$
\item Whenever $t \geq s$, the set $\overline{U_t - U_s} \cap \mathrm{supp}(F)$ is compact
\item Setting $Z_s = \bigcap_{t > s} \overline{U_t - U_s}$\footnote{In the source \cite[Prop. 2.7.2]{KS}, there is a typo in the placement of the overline to indicate the closure.}, whenever $s \leq t$ and $x \in Z_s - U_t$:
\begin{equation}
\label{eq:not-continuous-uh-oh}
(\Gamma_{X - U_t} F)_x = 0 
\end{equation}
Then for all $t \in \bR$, the natural map is an isomorphism:
\[
 \Gamma\left(\bigcup_{s\in \bR} U_s;F\right) \to \Gamma(U_t;F)
\]
\end{enumerate}
\end{prop*}

\begin{proof}
As in \cite[Prop 2.7.2]{KS}, the proof is an application of the following:
\begin{quote}
{\bf Kashiwara lemma:} Let $P:\bR^{\op} \to \Mod(\bk)$ be a presheaf on the poset $\bR$.  Suppose that for each $s \in \bR$, the maps
\[
\varinjlim_{t > s} P(t) \to P(s) \qquad P(s) \to \varprojlim_{t < s} P(t)
\]
are both isomorphisms.  Then $P(s_1) \to P(s_2)$ is an isomorphism for every $s_1 \geq s_2$.
\end{quote}
When $\bk = \bZ$ and under some boundedness hypotheses, a version of this first appeared in the proof of \cite[Th. 1.2]{Kash}.  We learned the following proof (which does not require such hypotheses) from Dmitri Pavlov --- perhaps it is the same proof that Kashiwara ommited for \cite[Lem. 1.3]{Kash}.  For each $s_2 \leq t \leq s_1$, let $C(t)$ denote the cone on $P(s_1) \to P(t)$.  Fix a map $f: \Sigma^k \bk \to C(s_2)$.  We will show that $f$ is nullhomotopic, by showing that $f$ factors through $C(s_1) = 0$.  As $C(s_2) = \varinjlim_{t > s_2} C(t)$ and $\Sigma^k \bk$ is compact in $\Mod(\bk)$, the map $f$ factors through $C(t)$ for some $t > s_2$.  Let $r$ be the supremum of all $t \leq s_1$ for which such a factorization can be found.  As $C(r) = \varprojlim_{t < r} C(t)$, it follows that a factorization through $C(r)$ can be found.  But then we must have $r = s_1$, for if this is false then (using $C(r) = \varinjlim_{t > r} C(t)$) we may find $t > r$ such that $f$ factors through $C(t)$, violating the definition of $r$.\end{proof}

\subsubsection{Singular support}
\label{subsubsec:sing-supp-27}
The singular support $\SS(F) \subset T^* X$ of a sheaf $F \in \Shall(X,\bk)$ is defined by defining its complement.  We say that $(x,\xi) \notin \SS(F)$ if the following conditions holds for some neighborhood $U \ni (x,\xi)$:
\begin{quote}
If $\psi$ is a real-valued $C^1$-function, defined in a neighborhood of $x_1$ and with $d\psi_{x_1} \in U$, then 
\begin{equation}
\label{eq:def-of-sing-supp}
(\Gamma_{\{x \mid \psi(x) \geq \psi(x_1)\}}(F))_{x_1} = 0. 
\end{equation}
\end{quote}
The stalk of $\Gamma_{\{x \mid \psi(x) \geq \psi(x_1)\}}(F)$ at $x_1$ depends only on the germ of $\psi$ near $x_1$.  

This definition makes clear that $\SS(F)$ is a closed conic subset of $T^* X$, and that it is functorial for pullback by $C^1$-homeomorphisms.  Other desirable properties are not immediate, and need a second look in the setting of unbounded complexes or spectra. In particular, since $\Gamma_Y$ is not a continuous functor, it is not immediate that if one is given infinitely many sheaves with $\SS(F_i) \subset Z$, for a closed conic subset $Z$, then one also has $\SS(\bigoplus F_i) \subset Z$.

In the following, we include another convenient formulation of singular support borrowing Tamarkin's notion of an ``$\Omega$-lense" to address the latter issue. With the noncharacteristic deformation lemma in the ``unbounded" situation at hand, one can prove several equivalent definitions of (the complement of) singular support as in \cite[Proposition 5.1.1]{KS} without any essential change. In fact, the notion of an ``$\Omega$-lense" is somewhat implicit in \emph{loc. cit.}.

Given a conic open subset $\Omega\subset T^*X$, we say a quadruple $(f, \chi, a, \epsilon_0)\in C^\infty(X)\times  C^\infty(X)\times \bR\times \bR_+$ (one can replace $C^\infty(X)$ by $C^r(X)$ for any $r\geq 1$) is an \emph{$\Omega$-lense} if the followings hold:
\begin{itemize}
\item the function $\chi$ with range in $[0,1]$ but not constantly $0$ (as a characteristic function) is compactly supported on $f^{-1}([a-\epsilon_0,a+\epsilon_0])$;
\item Define $H_{a,\epsilon}:=\{x: f(x)\leq a+\epsilon\chi(x)\}$ and $H^\dagger_{a,\epsilon}:=\{x: f(x)<a+\epsilon\chi(x)\}$. Then for any $\epsilon\in [0,\epsilon_0]$ and $x\in (H_{a,\epsilon}-H^\dagger_{a,\epsilon})\cap \overline{\{\chi>0\}}$, we have $\big(x,(df-\epsilon d\chi)_x\big)\in \Omega$. 
\end{itemize}
For a sheaf $F\in \Shall(X,\bk)$, we say that $F$ is \emph{$\Omega$-noncharacteristic} if for any $\Omega$-lense $(f,\chi, a,\epsilon_0)$, we have an isomorphism 
\begin{align}\label{eq: c,a,b open}
\Gamma(H^\dagger_{a,\epsilon_0};F)\overset{\sim}{\longrightarrow} \Gamma(H^\dagger_{a,0};F)
\end{align}
for the restriction map. We can equally characterize the $\Omega$-noncharacteristic property of $F$ by replacing (\ref{eq: c,a,b open}) with 
\begin{align}\label{eq: c,a,b closed}
\Gamma(H_{a,\epsilon_0}\cap \overline{\{\chi>0, f\geq a-\epsilon_0\}};F)\overset{\sim}{\longrightarrow} \Gamma(H_{a,0}\cap \overline{\{\chi>0, f\geq a-\epsilon_0\}};F). 
\end{align}
Since $H_{a,\epsilon}\cap \overline{\{\chi>0, f\geq a-\epsilon_0\}}$ is compact for any $\epsilon\in [0,\epsilon_0]$, both sides of (\ref{eq: c,a,b closed}) preserve (small) colimits in $F$. 

One can show that for any two conic open subsets $\Omega_1$ and $\Omega_2$, $F$ is $\Omega_1\cup \Omega_2$-noncharacteristic if and only if $F$ is both $\Omega_1$ and $\Omega_2$-noncharacteristic (the argument is a standard application of the noncharacteristic deformation lemma). Then $\SS(F)$ is defined to be the complement of the largest $\Omega$ so that $F$ is $\Omega$-noncharacteristic. Now it is clear that $\SS(\bigoplus\limits_i F_i)\subset\overline{\bigcup\limits_i \SS(F_i)}$. It is also clear from (\ref{eq: c,a,b open}) that $\SS(\prod\limits_i F_i)\subset\overline{\bigcup\limits_i \SS(F_i)}$.

\subsubsection{A ``dual" view of $\Omega$-lenses}\label{subsubsec: dual view Omega-lenses}
For each $\Omega$-lense $(f, \chi, a, \epsilon_0)$, let $C^\circ_{a,-\epsilon}=X-H_{a,\epsilon}$ (resp. $C_{a,-\epsilon}=X-H_{a,\epsilon}^\dagger$), then we have the open inclusion $C^\circ_{a,-\epsilon_0}\hookrightarrow C^\circ_{a,0}$. Let $\pi_{\pt}: C^\circ_{a,-\epsilon}\to \pt$ and $j_{C,-\epsilon}:  C^\circ_{a,-\epsilon}\hookrightarrow X$ be the open inclusions. 
For any $M\in \Mod(\bk)$, by adjunction 
\begin{align*}
\Hom_{\Shall(X)}(F, (j_{C,-\epsilon})_*\pi_\pt^!M)\cong \Hom_{\Mod(\bk)}(\Gamma_c(C^\circ_{a,-\epsilon}, F|_{C^\circ_{a,-\epsilon}}), M)
\end{align*}
Using the equivalent charaterization of $\Omega$-noncharacteristic by (\ref{eq: c,a,b closed}), we see that $F$ is $\Omega$-noncharacteristic if and only if for all $\Omega$-lenses $(f, \chi, a, \epsilon_0)$ and $M\in \Mod(\bk)$, the natural morphism 
\begin{align*}
\Hom_{\Shall(X)}(F, \Gamma_{C^\circ_{a,0}}\omega_X\otimes_\bk M)\to \Hom_{\Shall(X)}(F,  \Gamma_{C^\circ_{a,-\epsilon_0}}\omega_X\otimes_\bk M)
\end{align*}
is an isomorphism.

\subsubsection{Standard (dualizing) sheaves associated to $\Omega$-lenses}\label{subsubsec: std, Omega-lenses}
For any $\Omega$-lense $(f, \chi, a, \epsilon_0)$ and any $0\leq \delta_0<\delta_1\leq \epsilon_0$, we associate the sheaf 
\begin{align*}
F_{\delta_0, \delta_1}:=\Cone(\bk_{H_{a,\delta_0}^\dagger}\to \bk_{H_{a,\delta_1}^\dagger})=\bk_{H^\dagger_{a,\delta_1}-H^\dagger_{a,\delta_0}}. 
\end{align*}
  In the following, we call such a sheaf the \emph{standard sheaf} associated to the $\Omega$-lense $(f, \chi, a, \epsilon_0)$ and $\delta_0<\delta_1$. If $\delta_0=0$ and $\delta_1=\epsilon_0$, we simply say the sheaf is associated to the $\Omega$-lense $(f, \chi, a, \epsilon_0)$. 

Let
\begin{align*}
\check{F}_{\delta_0,\delta_1}:=\fiber(\Gamma_{C^\circ_{a,-\delta_0}}\omega_X\to \Gamma_{C^\circ_{a,-\delta_1}}\omega_X)\cong \Gamma_{C^\circ_{a,-\delta_0}-C^\circ_{a,-\delta_1}}\omega_X,
\end{align*}
For any $M\in \Mod(\bk)$, we call $\check{F}_{\delta,\delta_1}\otimes_\bk M$ a \emph{standard dualizing sheaf} associated to the $\Omega$-lense $(f, \chi, a, \epsilon_0)$, $0\leq \delta_0<\delta_1\leq \epsilon_0$, and the module $M$. 
We have the following lemma which is crucial for the enhancement of Kashiwara-Schapira's theory to the spectral setting, especially for proving the key statements in \S\ref{subsec:contact-transformations}. 

\begin{lemma*}
For any open conic subset $\Omega$ in $T^*X$. The standard sheaves (resp. standard dualizing sheaves) associated to all $\Omega$-lenses (resp. $\Omega$-lenses and $\bk$-modules) generate the left (resp. right) orthogonal complement of $\Sh_{T^*X-\Omega}(X;\bk)$ under colimits (resp. limits). 
\end{lemma*}
\begin{proof}
The generation under colimits for the left orthogonal complement immediately follows from the definition. The generation under limits for the right orthogonal complement follows from the ``dual" version in \S\ref{subsubsec: dual view Omega-lenses}. 
\end{proof}

\begin{remark*}
By GKS, for any compactly supported contact Hamiltonian isotopy $\varphi_t, t\in [0,1], \varphi_0=id$ on $T^\infty X$, whose support is contained in $\Omega^\infty$, the canonical kernel $K_t$ associated to $\varphi_t$ (\ref{subsec:GKS}) preserves both the left and right orthogonal complements of $\Sh_{T^*X-\Omega}(X)$ and induces an auto-equivalence on them. 
\end{remark*}

 \subsubsection{Regular subanalytic $\Omega$-lenses and their microlocal sheaf theory}\label{subsubsec: regular Omega-lenses}
 
 Now we assume that $\cU\subset T^{*, \circ}X$ is open conic.  Using the proof of \cite[Proposition 5.1.1]{KS}, it suffices to consider some special class of subanalytic $\cU$-lenses to generate the left/right orthogonal complement of $\Sh_{T^*X-\cU}(X)$. For example, given $\cU$ in $\text{(1)}_\omega$ of \emph{loc. cit.}, we can restrict to the class of $\widetilde{\Omega}_t(a)$ constructed explicitly in its proof about $\text{(1)}_\omega\implies\text{(3)}$, and for each pair $t_1<t_2$ sufficiently close, we can find a $C^1$ $\cU$-lense $(f,\chi, c,\epsilon_0)$ with $H_{c, 0}^\dagger=\widetilde{\Omega}_{t_1}(a)$ and $H_{c,\epsilon_0}^\dagger=\widetilde{\Omega}_{t_2}(a)$. Clearly we can choose the $\cU$-lense in the class $C^r$, by making $\partial\widetilde{\Omega}_t(a)$ of class $C^r$. These $\cU$-lenses are all regular in the sense of the following definition. 
 
\begin{defn*}
We say a subanalytic $\cU$-lense $(f,\chi, c,\epsilon_0)$ is \emph{regular}, if  there is a Whitney stratification of $X$ with 5 strata: 
\begin{align*}
\{C_+:=\partial H_{c,\epsilon_0}^\dagger-\partial H_{c,0}^\dagger, C_-:=\partial H_{c,0}^\dagger-\partial H_{c,\epsilon_0}^\dagger, \partial C_+, H_{c,\epsilon_0}^\dagger-\overline{H_{c,0}^\dagger}, X-\overline{H_{c,\epsilon_0}^\dagger-\overline{H_{c,0}^\dagger}}\},
\end{align*}
that is a compatible with $H_{c,\epsilon_0}^\dagger-\overline{H_{c,0}^\dagger}$ (cf. \cite[\S A.2.4 (2), (i)]{J-Perverse}), such that $H_{c,\epsilon_0}^\dagger-\overline{H_{c,0}^\dagger}$ is contractible and each of $\overline{C}_+, \overline{C}_+$ is diffeomorphic to a closed ball (of $\dim X-1$). 
\end{defn*} 

For any regular subanalytic $\cU$-lense $\cL:=(f,\chi, c,\epsilon_0)$ in the class $C^r$, let $\Lambda_\cL\subset T^{*,\infty}X$ be the union of the positive conormal direction (i.e. $d(f-\epsilon_0\chi)$ and $df$ respectively) along $\overline{C}_+$ and $\overline{C}_-$. It's clear that $\Lambda_\cL$ is a $C^{r-1}$ Legendrian unknot. Note that in this case $\check{F}_{\delta_0, \delta_1}=F_{\delta_0,\delta_1}\otimes_\bk N$ for some $N\in \Mod(\bk)$. 
Note also that the sheaf theory with singular support in $\Lambda_\cL$ is easy to understand from stratified Morse theory. Take any generic $(x,\xi)$ in $\Cone(\Lambda_\cL)$ (where $\mu_{(x,\xi)}$ is well defined) and $v$ a stratified Morse function as in Subsection \ref{subsubsec: microstalks}, then $\mu_{(x,\xi)}$ induces an equivalence 
\begin{align}\label{eq: Lambda_cL, Modk}
\Sh_{\Cone(\Lambda_\cL)\cup\zeta_X}(X;\bk)/\Loc(X;\bk)\overset{\sim}{\longrightarrow} \Mod(\bk). 
\end{align}
and the inverse functor is sending $\bk$ (up to a possible suspension) to the standard sheaf associated with $\cL$. Here the left and right orthogonal complements of $\Loc(X;\bk)$ are identical.

\subsubsection{The microlocal Morse lemma}\label{cor: MML}

\begin{corollary*}[Corollary 5.4.19 \cite{KS}]
Let $F\in \Shall(X;\bk)$ and let $f: X\to \bR$ be a $C^1$-function. Assume that $f|_{\Supp(F)}: \Supp(F)\to \bR$ is proper. 
\begin{itemize}
\item[(i)] For any $-\infty<a<b<\infty$, if $\{(x,df_x): x\in f^{-1}[a,b)\}\cap \SS(F)=\varnothing$, then the restriction morphisms give isomorphisms: 
\begin{align}\label{eq: i, microlocal Morse lemma}
\Gamma(f^{-1}(-\infty, b); F)\overset{\sim}{\to} \Gamma(f^{-1}(-\infty, a]; F)\overset{\sim}{\to} \Gamma(f^{-1}(-\infty, a); F). 
\end{align}

\item[(ii)] For any $-\infty<a<b<\infty$, if $\{(x,-df_x): x\in f^{-1}(a,b]\}\cap \SS(F)=\varnothing$ (resp. $\{(x,-df_x): x\in f^{-1}[a,b)\}\cap \SS(F)=\varnothing$), then the restriction (resp. co-restriction) morphisms give isomorphisms: 
\begin{align}
\label{eq: ii, microlocal Morse lemma}&\Gamma(f^{-1}(a,\infty); F)\overset{\sim}{\to}\Gamma(f^{-1}[b, \infty); F) \overset{\sim}{\to}\Gamma(f^{-1}(b, \infty); F)\\
\label{eq: ii, microlocal Morse lemma, Gamma_c}&(\text{resp. }\Gamma_c(f^{-1}(-\infty,a); F)\overset{\sim}{\to} \Gamma_c(f^{-1}(-\infty, b); F)). 
\end{align}
 \end{itemize}

\end{corollary*}

\begin{proof}
(i) Since $\Supp(F)_{[a,b]}:=f^{-1}([a,b])\cap \Supp(F)$ is compact, there is a compactly supported $C^1$-function $\chi: X\to [0,1]$ such that $\chi|_{\Supp(F)_{[a,b]}}=1$. Let $\Omega=T^*X-\SS(F)$. Then by assumption, $(f, \chi, a, b-a)$ is an $\Omega$-lense. Since $\Supp(F)\cap f^{-1}(-\infty,a')\subset \{x: f(x)<a+(a'-a)\chi(x)\}$, for all $a'\in [a,b]$, and $F$ is $\Omega$-noncharacteristic, the restriction morphisms are both isomorphisms: 
\begin{align*}
\Gamma(f^{-1}(-\infty,a');F)\overset{\sim}{\to} \Gamma(\{x: f(x)<a+(a'-a)\chi(x)\};F)\overset{\sim}{\to} \Gamma(f^{-1}(-\infty,a);F), 
\end{align*} 
for all $a'\in [a, b]$. Since $\Gamma(f^{-1}(-\infty,a];F)\simeq \varinjlim_{a'>a}\Gamma(f^{-1}(-\infty,a');F)$, the isomorphisms in (\ref{eq: i, microlocal Morse lemma}) follow.  

(ii) The isomorphism (\ref{eq: ii, microlocal Morse lemma}) follows from applying (i) to the function $-f$. To see (\ref{eq: ii, microlocal Morse lemma, Gamma_c}), consider the natural commutative diagram with each row a fiber sequence:  
\begin{align*}
\xymatrix{
\Gamma_c(f^{-1}(-\infty, a);F)\ar[d]\ar[r]& \Gamma_c(X;F)\ar[r]\ar@{=}[d] &\Gamma_c(f^{-1}([a, \infty);F)\ar[d]^{\varphi}\\
\Gamma_c(f^{-1}(-\infty, b);F)\ar[r] &\Gamma_c(X;F)\ar[r] &\Gamma_c(f^{-1}([b, \infty);F)
}.
\end{align*}
It suffices to show that $\varphi$ is an isomorphism. By the properness of $f|_{\Supp(F)}$, this is equivalent to showing that $\varphi': \Gamma(f^{-1}([a, d];F)\to \Gamma(f^{-1}([b, d];F)$ is an isomorphism, for a fixed $d\in (b,\infty)$. 

We observe that under the condition $\{(x,-df_x): x\in f^{-1}[a,b)\}\cap \SS(F)=\varnothing$, there exists $\ep>0$ such that $\{(x,-df_x): x\in f^{-1}(a-\ep,b)\}\cap \SS(F)=\varnothing$ as well. This can be seen as follows. First, by properness of $f|_{\Supp(F)}$, we have $f(\{x: -df_x=0\}\cap \Supp(F))$ is closed, Since it is disjoint from $[a, b)$, it is disjoint from $(a-\epsilon_1, b)$ for some $\epsilon_1>0$. Second, let 
\begin{align*}
&W_f:=\{(x, [-df_x]): x\in \Supp(F)\cap f^{-1}(a-\ep_1, b)\}\subset T^\infty X\\
&S:=\SS(F)^\infty\cap T^\infty X|_{f^{-1}(a-\ep_1,b)}. 
\end{align*}
Then under the composition $T^\infty X\overset{\pi}{\to} X\overset{f}{\to} \bR$, we have both $W_f\to (a-\ep_1, b)$ and $S\to (a-\ep_1, b)$ are proper, so is $W_f\cap S\to (a-\ep_1, b)$. In particular, $f( \pi(W_f\cap S))\cap  (a-\ep_1, b)$ is closed in $ (a-\ep_1, b)$. Since by assumption $f( \pi(W_f\cap S))\cap  [a, b)=\varnothing$, $f( \pi(W_f\cap S))$ is disjoint from $(a-\epsilon, b)$ for some $0<\epsilon\leq \epsilon_1$. Hence the observation follows. 

Now using (\ref{eq: ii, microlocal Morse lemma}), the morphism $\varphi'$ factors as the isomorphisms
\begin{align*}
 \Gamma(f^{-1}[a, d];F)\simeq \varinjlim_{a'\in (a-\ep, a)}\Gamma(f^{-1}(a', d];F)\overset{\sim}{\to} \Gamma(f^{-1}[b, d];F). 
\end{align*}
The proof is complete. 
\end{proof}

\subsection{The sheaf of microlocal sheaves}
\label{subsec:TSOMS}

This is another ``optional'' section, describing an approach to constructing sheaves of brane structures, or more general microlocal categories, that we will not work out in detail.

For each conic subset $U \subset T^* M$, Kashiwara and Schapira introduce a category of ``microlocal sheaves on $U$,'' called $D^b(M,U)$.  It is defined to be the Verdier quotient of the bounded derived category of sheaves on $M$ by the subcategory of sheaves with singular support outside of $U$.  
When $U$ is open\footnote{If $U$ is not open, the class of sheaves with singular support in the complement of $U$ is not closed under infinite direct sums} we make a similar definition, using $\infty$-categorical localization \S\ref{subsec:localization} in place of the Verdier quotient.  More precisely we define $\Shall(M,U)$ to be the right orthogonal to the full subcategory of sheaves whose singular support is in the complement of $U$.
\subsubsection{The presheaf of microlocal sheaves}
The assignment $U \mapsto \Shall(M,U)$ from conic open subsets of $T^* M$ to full subcategories of $\Shall(M)$ is inclusion-preserving. Each of the inclusions $\Shall(M,U) \to \Shall(M,V)$ has a continuous left adjoint, and by \cite[Cor. 5.5.3.4]{higher-topoi}, the left adjoints to these left adjoints assemble to a contravariant functor from the poset of conic open subsets of $T^* M$ to $\St_{\bk}$, i.e. to a presheaf on $T^* M$ that as in \cite[\S 3.1]{TZ} we denote by $\MSh^{\mathrm{p}}$.

The equality $\SS(F) \cap U = \SS(F') \cap U$ holds whenever $F$ and $F'$ become isomorphisms in $\Shall(M,U)$ --- this follows from the triangle inequalities.  Given a closed conic subset $Z \subset T^* M$, we define $\MSh^{\mathrm{p}}_Z$ to be the presheaf on $T^* M$ by
\[
\Gamma(U,\MSh^{\mathrm{p}}_Z) = \{F \in \Shall(M,U) \mid \SS(F) \cap U \subset Z\}
\]
This presheaf is supported on $Z$ in the sense that its sections over any open set in the complement of $Z$ give the zero category.

\subsubsection{Sheafification of the presheaf of microlocal sheaves}\label{subsubsec: MSh_Z}
We write $\MSh_Z$ for the sheafification of $\MSh^{\mathrm{p}}_Z$ with values in $\St_{\bk}$.  Let us first make a warning about sheafification, that makes it difficult to identify $\MSh_Z$ in general.
The $\infty$-category $\St_{\bk}$ is not compactly generated.  Because of this, it is not possible to tell whether a morphism between sheaves of categories is an equivalence, by checking that it is an equivalence on stalks. On the other hand, one could sheafify $\MSh^{\mathrm{p}}_Z$ in $\widehat{\Cat}_\infty$ (notation as in Appendix \ref{appendix B, descent}) which has a stalkwise criterion for equivalences. We refer the reader to \cite[Remark 6.1]{Nadler-Shende} for more comments on this. 

For example as $U$ runs through the open subsets of the real line $\bR$, the $\infty$-categories $\Shall(U)$ and their full subcategories $\Loc(U)$ assemble to sheaves of presentable stable $\infty$-categories, let us denote them by $\Loc$ and $\Shall$.  It follows from \cite[Thm. 5.5.3.18]{higher-topoi} that the inclusion functor $\Loc \to \Shall$ induces an equivalence of stalks.

However when $Z$ is a conic Lagrangian contained in the conormal variety of some Whitney stratification, $\MSh^{\mathrm{p}}_Z$ has the following constructibility property:
\begin{itemize}
\item the presheaf is constructible along $Z$ in the sense that every point has a fundamental system of neighborhoods $z \in Z$ for which the restriction maps $\Gamma(U_i,\MSh^{\mathrm{p}}_Z) \to \Gamma(U_j,\MSh^{\mathrm{p}}_Z)$ are equivalences.  
\end{itemize}
This property is inherited by $\MSh_Z$.  When $\cC$ and $\cC'$ are sheaves of $\infty$-categories obeying such a constructibility property, we have a stalkwise criterion for equivalences.  For example it follows that $\MSh_Z$ is locally constant over the complement of some codimension $1$ set in $Z$.

\subsection{Constructible sheaves}
\label{subsec:constructible-sheaves}
Write $\Loc(X,\bk) \subset \Shall(X,\bk)$ for the full subcategory of locally constant sheaves.  If $X$ is a smooth manifold and $\cS$ is a Whitney stratification of $X$, we say that $F \in \Shall(X,\bk)$ is $\cS$-quasiconstructible if its restriction to each stratum is locally constant.  We say that $F$ is $\cS$-constructible if furthermore its stalk at each point is a perfect $\bk$-module (i.e. a compact object of $\Mod(\bk)$).  The full subcategory of $\Shall(X,\bk)$ spanned by the $\cS$-quasiconstructible sheaves is closed under small colimits --- in particular it is an object of $\St_{\bk}$ that we denote by $\Sh_{\cS}(X,\bk)$. 

We warn that it is not usually the case that $\Sh_{\cS}(X,\bk)$ is generated under colimits by the $\cS$-constructible sheaves --- this is not even true for locally constant sheaves. For example, if $X=S^1$, then $\Loc(X,\bC)\simeq \QCoh(\Spec\ \bC[t,t^{-1}])$. The finite rank local systems are corresponding to torsion coherent sheaves (i.e. those with proper support). Since $\Hom(M, \bC(t))=0$ for all torsion $M$, we see that $\bC(t)$ is not generated by finite rank local systems under taking colimits. However it is true when $\cS$ is a regular cell complex.

\subsubsection{Operations}
\label{subsec:constructible-operations}

Suppose $X$ and $Y$ are Whitney stratified manifolds with stratifications $\cS_X$ and $\cS_Y$. 
We say that $f:X \to Y$ is a stratified mapping if there is a factorization $X \to \overline{X} \to Y$, and a stratification of $\overline{X}$, such that $X \to \overline{X}$ is the inclusion of an open union of strata and $\overline{X} \to Y$ is proper and restricts to a submersion on each stratum. In general,  suppose $X'$ and $Y'$ are Whitney stratified manifolds and $f: X'\to Y'$ is a stratified map. Let $X\subset X'$ and $Y\subset Y'$ be the union of some of the strata in $\cS_{X'}$ and $\cS_{Y'}$, respectively, such that $f(X)\subset Y$. Then we say $f|_X: X\to Y$ a stratified mapping for the Whitney stratified spaces $X$ and $Y$. 
If $f$ is a stratified mapping then $f^*, f_*, f_!$ and $f^!$ preserve the subcategories of $\cS$-quasiconstructible sheaves. Moreover, we see that $f_*:\Sh_{\cS_X}(X;\bk) \to \Sh_{\cS_Y}(Y;\bk)$ is always continuous. 

Let's sketch a proof of the fact that the four functors preserve $\cS$-quasiconstructible sheaves. 
For $f$ open, this directly follows from the existence of good tubes \cite[\S 2.4]{Nadler-Tilting} and Thom's first and second isotopy Lemmas \cite[\S 11]{Mather}. In particular, one sees that $f_*$ is continuous on quasiconstructible sheaves. For $f$ proper, the assertion for $f_*=f_!$ follows again from Thom's isotopy Lemmas. 
For $f$ a topological submersion, the assertion for $f^*$ is obvious and the assertion for $f^!$ follows from (\ref{eq:submersion}). For $f$ a closed embedding, the assertion for $f^*$ and $f^!$ follows from the fiber sequences 
\begin{align*}
f_!f^!\to id\to j_*j^*,\quad j_!j^!\to id\to f_*f^*
\end{align*}
where $j$ is the open inclusion of the complement of $f(X)\subset Y$. In general, to prove $f^*$ and $f^!$ preserve $\cS$-quasiconstructible sheaves, one can do induction on the dimension of $Y$. By the above observations, WLOG, we may assume $f$ is proper. If $Y$ is a point, there is nothing to prove. Suppose we have proved the case for $\dim Y'<n$. Now let $\dim Y=n$ and let $j_n: Y_n\hookrightarrow Y$ be the open inclusion of the union of top dimensional strata in $\cS$ and $i_{<n}: Y_{<n}\hookrightarrow Y$ be the closed embedding of the complement. Let $j: f^{-1}(Y_{n})\hookrightarrow X$ and $i: f^{-1}(Y_{<n})\hookrightarrow X$ be the open and closed embeddings, respectively. 
We only sketch the proof for $f^!$, since the proof for $f^*$ is similar. We have the fiber sequence 
\begin{align*}
i_{!}i^!f^!\to f^!\to j_{*}j^!f^!
\end{align*} 
where (1) $i_!i^!f^!\cong  i_!(f i)^!$ preserves $\cS$-quasiconstructible sheaves from induction, for $f i$ factors through $i_{<n}$; (2) $j_*j^!f^!\cong j_*(f j)^!$  preserves $\cS$-quasiconstructible sheaves, since $f j$ factors through $j_n$, and on any contractible open $U$ in $Y_n$, we can write the local system as $p^!M_{\pt}$, $p: U\to \pt$ and $M\in \Mod(\bk)$, so then we reduce the case locally to the case that the base space is a point. Now the assertion for $f^!$ follows.

\subsubsection{Microlocal stalks}\label{subsubsec: microstalks}
If $X$ is a manifold and $\cS$ is a Whitney stratification of $X$, let $\Lambda_\cS \subset T^* X$ denote its conormal variety
\begin{equation}
\label{eq:conormal-variety}
\Lambda_{\cS} := \bigcup_{S_{\alpha} \in \cS} T^*_{S_{\alpha}} X
\end{equation}
Let $\Lambda_{\cS}^\sm \subset \Lambda_{\cS}$ be the smooth part.  For every $(x,\xi) \in \Lambda^\sm_{\cS}$, the microlocal stalk functor $\mu_{x,\xi}:\Sh_{\cS}(X) \to \Mod(\bk)$ is defined by
\begin{equation}
\label{eq:microstalk}
\mu_{x,\xi}(F) := \Sigma^d \mathrm{Cone}\left(\Gamma_c(B_{\epsilon}(x) \cap \upsilon^{-1}(\eta,\infty),F) \to \Gamma_c(B_{\epsilon}(x),F)\right)
\end{equation}
where $\upsilon$ is an $\cS$-stratified Morse function (\cite[\S 2.1]{stratified-morse-theory}) defined in a neighborhood $B_\epsilon(x)$ of $x$, with $d\upsilon_x = \xi$, $d$ is the index of $\upsilon$ at $x$, and $0 < \eta \ll \epsilon$ are sufficiently small (in the sense that $(\eta,\epsilon)$ belongs to a fringed set \cite[\S 5]{stratified-morse-theory}).  The ball $B_{\epsilon}(x)$ is defined with respect to a Riemannian metric on $X$.  The ``the stratified Morse theorem, part B'' of \cite{stratified-morse-theory} implies that this functor is well-defined up to isomorphism --- independent of the metric and of $\upsilon, \epsilon,\eta$.  For any positive real $r$ we have canonically $\mu_{x,r\xi} F \cong \mu_{x,\xi}$.  The functor $\mu_{x,\xi}$ is sensitive to the stratification only in the sense that if $\cS'$ refines $\cS$, the functor $\mu_{x,\xi}$ may not be defined on $\cS'$-constructible sheaves if $(x,\xi)$ is not a smooth point of $\Lambda_{\cS'}$ (equivalently, if the $\cS'$-stratum containing $x$ is of lower dimension than the $\cS$-stratum containing $x$).

\subsubsection{Singular support}
\label{subsubsec:sing-supp}
For constructible sheaves, the notion of singular support \S\ref{subsubsec:sing-supp-27} specializes to the following.  If $\Lambda \subset \Lambda_{\cS}$ is the closure of a union of components of $\Lambda_{\cS}^{\sm}$, we write $\Sh_{\Lambda}(X)$ for the full subcategory of  $\Sh_{\cS}(X)$ spanned by sheaves $F$ that obey the condition:
\begin{equation}
\label{eq:SS}
\mu_{x,\xi}(F) = 0 \text{ whenever }(x,\xi) \in \Lambda_{\cS} -\Lambda
\end{equation}
The category $\Sh_{\Lambda}(X)$ is independent of the stratification $\cS$ with $\Lambda_{\cS} \supset \Lambda$.  If $F \in \Sh_{\Lambda}(X)$, then we say that the singular support of $F$ is contained in $\Lambda$, and write $\SS(F) \subset \Lambda$.  As $\mu_{x,\xi}$ is continuous, $\Sh_{\Lambda}(X) \in \St_{\bk}$ is presentable.

Kashiwara and Schapira give a variety of bounds on the singular support of $F'$ in terms of the singular support of $F$, when $F'$ is obtained by applying a sheaf operation ($f_!,f^*,\otimes_{\bk},\cdots$) to $F$.  They are all easy to verify, even in the presentable setting, if one assumes that $F$ is quasiconstructible and that $f$ is a stratified mapping.

\subsection{Contact transformations}
\label{subsec:contact-transformations}

Let $\tilde{V}$ and $\tilde{W}$ be vector spaces of the same dimension.  Write $T^{*,\circ} \tilde{V}$ and $T^{*,\circ} \tilde{W}$ for the deleted cotangent bundles, i.e. the complements of the zero section in $T^* \tilde{V}$ and $T^* \tilde{W}$.  Let $\bL$ denote the germ of a smooth conic Lagrangian near the point $(0,p_0) \in T^{*,\circ} \tilde{V}$.  A contact transformation is a kind of Darboux coordinate change near $(0,p_0)$, with additional requirements imposed to make it induce functors on sheaf categories associated to $\bL$.

\subsubsection{Definition}
\label{subsec:def-contact-transf}
A contact transformation\footnote{In \cite{KS}, a contact transformation only needs to satisfy condition (1). Condition (2) specifies a generic class of contact transformations. Since we always use this generic class of contact transformations, we simply refer them as contact transformations.} of $\bL$ is the germ of a conic Lagrangian $\chi \subset T^{*,\circ} \tilde{V}  \times T^{*,\circ} \tilde{W}$ through a point $((0,-p_0),(0,p'))$ that obeys the following:
\begin{enumerate}
\item $\chi$ is the germ of the graph of a symplectomorphism $U_0 \stackrel{\sim}{\to} U'$, where $U_0 \subset T^* \tilde{V}$ is a neighborhood of $(0,p_0)$ and $U' \subset T^* \tilde{W}$ is a neighborhood of $(0,p')$.  (The ``graph'' is modified to be Lagrangian, by applying $(x,\xi) \mapsto (x,-\xi)$ in the $T^* \tilde{V}$ factor.)
\item Write $\chi(\bL)$ for the image of $\bL \cap U_0$ under such a symplectomorphism.  Then $\chi$ is the germ of an open subset of the conormal bundle to a smooth hypersurface in $\tilde{V} \times \tilde{W}$ passing through the origin, and $\chi(\bL)$ is the germ of an open subset of the conormal bundle to a smooth hypersurface in $\tilde{W}$ passing through the origin.
\end{enumerate}
It is proved in \cite[Prop. A.2.5, Cor. A.2.7]{KS} that for every $\bL$ one can find such a $\chi$.  

\subsubsection{Effect of a contact transformation on categories} 
If $\tilde{V}_0 \subset \tilde{V}$ and $\tilde{W}_0 \subset \tilde{W}$ are sufficiently small neighborhoods of the origins, by a slight abuse of the notion of ``germ'' we may regard $\bL$ as a conic Lagrangian in $T^* \tilde{V}_0$ whose front projection to $\tilde{V}_0$ is closed (and recall the assumption (\ref{eq:finite-to-one})) and $\chi(\bL)$ as a conic Lagrangian in $T^* \tilde{W}_0$.  
The theory of contact transformations in \cite[\S 7.2]{KS} gives an equivalence between the localizations (recall that $\zeta_M$ denotes the zero-section in $T^*M$ for any manifold $M$)
\begin{equation}
\label{eq:L-chiL}
\Sh_{\bL \cup \zeta_{\tilde{V}_0}} (\tilde{V}_0;\bk)/\Loc(\tilde{V}_0) \simeq \Sh_{\chi(\bL) \cup \zeta_{\tilde{W}_0}}(\tilde{W}_0;\bk)/\Loc(\tilde{W}_0)
\end{equation}
The notation $\Sh_{\Lambda}$ is as in \S\ref{subsubsec:sing-supp} and the notation $C/C'$ is as in \S\ref{subsec:localization}.  This is proved in Appendix \ref{Appendix: contact} \S\ref{subsec: Appendix Proof of eq:L-chiL}.

The equivalence \eqref{eq:L-chiL} is described by a kernel $K \in \Sh(\tilde{V}_0 \times \tilde{W}_0)$ \S\ref{subsec:kernel}.  If $H \subset \tilde{V}_0 \times \tilde{W}_0$ is a hypersurface as in \S\ref{subsec:def-contact-transf}(2), then putting $K$ to be the constant sheaf on $H$ (extended by zero), one computes that $K^{-1}$ \eqref{eq:K-inverse} is also a suspension of constant sheaf on $\mathrm{flip}(H) \subset \tilde{V}_0 \times \tilde{W}_0$.

\subsubsection{Contact transformations and microlocal stalks}
\label{subsubsec:ctams}
Since $\chi(\bL)$ is the conormal of a smooth hypersurface in $\tilde{W}_0$, the category on the right-hand side of \eqref{eq:L-chiL} is easy to describe: the functor $\mu_{(0,p')}$ \eqref{eq:microstalk} gives an equivalence to $\Mod(\bk)$.  In fact the numerator of the right-hand side is described in the introduction \S\ref{intro:smooth-part}; in those terms the microlocal stalk functor is the cone on the map \eqref{eq:smooth-part}.

The group of origin-preserving diffeomorphisms of $\tilde{W}$ acts on the set of contact transformations for $\bL,(0,p_0)$ in the obvious way, with each diffeomorphism moving $(0,p')$ to $(0,p'')$ for another nonzero $p''$.  If $\chi_1$ and $\chi_2$ differ by such a diffeomorphism, the composite functors $\Sh_{\bL \cup \zeta_{\tilde{V}_0}} (\tilde{V}_0;\bk)/\Loc(\tilde{V}_0) \cong \Mod(\bk)$ are canonically isomorphic in a strong sense, so that there is a natural map from the space of diffeomorphism-classes of contact transformations and the space of equivalences to $\Mod(\bk)$.

\subsection{Guillermou-Kashiwara-Schapira}
\label{subsec:GKS}
Let $M$ be a manifold, let $I \subset \bR$ be an open set containing $0$ and let $\varphi_t, t \in I$ be a one-parameter family of symplectomorphisms of $T^{*,\circ} M$ obeying
\[
\varphi_0(x,\xi) = (x,\xi) \qquad \forall t\in I, r \in \bR_{>0}, \, \varphi_t(x,r\xi) = r\varphi_t(x,\xi)
\]
In other words, $\varphi_t$ is a homogeneous Hamiltonian isotopy of $T^{*,\circ} M$.  Suppose that $\varphi$ is horizontally compactly supported, in the sense that there is a compact open subset $A \subset M$ such that $\varphi_t$ is the identity outside of $T^* A$ for all $t$.
The main theorem in \cite{GKS} is the uniqueness and existence of a kernel $K_t$ (\S\ref{subsec:kernel}) on $M \times M$ such that
\begin{enumerate}
\item $K_t \circ$ and $K_t^{-1} \circ$ are inverse equivalences on $\Shall(M,\bk)$, notation as in \eqref{eq:K-inverse}.
\item Away from the zero section one has $\SS(K_t \circ F) = \varphi_t(\SS(F))$ for all $F\in \Shall(M,\bk)$. 
\end{enumerate} 
They furthermore prove that each $K_t$ is locally bounded.  If we make an additional tameness hypothesis on $\varphi$, the $K_t$ are constructible on some Whitney stratification of $M \times M$.  With this tameness hypothesis in place, essentially the same proof works in the presentable setting, giving constructible $K_t$ whose fibers are perfect $\bk$-modules --- this perfectness property replacing the locally bounded property, when $\bk$ is a ring spectrum.

We indicate a few details.  The graphs $G_t$ of $\varphi_t$ determine a conic Lagrangian subset $G \subset T^{*,\circ} (M \times M \times I)$, so that $G_t$ is the projection to $T^*(M \times M)$ of the slice at $t \in I$.  We seek a sheaf $K$ with singular support in $G \cup \zeta_{M \times M \times I}$ whose restriction to $M \times M \times \{0\}$ is the constant sheaf supported on the diagonal.  The uniqueness of $K$ can be proved exactly as in \cite[Prop. 3.2]{GKS}.  

For the existence of $K$, we impose the additional tameness assumption that $\varphi$ is contained in the conormal variety \eqref{eq:conormal-variety} of a Whitney stratification of $M \times M \times I$ --- for example this holds if $\varphi$ is the Hamiltonian flow of a subanalytic function.  Without further loss of generality we can assume that the projection $M \times M \times I \to I$ is a stratified Morse function, so that to construct $K$, it suffices to construct it locally around each critical value of $M \times M \times I \to I$ and argue that they can be glued.  Steps (A) and (B) of the proof of \cite[Prop. 3.5]{GKS} give the local construction and the gluing argument. 

Another consequence of the existence of $K$, is that for each $t \in I$ and each conic Lagrangian $\Lambda \subset T^{*,\circ} M$, the restriction functor
\begin{equation}
\label{eq:312}
\Sh_{G \circ \Lambda}(M \times I) \to \Sh_{G_t \circ \Lambda}(M \times \{t\})
\end{equation}
is an equivalence --- the proof is the same as \cite[Prop. 3.12]{GKS}. It also follows from the uniqueness of $K$ that the GKS construction induces a monoidal functor 
\begin{align*}
\text{Path}^{sm,\delta}_*\text{Ham}(M)\longrightarrow \Shall(M\times M\times I),\ \varphi\mapsto K, 
\end{align*}
where  $\text{Path}^{sm,\delta}_*\text{Ham}(M)$ is the space of smooth Hamiltonian isotopies $\varphi=\{\varphi_t, t\in I\}, \varphi_0=id$, equipped with the discrete topology. The space $\text{Path}^{sm,\delta}_*\text{Ham}(M)$ has an obvious group structure $(\varphi'\circ\varphi)_t=\varphi'_t\circ\varphi_t$, and $\Shall(M\times M\times I)$ has the obvious convolution monoidal structure coming from $M\times M$ (i.e. it is doing $K'_t\circ K_t$), and the monoidal functor is with respect to these monoidal structures. 

\section{Wavefronts, brane structures, and Nadler-Zaslow}
\label{sec:wavefront-notation}
We will mostly apply the ideas of \S \ref{sec:som} in a somewhat special situation, for which collect notation here.

\subsection{Cotangent bundles}
\label{subsec:cotangent-bundles}
If $M$ is a manifold, its cotangent bundle $T^* M$ is an symplectic symplectic manifold: if $x = (x_1,\ldots,x_n)$ are coordiantes on $M$ and $\xi = (\xi_1,\ldots,\xi_n)$ are the dual cotangent coordinates, the canonical one-form and symplectic form are 
\[
\alpha := \xi_1 dx_1 + \cdots + \xi_n  dx_n \qquad \omega = d\alpha
\]
We let $\overline{T}^* M := T^* M \, \amalg\,  T^{\infty} M$ be the relative compactification of $T^* M$ given in \cite[\S 5.1.1]{NZ}.  The boundary $T^{\infty} M$ of $\overline{T}^* M$ is called ``contact infinity''

\subsection{Exact Lagrangians}
\label{subsec:exact-lags}
An immersion $i:L \to T^* M$ is Lagrangian if $i^* \omega$ vanishes.  It is \emph{exact} if $i^* \alpha$ is an exact one-form.  A function $f:L \to \bR$ for which $df = \alpha$ is called a primitive for $L$.  If $L$ is connected then the primitive is unique up to an additive constant. 
Throughout the paper, we assume that the image of the projection $L\rightarrow M$ is precompact.   
 We write $\Lambda$ for the intersection $\overline{L} \cap T^{\infty} M$. It is not always necessary that $\Lambda$ is smooth, but we will assume that $L$ and $\overline{L}$ are tame in the sense of \cite[\S 5.2]{NZ}, which in particular implies that $\Lambda$ is Legendrian wherever it is smooth.  We let $\Phi \subset M$ denote the image of $\Lambda$ under the front projection.

We will say that an exact Lagrangian is \emph{lower exact} if any primitive for it is proper and bounded above.  The tameness condition implies furthermore that the primitive has only finitely many critical values.
The standard Lagrangians of \cite{NZ}, when taken over an open subset of $M$, are always lower exact, as $\log(m) \to -\infty$ at the boundary of the open set.   The costandard Lagrangians are not lower exact, however one can always find a Hamiltonian equivalent lower exact Lagrangian.  For example the figure below depicts a costandard Lagrangian in $T^* \bR$ on the left, and a Hamiltonian equivalent lower exact Lagrangian on the right.
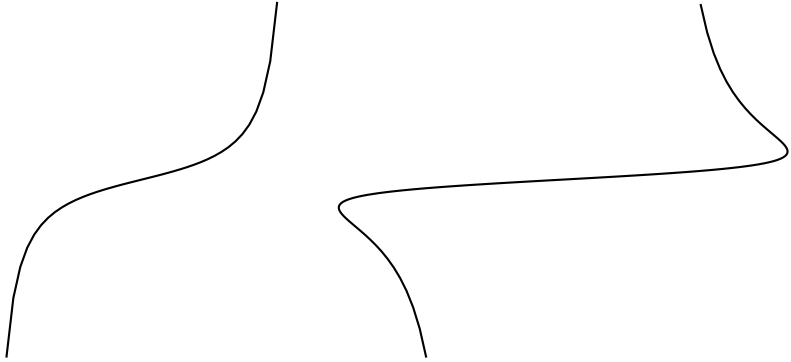
\begin{center}
\begin{tikzpicture}
\draw [thick,  domain=-1.8:1.8, samples=40] 
 plot ({\x}, {\x/(4-\x*\x)} );
\end{tikzpicture}
\qquad
\begin{tikzpicture}
\draw [thick, domain = -1.0:1.0,samples = 100]
plot ({0.16*\x*\x*\x*\x*\x*\x*\x*\x*\x-1.19*\x*\x*\x*\x*\x*\x*\x+5.09*\x*\x*\x*\x*\x-10.32*\x*\x*\x+8.08*\x},{(1.8*\x)/(4-(1.8*\x)*(1.8*\x))});
\end{tikzpicture}

\end{center}

\subsection{Legendrian lift and wavefront projection}
\label{subsec:bold-notation}
We equip $(T^* M) \times \bR$ with the one-form $dt - \alpha$, where $t$ denotes the projection to $\bR$.  (This is sometimes denoted $J^1 M$, the one-jet space of $M$.)  The projection $(T^* M) \times \bR \to T^*M$ is called the \emph{Lagrangian projection}, as it carries generic Legendrian subanifolds to immersed exact Lagrangians.  The projection $(T^* M) \times \bR \to M \times \bR$ is called the \emph{wavefront projection}, and it carries generic Legendrian submanifolds to immersed hypersurfaces in $M \times \bR$ with wavefront singularities.  We will only need to impose \eqref{eq:finite-to-one}, a weaker condition than genericity on $L$, that the wavefront projection is finite-to-one and is one-to-one over the smooth locus of the wavefront projection.

We often use the following notation.  If $L$ is connected and $L \hookrightarrow T^*M$ is an exact Lagrangian immersion, we let $\bL \subset (T^* M) \times \bR$ denote the Legendrian submanifold that lifts $L$, and $\bF \subset M \times \bR$ the image of $\bL$ under $\XZ$.  These subsets are well-defined up to translation in the $\bR$-coordinate.

If $\bL$ obeys \eqref{eq:finite-to-one}, then we may recover $\bL$ from $\bF$ by writing a branch of $\bF$ as the graph of a smooth function $t(x_1,\ldots,x_n)$ defined over an open subset of $M$, and putting $\xi$ to be the derivative of $t$.  The meaning of ``wavefront singularities'' is that the closure of the set where this recipe is well-defined is smooth, and we recover $\bL$ as this closure.

Our tameness hypotheses along with lower exactness \S\ref{subsec:exact-lags} imply that the closure $\overline{\bF}$ of $\bF$ in $M \times \leftbracket -\infty,\infty \rightparen$ is $\bF \amalg \Phi \times \{-\infty\}$.  It furthermore implies that $M \times \leftbracket -\infty,\infty \rightparen$ has a Whitney stratification refining the decomposition by 
\begin{equation}
\label{eq:refine-this}
\bF,\, M \times \bR - \bF,\,  \Phi \times \{-\infty\},\, \text{ and }M \times \{-\infty\}
\end{equation}
If $\Lambda$ is smooth, then for $K \ll 0$ we have a stratification-preserving homeomorphism of $\overline{\bF}$ with $\Phi \times \leftbracket -\infty,K \rightparen$.

\subsection{Example}
\label{example:wavefront}
For example, the wavefront projection of the right-hand Lagrangian  $\bL$ of the figure of \S\ref{subsec:exact-lags} looks like this:

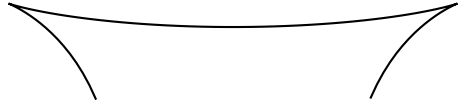
\begin{center}
\begin{tikzpicture}
\draw [thick, domain = -1.0:1.0,samples = 100]
plot ({0.16*\x*\x*\x*\x*\x*\x*\x*\x*\x-1.19*\x*\x*\x*\x*\x*\x*\x+5.09*\x*\x*\x*\x*\x-10.32*\x*\x*\x+8.08*\x},{-0.102*\x*\x*\x*\x*\x*\x*\x*\x + 0.606*\x*\x*\x*\x*\x*\x -2.411*\x*\x*\x*\x + 2.64*\x*\x +1.023*ln(100-81*\x*\x)});
\end{tikzpicture}
\end{center}
In the next figure we indicate a sheaf on $\bR^2$ with singular support in $\bL$ that is free of rank one on the under-side of the wavefront and zero on the other side --- the solid line indicates ``standard'' boundary conditions and the dotted lines indicate ``costandard'' boundary conditions, in the sense of \cite[\S 4.1]{NZ}.

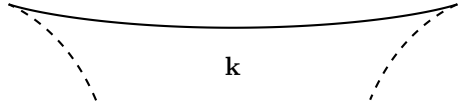
\begin{center}
\begin{tikzpicture}
\draw [thick, dashed, domain = -1.0:-0.5928782718649503,samples = 100]
plot ({0.16*\x*\x*\x*\x*\x*\x*\x*\x*\x-1.19*\x*\x*\x*\x*\x*\x*\x+5.09*\x*\x*\x*\x*\x-10.32*\x*\x*\x+8.08*\x},{-0.102*\x*\x*\x*\x*\x*\x*\x*\x + 0.606*\x*\x*\x*\x*\x*\x -2.411*\x*\x*\x*\x + 2.64*\x*\x +1.023*ln(100-81*\x*\x)});
\draw [thick, domain = -0.5928782718649503:0.5928782718649503,samples = 100]
plot ({0.16*\x*\x*\x*\x*\x*\x*\x*\x*\x-1.19*\x*\x*\x*\x*\x*\x*\x+5.09*\x*\x*\x*\x*\x-10.32*\x*\x*\x+8.08*\x},{-0.102*\x*\x*\x*\x*\x*\x*\x*\x + 0.606*\x*\x*\x*\x*\x*\x -2.411*\x*\x*\x*\x + 2.64*\x*\x +1.023*ln(100-81*\x*\x)});
\draw [thick,  dashed, domain = 0.5928782718649503:1,samples = 100]
plot ({0.16*\x*\x*\x*\x*\x*\x*\x*\x*\x-1.19*\x*\x*\x*\x*\x*\x*\x+5.09*\x*\x*\x*\x*\x-10.32*\x*\x*\x+8.08*\x},{-0.102*\x*\x*\x*\x*\x*\x*\x*\x + 0.606*\x*\x*\x*\x*\x*\x -2.411*\x*\x*\x*\x + 2.64*\x*\x +1.023*ln(100-81*\x*\x)});
\node at (0,4.2) {$\mathbf{k}$};
\end{tikzpicture}
\end{center}

\subsection{$\bL$ as a conic Lagrangian}  
\label{subsec:conic-bL}

Let $L$ be as in \S\ref{subsec:exact-lags} and $\bL,\bF$ as in \S\ref{subsec:bold-notation}.  If we write $(x,t,\xi,\tau)$ for coordinates on $T^* (M \times \bR)$ (where $(x,\xi) \in T^* M$ and $(t,\tau) \in T^* \bR$, in the notation of \S\ref{subsec:cotangent-bundles}) then it is natural to regard $T^* M \times \bR$ as the contact hypersurface $\{\tau = -1\}$ in $T^* (M \times \bR)$, and $\bL$ determines a conic Lagrangian
$
\Cone(\bL) \subset T^*(M \times \bR)
$
which we define as
\begin{equation}
\label{eq:cone-bL}
\Cone(\bL) = \zeta_{M \times \bR} \cup \left\{(x,t;\xi,\tau) \mid \tau < 0 \text{ and }(x,-\xi/\tau,t) \in \bL\right\}.
\end{equation}
We will write $\Sh_{\bL}(M \times \bR)$ in place of $\Sh_{\Cone(\bL)}(M \times \bR)$.  
If $U$ is an open subset of $M \times \bR$ and $\Omega$ is a connected component of $\XZ^{-1}(U) \cap \bL$, we similarly define $\Cone(\Omega) \subset T^* U$, and write $\Sh_{\Omega}(U)$ in place of $\Sh_{\Cone(\Omega)}(U)$. Since there is a canonical identification between $L$ and $\bL$, in the following, by some abuse of notations, we will freely identify ``things" on $L$ with ``things" on $\bL$. For example, for any open subset $\Omega\subset L$ (resp. $\Omega\subset \bL$), we equally regard it as an open subset in $L$ and $\bL$. When one comes to talk about sheaves on $M\times\bR$ and singular support conditions, one always deal with the Legendrian $\bL$.

Let $\bF^\sm$ be the smooth portion of $\bF \subset M \times \bR$, and let $\bL^\sm$ be the preimage under the front projection $\XZ$ --- thus, $\XZ|_{\bL^{\sm}}:\bL^{\sm} \stackrel{\sim}{\to} \bF^{\sm}$.  Each point $p \in \bL^{\sm}$ determines a microlocal stalk functor \eqref{eq:microstalk}
\begin{equation}
\label{eq:mup}
\mu_p:\Sh_{\bL}(M \times \bR) \to \Mod(\bk),
\end{equation}
as we may find a stratification $\cS$ of $M \times \bR$ such that the connected components of $\bF^{\sm}$ are the codimension-one strata and the connected components of $M \times \bR - \bF$ are the open strata.  If $p,q$ are two points in the same connected component of $\bL^{\sm}$, then $\mu_p \cong \mu_q$.

By compactifying this stratification refining \eqref{eq:refine-this}, one sees that the restriction of the functor $\proj_{1,*}:\Shall(M \times \bR) \to \Shall(M)$ to $\Sh_{\bL}(M \times \bR)$ is continuous by \S\ref{subsec:constructible-operations}.

\subsection{Lower exact Lagrangian and eventually conic Lagrangians}\label{subsec:eventually-conic}
In this section, we show that there is a close relation between lower exact Lagrangians (with a smooth Legendrian boundary) and eventually conical Lagrangians. As an application, this perspective gives an alternative way of deforming products of eventually conic Lagrangians to that in \cite{GPS}. 

Fix a Riemannian metric on $M$, and let $ST^* M \subset T^* M$ be the unit sphere bundle with its contact $1$-form $\alpha_1$. Let $R_{\alpha_1}$ denote the Reeb vector field. Let $(T^* M)_{|\xi| \leq 1} \subset T^* M$ be the unit disk bundle.  It is a Weinstein domain (in the sense of e.g. \cite[Def. 11.10]{CielebakEliashberg}). If $L$ is a lower exact Lagrangian with a primitive $f: L\rightarrow \bR$ and $\Lambda=L^\infty$ is smooth (i.e. $\overline{L}$ is a smooth manifold with a smooth boundary), there exists $K_0\ll 0$ such that each level set $\{f=K\}, K\leq K_0$, gives a smooth Legendrian submanifold $\Lambda_{f=K}$ in $q^\circ: T^{*,\circ}M/\bR_{>0}\cong  ST^*M$ via the obvious quotient map $T^{*,\circ}M\longrightarrow T^{*,\circ}M/\bR_{>0}$, where $T^{*,\circ}M:=T^*M-\zeta_M$. In the following, for a time dependent (contact) Hamiltonian $H$ with Hamiltonian vector field $X_H$, let $\varphi_H^{[a,s]}$ (or $\varphi_{H_s}^{s-a}$) denote the flow satisfying 
\begin{align*}
\frac{d}{ds}\varphi_H^{[a,s]}=X_{H_s}(\varphi_H^{[a,s]}),\ \varphi_{H}^{[a,a]}=id. 
\end{align*}
If $a=0$, then $\varphi_H^{[a,s]}$ is also denoted by $\varphi_H^{s}$. Then for $K_0$ sufficiently negative, the family of Legendrians $\{\Lambda_{f=K}\}_{-\infty\leq K\leq K_0}$ starting from $\Lambda_{f=-\infty}:=\Lambda$ is a smooth Legendrian isotopy $\{\Lambda_s\}_{0\leq s\leq 1}$ with $\Lambda_0=\Lambda$, after a renormalization of the interval $\eta: [-\infty, K_0]\overset{\sim}{\longrightarrow}[0,1]$ (one may use this condition on $L$ in place of the condition that $\Lambda=L^\infty$ is smooth and $L$ is tame).

\begin{lemma*}\label{Lag ends}
\begin{itemize}
\item[(a)] Under the above assumptions, there is a  time dependent (strictly) positive contact Hamiltonian function $H_L^s, s\in (0,1]$ (cf. \cite[2.3]{Geiges}), with $H_L^0:=\lim\limits_{s\rightarrow0^+}H_L^s=0$, such that $\varphi_{H_L}^{[a, b]}(\Lambda_a)=\Lambda_b$ for any $0<a\leq b\leq 1$. 

\item[(b)] There is a fringed set $I\subset \bR^2_+$ (cf. \cite[Section 5.2]{NZ}), such that for any $(s,t)\in I$, we have 
\begin{equation*}
\Lambda_s\cap \Lambda_{s+t}=\varnothing. 
\end{equation*}
In particular, for any $\epsilon>0$ there exists $\delta_0>0$ such that $\Lambda_{\epsilon+s}\cap \Lambda_{\epsilon+s+t}=\varnothing$ for any $(s,t)\in [0, \delta_0)\times (0,\delta_0)$. 
\end{itemize}
\end{lemma*}
\begin{proof}
(a) 
For any $(x,\xi)\in \{f=K\}$, let $(\bar{x}, \bar{\xi})\in \Lambda_{f=K}$ denote the corresponding point in $\Lambda_{f=K}\subset ST^*M$. The tangent space $T_{(x, \xi)}\{f=K\}$ can be identified with the graph of an element $b_{(x,\xi)}\in \Hom(T_{(\bar{x}, \bar{\xi})}\Lambda_{f=K}, \bR)$, where we identify $\bR\cong \bR\cdot Z_{(x,\xi)}$ with $Z$ the Liouville vector field. Note that the function $b_{(x,\xi)}$ is invariant under the dilation action by $Z$. i.e. if we dilate $L$ with $(x,\xi)$ moving to $(x,\xi')$, then the resulting function (using $Z_{(x, \xi')}$ instead) is the same. 
The gradient vector field $X_f$ of $f$ at $(x,\xi)\in \{f=K\}$, with respect to the induced metric on $L$, has the following decomposition 
\begin{align}
\nonumber&X_f=\lambda Z_{(x,\xi)}+\mu(R_{\alpha_1}+w'),\lambda,\mu\in \bR, \text{ where }w'\in \ker\alpha_1\subset T_{(\bar{x}, \bar{\xi})}ST^*M\text{ satisfying:}\\
\nonumber&\omega_{(\bar{x},\bar{\xi})}(v+b_{(x,\xi)}(v)\cdot Z, R_{\alpha_1}+w')=0\text{ for all }v\in T_{(\bar{x}, \bar{\xi})}\Lambda_{f=K}\\
\label{eq: w', b_x,xi}\Leftrightarrow&\omega_{(\bar{x},\bar{\xi})}(v,w')-b_{(x,\xi)}(v)=0, \text{ for all }v\in T_{(\bar{x}, \bar{\xi})}\Lambda_{f=K}.
\end{align}
Note that $w'$ is uniquely determined modulo $T_{(\bar{x}, \bar{\xi})}\Lambda_{f=K}$. 
By the assumption that $f$ has only finitely many critical values, we have $\mu>0$ for $K\ll 0$ (in particular, if $X_f$ is normalized to be $\frac{X_f}{|X_f|^2}$ so that $df(\frac{X_f}{|X_f|^2})=1$, then $\mu=\frac{1}{|\xi|}$, for $\alpha|_{L}=df$). 
Therefore, under the projection $T^{*,\circ}X\longrightarrow T^{*,\circ}X/\bR_{>0}\simeq ST^*X$, the push-forward of $X_f$ is a positive multiple of  the Reeb vector field modulo $\ker\alpha$. This means $\{\Lambda_{f=K}\}_{K\leq K_0}$, as a Legendrian isotopy starting from $\Lambda_{f=K_0}$ is generated by the flow of a  time-dependent vector field $Y^s$ with $\alpha_1(Y^s)|_{\Lambda_s}=-\frac{1}{|\xi|}$, then part (a) of the lemma follows.

(b)  follows from a similar argument as in the proof of \cite[Lemma 5.2.5]{NZ}. 
\end{proof}

Conversely, starting from a positive isotopy of Legendrians in $ST^*M$, we can construct a ``suspension Lagrangian" in $T^{*,\circ}M$ as follows. 

\begin{prop*}
\item[(i)]
Let $\Lambda_0$ be a smooth closed Legendrian in $ST^*M$. Given any positive isotopy $\Lambda_s, 0\leq s\leq 1$, there exists an exact (smooth, embedded) suspension Lagrangian $L^\epsilon_{\Lambda}\subset T^{*,\circ}M$ of $\{\Lambda_s\}_{0\leq s<\epsilon}$, for $\epsilon>0$ sufficiently small,  such that the level sets of any primitive project to $\{\Lambda_s\}_{0\leq s<\epsilon}$ under $q^\circ$. 

\item[(ii)]
Every such suspension Lagrangian $L^\epsilon_{\Lambda}$ is uniquely determined by a positive smooth function $c(s)$ on $[0,\epsilon)$, in that the levels of (the preimage of) $\Lambda_s$ with respect to any primitive $f$, denoted by $f(s)$ by some abuse of notations, satisfies 
\begin{align*}
\frac{df(s)}{ds}=c(s).
\end{align*}
. 
\end{prop*}
\begin{proof}
First, we give a construction of $L_\Lambda^\epsilon$. Fix a time-dependent positive Hamiltonian $H$ that generates the positive isotopy. We may choose $H$ so that $dH_s(R_\alpha)=0$ along $\Lambda_s$. We make the identitications: 
\begin{align*}
\varrho: \Lambda_0\times [0,1]&\overset{\sim}{\longrightarrow} \bigcup\limits_{s\in [0,1]}\Lambda_s\times \{s\}\\
\nonumber (p, s)&\mapsto (\varphi_H^s(p), s);\\
T^{*,\circ}M&\cong ST^*M\times \bR_+ \\
\xi&\mapsto ([\xi], |\xi|).
\end{align*}
By the previous lemma part (b), $L_\Lambda^\epsilon$, if exists, is determined by 
\begin{align}
\label{eq: L_Lambda, eps}\Lambda_0\times [0,1]\cong \bigcup\limits_{s\in [0,1]}\Lambda_s\times \{s\}&\hookrightarrow ST^*M\times \bR_+\\
\nonumber (p, s)&\mapsto (\varphi_H^s(p), \ell(p,s)),
\end{align}
for a height function $\ell(p, s)$ (also denoted by $\ell(\varphi_H^s(p), s)$ by some abuse of notations), subject to the condition corresponding to (\ref{eq: w', b_x,xi}) that in $ST^*M$
\begin{align*}
\big(\frac{1}{H_s}\iota_{X_{H_s}}d\alpha_1-\frac{d_p\ell}{\ell}\big)|_{\Lambda_s}=0.
\end{align*}
This implies that 
\begin{align*}
\ell(p, s)=\frac{c(s)}{H_s(\varphi_H^s(p))}
\end{align*}
for some positive function $c(s)$. On the other hand, this is a sufficient condition for the embedding (\ref{eq: L_Lambda, eps}) to be a Lagrangian. 

Lastly, we have 
\begin{align*}
\frac{df(s)}{ds}=\ell(p,s)\alpha_1(X_{H_s})|_{\Lambda_s}=c(s). 
\end{align*}
\end{proof}

We say an exact Lagrangian is \emph{eventually conical} if outside a compact region it is conical. 

\begin{corollary*}\label{cor: normalizedReeb}
\item[(i)]
After a normalized Reeb perturbation, an eventually conical exact Lagrangian $L$ is a lower exact Lagrangian.

\item[(ii)] Conversely, for any lower exact Lagrangian $L$, assume that there exists $\delta>0$, such that $\Lambda_s\cap \Lambda_t=\varnothing$ for all $s\neq t\in (0, \delta)$, then there exists a Hamiltonian isotopy that takes it to be eventually conic\footnote{The assertion holds even if the limit Legendrian $\Lambda_0$ is not smooth or embedded, i.e. the isotopy $\{\Lambda_s\}_{s\in (0,1]}$ does not need to extend to an isotopy $\{\Lambda_s\}_{s\in [0,1]}$ }.
\end{corollary*}

\begin{proof}
(i)
Using the above proposition, let $\Lambda_0=q^\circ(L^\infty)$, and let $H_s=\delta s$ and $c(s)=\frac{\delta }{s}$ for $s\in (0,\epsilon)$ (then $\ell=\frac{1}{s^2}$), and $\delta>0$. This gives a suspension Lagrangian whose primitive goes to $-\infty$ at the infinity ends. Since this is doing the time-$\delta$ Hamiltonian flow of $2\sqrt{|\xi|}$ on a portion of $\Cone(\Lambda_0)$, we can extend the Hamiltonian to be supported in a neighborhood of $T^\infty M$, and the resulting flow moves $L$ to be lower exact.

(ii) We start with $\{\Lambda_s\}_{0<s\leq 1}$,  $H_L^s$ (and $\eta$) as in Lemma \ref{Lag ends}, which determines $c(s)$ as in the above proposition. Note that we have $\int_{0_+}^1 c(s)ds=\infty$. For $\epsilon\in (0,1]$, choose a family of $C^{r+1}$ subanalytic cut-off functions $b_\epsilon: [0,1]\rightarrow [0,1]$, continuously depending on $\epsilon$, such that $b_\epsilon|_{[0,\epsilon]}=\epsilon$, $b_\epsilon(s)=s$ for $s\in [4\epsilon, 1]$, and $b_\epsilon'(s)>0$ on $(\epsilon,1)$. Define 
$\Lambda^\epsilon_s:=\Lambda_{b_\epsilon(s)}$. Then $\{\Lambda_s^\epsilon\}_{\epsilon< s\leq 1}$ is a positive isotopy of $\Lambda_{\epsilon}$, and the positive contact Hamiltonian $H^s_{L,\epsilon}$ for $\epsilon< s\leq 1$ can be chosen to be $b_\epsilon'(s)H^{b_\epsilon(s)}_L$. 
Choose any $0<\epsilon\ll \delta\ll 1$ such that the maximum of the height function $\ell$ on $\{f\geq \eta^{-1}(\delta)\}$ is strictly less than the minimum of the height function on $\{f\leq \eta^{-1}(\epsilon)\}$. Let $\delta(\epsilon)$ be any number in $(0, \delta-\epsilon)$. 
Choose $\varrho: (0,1]\rightarrow\bR_+$ such that $\varrho|_{[\epsilon+\delta(\epsilon), 1]}=1$ and $\varrho|_{(0, \epsilon+\frac{1}{2}\delta(\epsilon)]}=\frac{s}{4\epsilon}$. 

Let $c_\epsilon(s)=\frac{1}{\varrho(s)}b'_\epsilon(s)c(b_\epsilon(s)), \epsilon<s\leq 1$. Consider the suspension Lagrangian $L^\circ_{\Lambda,\epsilon}$ associated to $\big(\{\Lambda_s^\epsilon\}_{\epsilon<s\leq 1}, c_\epsilon(s)\big)$. Let $f^\circ_\epsilon$ be the primitive on it so that $f^\circ_\epsilon(1)=K_0$. 
Since
\begin{align*}
L^\circ_{\Lambda,\epsilon}|_{\eta^{-1}(4\epsilon)\leq f^\circ_\epsilon\leq K_0}=L|_{\eta^{-1}(4\epsilon)\leq f\leq K_0},
\end{align*}
we can glue $L^\circ_{\Lambda,\epsilon}$ with $L|_{f\geq K_0}$ along the end given by $\Lambda_1^\epsilon=\Lambda_1$. Denote the resulting Lagrangian by $L^\circ_{\Lambda,\epsilon}$  as well. Near the end of $L^\circ_{\Lambda,\epsilon}$ given by $\Lambda_\epsilon^\epsilon=\Lambda_\epsilon$,  the Lagrangian $L^\circ_{\Lambda,\epsilon}$ is determined by the graph of the map
\begin{align}\label{eq: Lambda_eps}
\Lambda_\epsilon\times (\epsilon, 2\epsilon)\cong\bigcup\limits_{\epsilon<s<2\epsilon} \nonumber\Lambda_s^\epsilon\times\{s\}&\hookrightarrow ST^*M\times \bR_+\\
(p, s)&\mapsto (\varphi_{H_{L,\epsilon}}^{[\epsilon, s]}(p),\frac{c(b_\epsilon(s))}{\varrho(s) H_L^{b_\epsilon(s)}}=\frac{1}{\varrho(s)}\ell(\varphi_{H_L}^{[\epsilon, b_\epsilon(s)]}(p), b_\epsilon(s)))
\end{align}
where $\ell$ is the height function associated to $L$.

We make the zero extension of $H^s_{L,\epsilon}$ on $s\in [0, 1]$. We also make the extension of $\ell_\epsilon(p, s):=\ell(\varphi_{H_L}^{[\epsilon, b_\epsilon(s)]}(p), b_\epsilon(s)))$ originally defined on $\Lambda_\epsilon\times [\epsilon, 2\epsilon)$ to $\Lambda_\epsilon\times (0, 2\epsilon)$, given by 
\begin{align*}
\ell_\epsilon(p, s)=\ell_{\epsilon}(p, \epsilon), \text{ for }s\in (0, \epsilon).
\end{align*}
Since $\frac{\partial^k \ell_\epsilon(p,s)}{\partial s^k}|_{s=\epsilon}=0$ for all $1\leq k\leq r$, the extension is well defined.

Now we can make a $C^r$ subanalytic extension of (\ref{eq: Lambda_eps}) as 
\begin{align*}
\Lambda_\epsilon\times (0, 2\epsilon)\cong\big(\Lambda_\epsilon\times (0, \epsilon]\big)\cup \bigcup\limits_{\epsilon<s<2\epsilon} \nonumber\Lambda_s^\epsilon\times\{s\}&\hookrightarrow ST^*M\times \bR_+\\
(p, s)&\mapsto (\varphi_{H_{L,\epsilon}}^{[0, s]}(p), \frac{1}{\varrho(s)}\ell_\epsilon(p, s)).
\end{align*}  
This extends $L^\circ_{\Lambda,\epsilon}$ to be an eventually conic Lagrangian, denoted by $L_{\Lambda, \epsilon}$. By the assumption in the statement of (ii) and the relation between $\epsilon$ and $\delta$, we see the Lagrangian is embedded. Thus the family $\{L_{\Lambda, \epsilon}\}_{\epsilon>0}$ gives an eventually conic isotopy of $L$ as desired. 
\end{proof}

\begin{remark*}
The property of being eventually conic is not well behaved under taking products, but the property of being lower exact is. The above gives a way to deform the product of two eventually conic Lagrangians to be eventually conic, by first deforming the two Lagrangians to be lower exact and then deforming the product lower exact Lagrangian to be eventually conic. Indeed, it is not hard to see that for two lower exact Lagrangians $L_1\subset T^*M_1$ and $L_2\subset T^*M_2$ both satisfying the assumption in Corollary \ref{cor: normalizedReeb} (ii), their product in $T^*M_1\times T^*M_2$ also satisfies that assumption. 
The construction certainly generalizes to arbitrary Liouville manifolds, giving an alternative to that in \cite{GPS}. 

\end{remark*}

For later applications and for the main results, we make the following assumption on any lower exact Lagrangian $L$. The lower exact Lagrangians produced from eventually conic Lagrangians automatically satisfy the assumption. 

\begin{assumption}\label{assump: lel}
Under the setting of Lemma \ref{Lag ends} (a), we assume there exists $\delta>0$ such that for any $s\neq t\in [0,\delta)$, 
\begin{align*}
\Lambda_s\cap \Lambda_{t}=\varnothing. 
\end{align*}
\end{assumption}

We remark that Assumption \ref{assump: lel} will be used in Subsection \ref{subsec:total-Tam} and is needed for making the main results true (e.g. the Nadler--Zaslow functor and Theorem \ref{thm:312}). Here is a counterexample to the main results when Assumption \ref{assump: lel} is not satisfied\footnote{We thank Wenyuan Li for bringing this counterexample to our attention. In this counterexample, $L^\infty$ is not smooth. We expect that for a subanalytic smooth Legendrian isotopy $\{\Lambda_s\}_{s\in [0,1]}$ that is strictly positive for $s>0$, Assumption \ref{assump: lel} always holds, but we are not aware of a proof.}. Let $\Lambda'$ be any (connected closed) smooth \emph{loose} Legendrian in $T^*M_1\times\bR$, and let $M=M_1\times \bR$. Let $I$ be the minimal (closed) interval that contains the image of $\Lambda'$ to $\bR$. For any closed interval $I'$ (of positive length), there is a unique (orientation preserving) affine transformation of $\bR$ that sends $I$ to $I'$, which induces a diffeomorphism on $M_1\times \bR$. Let $\Lambda'_{I'}$ be the Legendrian in $T^*M_1\times\bR$ whose front projection is the image of $\XZ(\Lambda')$ under that diffeomorphism. 
Let $f_+: (-1,1)\rightarrow \bR$ be a smooth subanalytic function such that $\lim_{x\rightarrow\pm 1}f_-(x)=-\infty$ and let $\delta: (-1,1)\rightarrow \bR$ be a subanalytic function such that $\delta'(x)<0$ everywhere, $\lim_{x\rightarrow -1}\delta(x)=-1$ and $\lim_{x\rightarrow 1}\delta(x)=-5$. Let $f_-(x)=f_+(x)+\delta(x)$. 
Let $\Lambda$ be the suspension Legendrian of $\Lambda'_{[f_-(x), f_+(x)]}, x\in (-1,1)$, in $T^*(M_1\times (-1,1))\times \bR.$ Then $\Lambda$ is also a properly embedded Legendrian in $T^*M\times \bR$ (under the obvious embedding), whose Lagrangian projection in $T^*M$ is a smooth embedded lower exact Lagrangian. By the loose condition on $\Lambda'$, there is no interesting sheaf in $\Sh_\Lambda(M\times \bR)$ (except for the locally constant sheaves). On the other hand, it is clear that Assumption \ref{assump: lel} does not hold in this example.

\subsection{Brane structures, Tamarkin-Guillermou theory, and the Nadler-Zaslow functor}
\label{sec:TSOBS}
Let $L \to T^* M$ be a lower exact Lagrangian immersion with Legendrian lift $\bL \subset (T^* M) \times \bR$.  In the rest of this section we construct $\Brane_{\bL}$, a locally constant sheaf of stable $\infty$-categories on $\bL$, with fiber $\Mod(\bk) \in \St_{\bk}$.  It is a special case of the more general construction of \S\ref{subsec:TSOMS} which puts a sheaf of categories on any conic Lagrangian subset of $T^* M$.

\subsection{Good neighborhoods and good cylinders}
\label{subsec:good-cylinders}
We define a \emph{good neighborhood} in $M \times \bR$ (with respect to $\bL$) to be a contractible open subset $U \subset M \times \bR$ with a piecewise-smooth boundary, such that $\partial U$ is transverse to $\bF$, and the following condition holds for each of the finitely many connected components $\Omega \subset \XZ^{-1}(U)\cap \bL$:
\begin{quote}
there exists a point $p = (x,t;\xi) \in \Omega$, and a stratification-preserving flow of $U$ (for a stratification compatible with $\XZ(\Omega)$) that contracts $U$ to $x$.
\end{quote}
We say that a good neighborhood is a \emph{good cylinder} if it has the form $B \times I$ where $B$ is a contractible open subset of $M$, $I \subset \bR$ is an open interval, and $(B \times \partial I) \cap \bF$ is empty.  Our tameness assumptions on $L$ \S\ref{subsec:exact-lags} imply that every point of $M \times \bR$ is contained in a good neighborhood (resp. good cylinder) that can be taken arbitrarily small (see \S\ref{subsec:good-cyl-cover} for the proof).  We call the connected components $\Omega$ of $\XZ^{-1}(U)\cap \bL \subset \bL$ the ``branches'' of $\bL$ in the good neighborhood (resp. good cylinder).  For each branch $\Omega$, there is a local contact transformation on $T^* (B \times I) - \zeta_{B \times I}$ sending $\Omega$ to the conormal of a smooth hypersurface \ref{subsec:contact-transformations}, and we can deduce the following:

\begin{lemma*}[Branch lemma]
Let $U$ be a good neighborhood and let $\Omega$ be a branch of $\bL$ over $U$.  Let $p \in \Omega \cap \bL^{\mathit{sm}}$.  Then the sequence of presentable stable $\infty$-categories
\begin{equation}
\label{eq:branch-lemma}
\Loc(U) \to \Sh_{\Omega}(U) \xrightarrow{\mu_p} \Mod(\bk)
\end{equation}
where the first map is the full inclusion and the second map is the microlocal stalk
\eqref{eq:microstalk}, \eqref{eq:mup}, is a localization sequence \S\ref{subsec:localization}.  That is, $\Loc(U)$ is the kernel of $\mu_{p}$ and the right adjoint to $\mu_{p}$ is fully faithful.
\end{lemma*}

If we write $\bL\vert_U := \XZ^{-1}(U)\cap \bL$ for the union of all branches of $\bL$ over $U$, the square in the following diagram
\begin{equation}
\label{eq:branch-variant}
\xymatrix{
\Loc(U) \ar[r] \ar[d] & \Sh_{\Omega}(U) \ar[r]^{\mu_p} \ar[d] & \Mod(\bk) \\
\Sh_{\bL\vert_U - \Omega}(U) \ar[r] & \Sh_{\bL}(U) \ar[r]_{\mu_p} & \Mod(\bk)
}
\end{equation}
commutes --- each arrow in the square is a full inclusion.  Both of the rows are localization sequences, so that $\Sh_{\Omega}(U)/\Loc(U)$ is canonically isomorphic to $\Sh_{\bL}(U)/\Sh_{\bL\vert_U - \Omega}(U)$.

Let $\Nb(L)$ denote the set of ``branched neighborhoods'', i.e. the set of pairs $(U,\Omega)$ such that $U \subset M \times \bR$ is a good neighborhood and $\Omega$ is a branch of $\bL$ over $U$.  We endow $\Nb(U)$ with a partial order by putting $(U,\Omega) \leq (U',\Omega')$ if $U \subset U'$ and $\Omega = U \cap \Omega'$.  The poset $\Nb(L)$ indexes a complete cover of $\bL$ in the sense of \S\ref{subsec:covers}(3), by $(U,\Omega) \mapsto \Omega$.  It also projects finite-to-one onto the poset of open subsets of $M \times \bR$, by $(U,\Omega) \mapsto U$.

The same remarks hold for the poset of ``branched cylinders'', i.e. the set of triples $(B,I,\Omega)$ such that $B \times I$ is a good cylinder and $(B \times I,\Omega) \in \Nb(L)$.  We will denote this poset by $\Cylb(L)$.

\subsection{Multiple branches}
\label{subsec:multi-branch}
A good neighborhood $U$ meets the complement of $\bF$ in finitely many chambers.  If $U = B \times (a,b)$ is a good cylinder, we may speak of the ``bottom chamber'' that is incident with $B \times \{a\}$ and the ``top chamber'' that is incident with $B \times \{b\}$.  (It is possible for these to coincide.)  The image of the right adjoint of $\mu_p$ \eqref{eq:branch-lemma} (i.e. the right orthogonal to $\Loc(U)$) is the full subcategory of $\Sh_{\Omega}(U)$ spanned by sheaves that vanish in the bottom region of $U$.

Let $\Sh_{\Omega}^0(U)$ denote the full subcategory of sheaves that vanish in the \emph{top} region of $U$.  Although it is not the image of the right adjoint of $\mu_p$ \eqref{eq:branch-lemma} it is canonically equivalent to $\Sh_{\Omega}(U)/\Loc(U)$ by the composite $\Sh^0_{\Omega}(U) \subset \Sh_{\Omega}(U) \to \Sh_{\Omega}(U)/\Loc(U)$.  As $U$ is contractible there is a \emph{left} adjoint to the inclusion functor $\Loc(U) \subset \Sh_{\Omega}(U)$, which carries $F$ to the tensor product of $\omega_U$ and the constant sheaf with fiber $\Gamma_c(U;F)$ \eqref{eq:submersion}, and it vanishes on $\Sh_{\Omega}^0(U)$.  In other words
\begin{equation}
\label{eq:colocal}
\Sh_{\Omega}^0(U) \subset \Sh_{\Omega}(U) \to \Loc(U)
\end{equation}
is a localization sequence \S\ref{subsec:localization}.

\begin{prop*}
\item[(i)] 
Let $U$ be a good cylinder and let $\bL\vert_U = \Omega_1 \amalg \cdots \amalg \Omega_k$ be the branches of $\bL$ over $U$.  Then the functor
$\bigoplus_{i = 1}^k \Sh_{\Omega_i}^0(U) \to \Sh^0_{\bL\vert_U}(U)
$
induced by the inclusions $\Sh_{\Omega_i}^0(U) \to \Sh_{\bL\vert_U}^0(U)$ is an equivalence.

\item[(ii)] For $\eta>0$ sufficiently small relative to $B_\alpha \times I_{\alpha}$, the functor $T_{(-\eta, 0]}$ (introduced in Section \ref{subsec:tam-conv-2}) induces fully faithful embeddings
\begin{align}
\label{eq:3102}
\Sh^0_{\Omega_{\alpha}}(B_{\alpha} \times I_{\alpha}) \hookrightarrow \Sh_{\Omega_{\alpha} \cup T_{-\eta} \Omega_{\alpha}}(B_{\alpha} \times I_{\alpha}) \\ 
\label{eq:3103} \Sh_{\bL}^0(B_\alpha\times I_\alpha) \hookrightarrow \Sh_{\bL\cup T_{-\eta}\bL}(B_\alpha\times I_\alpha)
\end{align}
The right adjoint functors to \eqref{eq:3102}, resp. \eqref{eq:3103}, are given by microlocalization along $\Omega_{\alpha} \subset \Omega_{\alpha} \cup T_{-\eta} \Omega_{\alpha}$, resp. $\bL \subset \bL \cup T_{-\eta} \bL$, along with \eqref{eq:branch-lemma}.
\end{prop*}

\begin{proof}
(i)
It suffices to prove that $\Hom(F_i,F_j)= 0$ when $F_i \in \Sh_{\Omega_i}^0(U)$, $F_j \in \Sh_{\Omega_j}^0(U)$ and $i \neq j$. Since everything is local, we may take $M=V$, where $V$ is a vector space with an Euclidean inner product. Let $(0, 0;p_i, -dt)$ be the center of $\Omega_i$. Now we use notations from Appendix \ref{Appendix: contact}. We may take $U$ to be $B_\epsilon\times \bR$ for all $\epsilon>0$ sufficiently small. 
Applying $K_{\chi_\sfF}\circ $ for the Fourier transform $T^*V\to T^*V^*$, we get an equivalence
\begin{align*}
\Sh_{\Omega_i\sqcup \Omega_j}^0(B_\epsilon\times\bR)\simeq \Sh^{<0, D_\epsilon^*V^*\times\bR}_{\chi_{\sfF}(\Omega_i)\sqcup\chi_{\sfF}(\Omega_j)}(V^*\times\bR), i\neq j. 
\end{align*} 
First, as in the proof of Lemma \ref{lemma: Phi equivalence}, for each $\Omega_i$, we have 
$$\XZ(\Omega_i)\subset (B^*_{\nu_i(\ep)}+p_i)\times (-\epsilon h_i(\epsilon), \epsilon h_i(\epsilon)), $$ 
for some $h_i(\ep)\to 0,\ \nu_i(\ep)\to 0$ as $\ep\to 0$. 
Second, let $\cQ_i=T^*(\overline{B^*}_{\nu_i(\ep)}+p_i)\times [-\epsilon h_i(\epsilon), \epsilon h_i(\epsilon)]$. 
Using that $\rho_{\cQ_{i,\ep}}$, as in (\ref{eq: cor: C, Phi_C, ep,delta}) for $W=V^*$, is essentially surjective, we know that for every $G_i\in \Sh^{<0, D_\epsilon^*V^*\times\bR}_{\chi_{\sfF}(\Omega_i)}(V^*\times\bR)$, there exists $\widetilde{G}_i\in \Sh^{<0}_{\cQ_i}(V^*\times\bR)$ such that $\rho_\ep (\widetilde{G}_i)\cong G_i$. Since
\begin{align*}
\epsilon>\frac{2\ep(h_i(\ep)+h_j(\ep))}{|p_i-p_j|-\nu_i(\ep)-\nu_j(\ep)} \text{ as }\ep\to 0,
\end{align*}
we are able to apply Corollary \ref{lemma: disjoint support} to $\widetilde{G}_i$ with 
\begin{align*}
U_i=(B^*_{\nu_i(\ep)}+p_i)\times (-\epsilon h_i(\epsilon), \epsilon h_i(\epsilon)), 0<\epsilon\ll |p_i-p_j|
\end{align*}
and get that $\Hom(G_i, G_j)\cong\Hom(\rho_\ep(\widetilde{G}_i), \rho_\ep(\widetilde{G}_j))=0$ for $i\neq j$, hence the same holds for $\Hom(F_i, F_j)$.

(ii) Using (i), it suffices to prove (\ref{eq:3102}). Again, WLOG, we may replace $B_\alpha\times I_\alpha$ (resp. $\Omega_\alpha$) by $B_\epsilon\times \bR\subset V\times\bR$ (resp. $\Omega$), as in Appendix \ref{Appendix: contact} Notation \ref{Appendix: Notations}. The argument is a slight modification of that in \S\ref{subsec: proof of prop: K_chi, Modk}. 

First, by replacing $\Omega$ with $\Omega\cup T_{-\eta}\Omega$ for $0<\eta\ll \ep$ and using that the standard contact transformations from \S\ref{subsubsec: std contact tr} are $\bR$-equivariant, the argument in \S\ref{subsec: proof of prop: K_chi, Modk} applies and gives the upper equivalence in the following diagram: 
\begin{align*}
\xymatrix{
\Sh^{<0}_{\Omega\cup T_{-\eta}\Omega}(B_\ep\times\bR)
\ar[r]^{\sim} &\Sh^{<0}_{\Omega_0'\cup T_{-\eta}\Omega_0'}(B_\ep^*\times\bR)\\
\Sh^{<0}_{\Omega}(B_\ep\times\bR)\ar[u]
\ar[r]^{\sim} &\Sh^{<0}_{\Omega_0'}(B_\ep^*\times\bR)\ar[u]
}. 
\end{align*}
Second, by functoriality, we also get the commutative diagram above, with the two vertical arrows the natural functors. Since $\Omega_0'$ is the Legendrian boundary of the negative conormal of $V^*\times\{0\}$ in $V^*\times\bR$, the right upward arrow in the above diagram is fully faithful, then so is the left upward arrow. Now the proof is complete.  
\end{proof}

\begin{remark*}
Following Remark \ref{remark: cM0;p0toMShp}, part (i) in the above proposition directly generalizes when $\Omega_i$ is replaced by $C_i$ satisfying the conditions in Remark \ref{remark: cM0;p0toMShp} but with a shifted center $(0,0;p_i)$, $p_i\neq p_j$ for $i\neq j$. 

\end{remark*}

\subsection{Definition of $\Brane_L$ and $\mumon$}
We will give a recipe for the sheaf of $\infty$-categories $\Brane_L$ on $L$ in terms of a functor $\Nb(L) \to \St_{\bk}$ which we will define first, in \eqref{eq:as-functor}.  We will abuse notation and denote this functor by $\Brane_L$ as well --- to avoid confusion we will write $\Gamma(\Omega,\Brane_L)$ for the sections of $\Brane_L$ over an open subset $\Omega \subset L$, and $\Brane_L(U,\Omega)$ for the value of \eqref{eq:as-functor} on $(U,\Omega) \in \Nb(L)$.

For each $(U,\Omega) \in \Nb(L)$, the category $\Shall(U;\bk)$ has a pair of full subcategories
\[
\Loc(U) \subset \Sh_{\Omega}(U) \subset \Shall(U)
\]
which are preserved by the pullback functors $\Shall(U) \to \Shall(U')$ whenever $(U',\Omega') \leq (U,\Omega)$.  We may regard it as a functor
\begin{equation}
\label{eq:to-arrow-cat}
\Nb(L) \to \Fun(\Delta^1,\St_{\bk}) \qquad (U,\Omega) \to \big[ \Loc(U) \to \Sh_{\Omega}(U) \big]
\end{equation}
We define 
\begin{equation}
\label{eq:as-functor}
\Brane_L(U,\Omega) := \Sh_{\Omega}(U)/\Loc(U)
\end{equation}
or more formally to be the composite
${\Nb(L) \xrightarrow{\eqref{eq:to-arrow-cat}} \Fun(\Delta^1,\St_{\bk}) \xrightarrow{\eqref{eq:cone-of-categories}} \St_{\bk}}$.

By \eqref{eq:branch-lemma}, the value of \eqref{eq:as-functor} on any element of $\Nb(L)$ is (noncanonically) equivalent to $\Mod(\bk)$ --- in fact it carries every arrow in $\Nb(L)$ to an equivalence.  That is, \eqref{eq:as-functor} is a locally constant functor (in the sense of \S\ref{subsec:lcso}) $\Nb(L) \to \St_{\bk}$.  Since $\Nb(L)$ is a complete covering of $L$ by contractible open subsets, this functor determines a locally constant sheaf on $L$ --- it is given explicitly by
\begin{equation}
\label{eq:as-limit}
\Gamma(\Omega,\Brane_L) = \varprojlim_{\Nb(L)_{/\Omega}} \Brane_L(U',\Omega')
\end{equation}
where the limit is over the subposet of $\Nb(L)_{/\Omega} \subset \Nb(L)$ given by pairs $(U',\Omega')$ with $\Omega' \subset \Omega$.

For each $(U,\Omega) \in \Nb(L)$, we have the composite functor
\[
\Sh_{\bL}(M \times \bR) \to \Sh_{\bL\vert_U}(U) \to \Sh_{\bL\vert_U}(U)/\Sh_{\bL\vert_U - \Omega}(U) \cong \Brane_L(U,\Omega)
\]
where the first arrow is restriction, the second arrow is localization, and the third is the equivalence given by \eqref{eq:branch-variant}.  By the universal property of limits they assemble to a functor $\Sh_{\bL}(M \times \bR) \to \varprojlim \Brane_L(U,\Omega)$, i.e. to
\begin{equation}
\label{eq:mumon}
\Sh_{\bL}(M \times \bR) \to \Gamma(L,\Brane_L)
\end{equation}
which we denote by $\mumon$, microlocal monodromy.

\subsection{Cylindrical \v Cech covers}
\label{subsec:good-cyl-cover}

We will say that a \v Cech covering (in the sense of \S\ref{subsec:covers}) of $M \times \bR$ is a ``good cylindrical covering'' if every chart in the covering is either empty or a good cylinder.  If $\cV$ is a good cylindrical covering, then the set of triples $(B,I,\Omega) \in \Cylb(L)$ such that $B \times I \in \cV$ is itself a \v Cech covering of $L$.  Thus, for a good cylindrical covering we can compute $\Gamma(L,\Brane_L)$ as a smaller limit:
\[
\Gamma(L,\Brane_L) \cong \varprojlim_{(B,I,\Omega) \mid B \times I \in \cV} \Brane_L(B \times I,\Omega)
\]

Let us prove here that good cylindrical covers exist, and can be taken arbitrarily fine.  That is, if $\cV'$ is any covering of $M \times \bR$, there is a locally finite refinement $\cV$ such that $\cV$ is a good cylindrical cover.

\begin{proof}
Let $\{\Omega_\alpha'\}_{\alpha\in J'}$ be the covering of $\bL$ determined by $\cV'$. Fix a Whitney stratification $\cS$ on $M\times\bR$ compatible with $\bF$, whose open strata are the connected components of $M\times\bR-\bF$. We prove the statement by induction on the dimension of the strata. We will use $\cS_{\geq k}$ (resp. $\cS_{\leq k}$ and $\cS_{=k}$) to denote the subcollection of strata that have dimension greater than or equal to $k$ (resp. less than or equal to $k$ and equal to $k$). 

First, for each $q=(x_q, t_q)\in \cS_{=0}$ and $p\in\XZ^{-1}(q)\cap \bL$, one can choose a good cylinder $(B, I, \Omega)$ with $p\in \Omega$ as follows.   Let $\mathrm{r}_{x_q}^2$ (resp. $\mathrm{r}_{t_q}^2$) be the distance squared function from $x_q$ in $M$ (resp. from $t_q$ in $\bR$). Choose an $\epsilon_0>0$ such that $\mathrm{r}_{x_q}^2+\mathrm{r}_{t_q}^2$ has no $\Lambda_{\cS}$-critical value in $(0, \epsilon_0^2)$ and such that
\[(\{x_q\}\times\bR)\cap\bF\cap B_{\epsilon_0}(q)=\{q\}.\]
 In particular, we have $\partial (B_{\epsilon}(q))$ is transverse to $\cS$ for any $0<\epsilon<\epsilon_0$. Fix an $\epsilon\in (0, \epsilon_0)$ and let $\cS^\epsilon_q$ be the unique minimal stratification determined by the two transverse stratifications $\cS$ and $\{B_{\epsilon}(q), \partial (B_{\epsilon}(q)), M-\overline{B}_{\epsilon}(q)\}$. Choose $\delta_0>0$ such that $\mathrm{r}_{x_q}^2$ has no $\Lambda_{\cS^\epsilon_q}$-critical value in $(0,\delta_0^2)$ and $\bF\cap B_\epsilon(q)\cap \{\mathrm{r}_{t_q}^2\geq \epsilon^2-\delta_0^2\}=\varnothing.$ Now fix any $\delta\in (0,\delta_0)$, let $B=\{\mathrm{r}_{x_q}^2<\delta^2\}\subset M$, $I=\{\mathrm{r}_{t_q}^2<\epsilon^2-\delta^2\}\subset \bR$ and $\Omega$ be the unique component of $\bL\cap (T^*(B)\times I)$ that contains $p$, then $(B, I ,\Omega)$ is a good cylinder. One can make $\Omega$ arbitrarily small, and in particular, contained in $\bigcap\limits_{p\in\Omega'_\alpha}\Omega'_\alpha$.

Second, for any $S_\beta\in\cS$ of codimension greater than 0, the projection $\pi_M|_{S_\beta}: S_\beta\rightarrow \pi_M(S_\beta)$ is a local diffeomorphism. After a refinement of $\cS$, we can assume that $\pi_M|_{S_\beta}$ is a diffeomorphism, and there exists a tubular neighborhood $\cT_\beta$ of $S_\beta$, such that the projection 
\[
\cT_\beta\cap \left(\bigcup\limits_{S_\alpha\in\cS_{\leq \dim M}}S_\alpha\right)\rightarrow M
\]
is a stratified map for a stratification $\widetilde{\cS}$ on $M$ compatible with $\pi_M(S_\beta)$. Now choose a tube system of $\widetilde{\cS}$ with the tube data of $\pi_M(S_\beta)$ given by $(\widetilde{\cT}_\beta, \widetilde{\pi}_\beta, \widetilde{\rho}_\beta)$, where $\widetilde{\pi}_\beta: \widetilde{\cT}_\beta\rightarrow \pi_M(S_\beta)$ is a retraction and $\widetilde{\rho}_\beta$ is a distance function. Then the preimage $\pi_M^{-1}(\widetilde{\cT}_\beta)\cap \cT_\beta$ defines a tube for $S_\beta$, in the sense that there is a retraction from $\pi_M^{-1}(\widetilde{\cT}_\beta)\cap \cT_\beta$ to $S_\beta$ and $\rho_\beta=\sqrt{\widetilde{\rho}_\beta^2+\Delta t^2}$ is a distance function compatible with it, where $\Delta t$ is the difference in the $t$-coordinate. Now repeating the argument in the first step, we can show that for any point $(x,t)\in S_\beta$, there exists a contractible neighborhood $U_x\subset \pi_M(S_\beta)$ with piecewise smooth boundary, $\delta_x>0$ and $a_x<t<b_x$, such that 
\begin{equation}\label{eq: a good cylinder}
(\widetilde{\pi}_\beta^{-1}U_x\cap \{\tilde{\rho}_\beta<r\})\times (a_x,b_x)
\end{equation}
is a good cylinder for all $0<r<\delta_x$, and it is contained in one of the open sets in $\cV'$. It is clear that if we replace $U_x$ by any smaller contractible neighborhood $U_x'\subset U_x$ in (\ref{eq: a good cylinder}), then it is also a good cylinder.

Next, suppose we have chosen a good cylindrical cover $\{(B_\alpha, I_\alpha, \Omega_\alpha)\}_{\alpha\in J_{\leq k}}$ of a neighborhood of $\cS_{\leq k}$, which is a refinement of $\cV'$. For any $S_\beta\in \cS_{=k+1}$, let $U_{S_\beta, \alpha}=\pi_M((B_\alpha\times I_\alpha)\cap S_\beta)$ for $\alpha\in J_{\leq k}$. Now extend $\{U_{S_\beta, \alpha}\}_{\alpha\in J_{\leq k}}$ to a locally finite open covering  $\{U_{S_\beta, \alpha}\}_{\alpha\in J_\beta}$ ($J_\beta$ is an index set containing $J_{\leq k}$) of $S_\beta\simeq \pi_M(S_\beta)$, such that each of $\{U_{S_\beta,\alpha}\}_{\alpha\in J_\beta-J_{\leq k}}$ is contained in $U_x$ for some $x\in S_\beta$, cf. (\ref{eq: a good cylinder}) and there is a neighborhood of $\partial S_\beta$ where the newly added open sets don't intersect. Furthermore, we can refine the covering (only change the newly added ones) so that any finite intersection $\bigcap_{i=1}^\ell U_{S_\beta, \alpha_i}$ is either empty or contractible. By induction (which stops at $\cS_{\leq \dim M}$), we get the desired claim.

\end{proof}

\subsection{Tamarkin's convolution}
\label{subsec:tam-conv-2}
For each $s \in \bR$, we will write $T_s$ for the translation-in-the-second-coordinate operator on $M \times \bR$ and $(T^* M) \times \bR$, given by the formula $T_s(x,t) = (x,t+s)$.  Write $T_s \bL$  and $T_s \bF$ for the image of $\bL$ or $\bF$ under $T_s$.  We also abuse notation and write $T_s$ in place of $T_{s,*}$ for the functor $\Shall(M \times \bR) \to \Shall(M \times \bR)$ induced by $T_s$.

For each sheaf $I \in \Shall(\bR)$, Tamarkin studies an endofunctor $T_I:\Shall(M \times \bR) \to \Shall(M \times \bR)$.  It is given by the formula
\begin{equation}
\label{eq:tamarkin-maps}
T_I F := \overrightarrow{T}_! (\proj_2^* I \otimes_{\bk} \overleftarrow{T}^* F)\cong \overrightarrow{T_I}_! \overleftarrow{T_I}^* F
\end{equation}
where $\proj_2$ is the projection $M \times \bR \to \bR$, $\overrightarrow{T},\overleftarrow{T}:M \times \bR \times \bR \to M \times \bR$ (resp. $\overrightarrow{T_I},\overleftarrow{T_I}:M \times \bR \times  I \to M \times \bR$) are given by addition in the second two factors and projection onto the first two factors, respectively.
When $I$ is the skyscraper sheaf at $s$ then $T_I = T_s$.  We will also be interested in the functors $T_I$, restricted to the subcategory $\Sh_{\bL}(M \times \bR)$, when $I$ is a constant sheaf on a half-open interval or a ray.  

For example, the following figure indicates the boundary conditions of the image of the sheaf of \S\ref{example:wavefront} under $T_{\deltazero}$:
\begin{center}
\begin{tikzpicture}
\draw [thick, domain = -1.0:-0.920420967978104,samples = 100]
plot ({0.16*\x*\x*\x*\x*\x*\x*\x*\x*\x-1.19*\x*\x*\x*\x*\x*\x*\x+5.09*\x*\x*\x*\x*\x-10.32*\x*\x*\x+8.08*\x},{-0.102*\x*\x*\x*\x*\x*\x*\x*\x + 0.606*\x*\x*\x*\x*\x*\x -2.411*\x*\x*\x*\x + 2.64*\x*\x +1.023*ln(100-81*\x*\x)});

\draw [thick, dashed, domain = -0.920420967978104: -0.5928782718649503,samples = 100]
plot ({0.16*\x*\x*\x*\x*\x*\x*\x*\x*\x-1.19*\x*\x*\x*\x*\x*\x*\x+5.09*\x*\x*\x*\x*\x-10.32*\x*\x*\x+8.08*\x},{-0.102*\x*\x*\x*\x*\x*\x*\x*\x + 0.606*\x*\x*\x*\x*\x*\x -2.411*\x*\x*\x*\x + 2.64*\x*\x +1.023*ln(100-81*\x*\x)});
\draw [thick, domain = -0.5928782718649503:0.5928782718649503,samples = 100]
plot ({0.16*\x*\x*\x*\x*\x*\x*\x*\x*\x-1.19*\x*\x*\x*\x*\x*\x*\x+5.09*\x*\x*\x*\x*\x-10.32*\x*\x*\x+8.08*\x},{-0.102*\x*\x*\x*\x*\x*\x*\x*\x + 0.606*\x*\x*\x*\x*\x*\x -2.411*\x*\x*\x*\x + 2.64*\x*\x +1.023*ln(100-81*\x*\x)});
\draw [thick, dashed, domain = 0.5928782718649503:0.920420967978104,samples = 100]
plot ({0.16*\x*\x*\x*\x*\x*\x*\x*\x*\x-1.19*\x*\x*\x*\x*\x*\x*\x+5.09*\x*\x*\x*\x*\x-10.32*\x*\x*\x+8.08*\x},{-0.102*\x*\x*\x*\x*\x*\x*\x*\x + 0.606*\x*\x*\x*\x*\x*\x -2.411*\x*\x*\x*\x + 2.64*\x*\x +1.023*ln(100-81*\x*\x)});
\draw [thick, domain = 0.920420967978104:1,samples = 100]
plot ({0.16*\x*\x*\x*\x*\x*\x*\x*\x*\x-1.19*\x*\x*\x*\x*\x*\x*\x+5.09*\x*\x*\x*\x*\x-10.32*\x*\x*\x+8.08*\x},{-0.102*\x*\x*\x*\x*\x*\x*\x*\x + 0.606*\x*\x*\x*\x*\x*\x -2.411*\x*\x*\x*\x + 2.64*\x*\x +1.023*ln(100-81*\x*\x)});

\draw [thick, dashed, domain = -0.9375922925318348: -0.5928782718649503,samples = 100]
plot ({0.16*\x*\x*\x*\x*\x*\x*\x*\x*\x-1.19*\x*\x*\x*\x*\x*\x*\x+5.09*\x*\x*\x*\x*\x-10.32*\x*\x*\x+8.08*\x},{-0.102*\x*\x*\x*\x*\x*\x*\x*\x + 0.606*\x*\x*\x*\x*\x*\x -2.411*\x*\x*\x*\x + 2.64*\x*\x +1.023*ln(100-81*\x*\x)-.5});

\draw [thick, domain = -0.5928782718649503: -0.3004604429134293,samples = 100]
plot ({0.16*\x*\x*\x*\x*\x*\x*\x*\x*\x-1.19*\x*\x*\x*\x*\x*\x*\x+5.09*\x*\x*\x*\x*\x-10.32*\x*\x*\x+8.08*\x},{-0.102*\x*\x*\x*\x*\x*\x*\x*\x + 0.606*\x*\x*\x*\x*\x*\x -2.411*\x*\x*\x*\x + 2.64*\x*\x +1.023*ln(100-81*\x*\x)-.5});

\draw [thick, dashed, domain = -0.3004604429134293:0.3004604429134293,samples = 100]
plot ({0.16*\x*\x*\x*\x*\x*\x*\x*\x*\x-1.19*\x*\x*\x*\x*\x*\x*\x+5.09*\x*\x*\x*\x*\x-10.32*\x*\x*\x+8.08*\x},{-0.102*\x*\x*\x*\x*\x*\x*\x*\x + 0.606*\x*\x*\x*\x*\x*\x -2.411*\x*\x*\x*\x + 2.64*\x*\x +1.023*ln(100-81*\x*\x)-.5});

\draw [thick, domain = 0.3004604429134293:0.5928782718649503,samples = 100]
plot ({0.16*\x*\x*\x*\x*\x*\x*\x*\x*\x-1.19*\x*\x*\x*\x*\x*\x*\x+5.09*\x*\x*\x*\x*\x-10.32*\x*\x*\x+8.08*\x},{-0.102*\x*\x*\x*\x*\x*\x*\x*\x + 0.606*\x*\x*\x*\x*\x*\x -2.411*\x*\x*\x*\x + 2.64*\x*\x +1.023*ln(100-81*\x*\x)-.5});

\draw [thick, dashed, domain = 0.5928782718649503:0.9375922925318348,samples = 100]
plot ({0.16*\x*\x*\x*\x*\x*\x*\x*\x*\x-1.19*\x*\x*\x*\x*\x*\x*\x+5.09*\x*\x*\x*\x*\x-10.32*\x*\x*\x+8.08*\x},{-0.102*\x*\x*\x*\x*\x*\x*\x*\x + 0.606*\x*\x*\x*\x*\x*\x -2.411*\x*\x*\x*\x + 2.64*\x*\x +1.023*ln(100-81*\x*\x)-.5});

\end{tikzpicture}
\end{center}

\subsection{Tamarkin convolution on $\Sh_{\bL}(M \times \bR)$}
The exact triangles
\begin{eqnarray*}
\bk_{\inftyzero} \to \bk_{\{0\}} \to \Sigma \bk_{(-\infty,0)} \to \\
\bk_{\deltazero} \to \bk_{\inftyzero} \to \bk_{\inftydelta} \to 
\end{eqnarray*}
of sheaves on $\bR$ induce exact triangles of endofunctors
\begin{eqnarray}
\label{eq:ToneT}
T_{\inftyzero} \to 1 \to \Sigma \circ T_{(-\infty,0)} \to  \\
\label{eq:TTT}
T_{\deltazero} \to T_{\inftyzero} \to T_{\inftydelta} \to
\end{eqnarray}
The subcategory $\Sh_{\bL}(M \times \bR)$ is stable by the functor in \eqref{eq:ToneT}, and the map $1 \to \Sigma \circ T_{(-\infty,0)}$ is idempotent, and corresponds to the localization by the subcategory $\Sh_{\bL}^0(M \times \bR)$ (whose quotient is $\Loc(M \times \bR)$).  There is a projector from $\Sh_{\bL}(M\times \bR)$ to the right orthogonal complement of $\Loc(M\times \bR)$, given by $F\mapsto (\overrightarrow{T}_{[0,\infty)})_* (\overleftarrow{T}_{[0,\infty)})^! F$. 
Similar to \eqref{eq:colocal}, the composite
\[
\Sh_{\bL}^0(M \times \bR) \subset \Sh_{\bL}(M \times \bR) \to \Sh_{\bL}(M \times \bR)/\Loc(M \times \bR)
\]
is an equivalence, although the image of the right adjoint is the image of $(\overrightarrow{T}_{[0,\infty)})_* (\overleftarrow{T}_{[0,\infty)})^!$.

On $\Sh_{\bL}^0(M \times \bR)$, the functor $T_{\inftyzero}$ acts as the identity, and $T_{\inftydelta}$ acts as $T_{-\delta}$.  Thus the second arrow in \eqref{eq:TTT} induces a natural map 
\begin{equation}
\label{eq:FTF}
F \to T_{-\delta} F
\end{equation}
whenever $F$ is $\Sh^0_{\bL}(M \times \bR)$.  As $T_{-\delta} F$ lies in $\Sh_{T_{-\delta} \bL}^0(M \times \bR)$, the homotopy fiber $T_{\deltazero} F$ lies in $\Sh^0_{\bL \cup T_{-\delta} \bL}(M \times \bR)$ --- we study these categories in the next section.

\subsection{The total Tamarkin convolution}
\label{subsec:total-Tam}
From now on, we assume $L$ satisfies Assumption \ref{assump: lel}.
In this section we construct, for $\delta$ and $\delta'$ sufficiently small, a canonical equivalence
\begin{equation}
\label{eq:imm-small-delta}
\Sh^0_{\bL \cup T_{-\delta} \bL}(M \times \bR) \simeq\Sh^0_{\bL \cup T_{-\delta'} \bL}(M \times \bR). 
\end{equation}
The equivalence is an application of GKS as in \ref{subsec:GKS}, and a ``stronger result" showing in Proposition \ref{prop:shift-down} below, has an interpretation in terms of the invariance of $\Hom$ between sheaves under ``a (small) positive wrapping" by the Hamiltonian $-\tau$ (on the second copy of $\bL$). 
In fact, when $L \to T^* M$ is an embedding, \eqref{eq:imm-small-delta} holds for all $\delta,\delta'$.

Let $\cR \subset (0,\infty) \times M \times \bR \times \bR$ be the set of quadruples $(\delta,m,t_1,t_2)$ such that $-\delta \leq t_2 < 0$.  Define a correspondence
\[
M \times \bR \xleftarrow{\Ttotleft} \cR \xrightarrow{\Ttotright} (0,\infty) \times M \times \bR
\]
where $\Ttotleft(\delta,m,t_1,t_2) := (m,t_1)$ and $\Ttotright(\delta,m,t_1,t_2) := (\delta,m,t_1+t_2)$.  Then we have a continuous functor
\[
\Ttot = \left(\Ttotright\right)_! \left(\Ttotleft\right)^*:\Shall(M\times\bR) \to \Shall((0,\infty) \times M \times \bR)).
\]
It follows from the base-change formula that for each $\delta \in (0,\infty)$, the composite
\[
\Shall(M \times \bR) \xrightarrow{\Ttot} \Shall((0,\infty) \times M \times \bR) \to \Shall(M \times \bR)
\]
where the second arrow is the restriction to $\{\delta\} \times M \times \bR$, is isomorphic to $T_{\deltazero}$.  If $F \in \Sh^0_{\bL}(M \times \bR)$, then $\Ttot F$ has singular support in the suspension of the family of Legendrians $\bL\cup T_{-s}\bL, s\in (0,\infty)$ that we will denote by $\Ttot\bL \subset T^*((0,\infty) \times M) \times \bR$. Under the assumption that $L$ is embedded, we will construct \eqref{eq:imm-small-delta} by proving that for every $\delta > 0$, the restriction functor
\begin{equation}
\label{eq:autod}
\Sh^0_{\Ttot \bL}((0,\infty) \times M \times \bR) \to \Sh^0_{\bL \cup T_{-\delta} \bL}(M \times \bR)
\end{equation}
is an equivalence.  If $L$ is immersed, then for some $\epsilon > 0$, the same argument shows $\Sh^0_{\Ttot \bL}((0,\epsilon) \times M \times \bR) \to \Sh^0_{\bL \cup T_{-\delta} \bL}(M \times \bR)$ is an equivalence for every $\delta < \epsilon$, proving \eqref{eq:imm-small-delta}.

The two copies of $(0,\infty) \times \bL$ in $\Ttot \bL$ are disjoint from each other in $T^*((0,\infty) \times M) \times \bR$ (and their closures are disjoint in $T^*([0,\infty) \times M) \times \bR$ as well), with one copy the suspension of the constant family $\bL, s\in (0,\infty)$ and the other the suspension of the family $T_{-s}\bL, s\in (0,\infty)$.  If $L \to T^* M$ is an embedding, they are unlinked.  That is, for any $N> 0$, there is a homogeneous Hamiltonian isotopy (in the sense of \S\ref{subsec:GKS}) $h:T^{*,\circ}(M \times \bR) \times (0,\infty) \to T^{*,\circ}(M \times \bR)$ whose graph convolved with $\bL \cup T_{-N} \bL$ is $\Ttot \bL$.

To construct $h$, first let $H:T^{*,\circ}(M \times \bR) \to \bR$ be the function $(x,t;\xi,\tau) \mapsto -\tau$.  At time $s$, the Hamiltonian flow $\varphi_H^s$ is the shifting operator $T_{-s}$.  Choose a homogeneous bump function $b_{\epsilon}:T^{*,\circ}(M \times \bR) \to \bR$ smoothly depending on $\epsilon \in \bR_{>0}$ such that $b_{\epsilon} = 1$ in an open neighborhood of $\bigcup_{s > \epsilon} \mathrm{Cone}(T_{-s} \bL)$ and $\mathrm{supp}(b_{\epsilon}) \cap \Cone(\bL) = \varnothing$.  Since $\XY$ is an embedding, for any $N > 0$ the Hamiltonian flow $\varphi_{H_{\epsilon}}^{s - N}$ of the function $H_{\epsilon} := b_\epsilon \cdot H$ takes $\bL \cup T_{-N} \bL$ to $\bL \cup T_{-s} \bL$ for any $s > \epsilon$.  Then for $h$ we can take $\{\varphi^{s - N}_{H_{\frac{1}{2} s}}\}_{s \in (0,\infty)}$.

To prove \eqref{eq:autod}, we want to apply \eqref{eq:312} to $h$ --- but we first have to modify $h$ to be horizontally compactly supported over any compact subinterval $[a,b]$ of the time interval $(0,\infty)$  (i.e. the projection of $\mathrm{supp}(h)$ on $M\times \bR\times[a,b]$ is compact) . We will construct an isotopy $h^K$ for each $K\ll 0$ over the time interval $(0,|K|)$, and we will take $K\rightarrow -\infty$.  
Explicitly, for any $K \ll 0$ such that $\eta(4K)<\delta$ and $\eta(K)<\epsilon$, where $\eta, \delta$ are as in Lemma \ref{Lag ends} and Assumption \ref{assump: lel}, choose an increasing bump function $c_K:\bR \to \bR$ such that $\mathrm{supp}(c_K) \subset (4K,\infty)$ and $c_K\vert_{(K+1,\infty)} = 1$.  Let $H^K_{\frac{1}{2} s} = (c_K \circ t) \cdot H_{\frac{1}{2} s}$.  The bump functions $b_{\epsilon}$ and $c_{K}$, can be chosen so that the Hamiltonian isotopy
\[
h^K:=\{\varphi_{H^{K}_{\frac{1}{2}s}}^{s-\delta}\}_{s\in (0,|K|)}
\]
is horizontally compactly supported over any compact subinterval of $(0,|K|)$, 
and for all $s\in (0,|K|)$
\begin{align*}
&\XZ(\varphi_{H^{K-s}_{\frac{1}{2}s}}^{s-\delta}(T_{-\delta}\bL))\cap M\times (4K, 2K+1)\\
=&\bigcup\limits_{|2K+1|<r<|4K|}\XZ(\Lambda_{{\eta(-r)+s(r)}})\times \{-r\},
\end{align*}
for some function $0\leq s(r)<\epsilon$ and $\Lambda_u$ is as in Lemma  \ref{Lag ends}. 

By Assumption \ref{assump: lel}, we have compatible natural equivalences $\Sh_{\bL\cup T_{-\delta}\bL}(M\times (K_1, K_2))\simeq \Sh_{\Lambda\cup \Lambda_\epsilon}(M)$ for $-\infty\leq K_1<K_2\ll 0$ and $0<\epsilon\ll \delta$. These give a natural equivalence $\Sh_{\bL\cup T_{-\delta}\bL}(M\times \bR)\simeq \Sh_{\bL\cup T_{-\delta}\bL}(M\times (K,\infty))$ for $K\ll 0$. Let $h^K(\bL\cup T_{-\delta}\bL)$ be the image of $\bL\cup T_{-\delta}\bL$ under the Hamiltonian isotopy $h^K$.  In the same way, we have compatible equivalences $\Sh_{h^K(\bL\cup T_{-\delta}\bL)}((0, -\frac{1}{2}K)\times M\times\bR)\simeq \Sh_{\Ttot\bL}((0,-\frac{1}{2}K)\times M\times\bR)$ for $K\ll 0$. Taking $K\rightarrow -\infty$ and applying \eqref{eq:312}, 
we have the desired equivalence $\Sh_{\Ttot \bL}((0,\infty) \times M \times \bR) \to \Sh_{\bL \cup T_{-\delta} \bL}(M \times \bR)$.

\begin{prop}
\label{prop:shift-down}
Suppose $F,G \in \Sh_{\bL}^0(M \times \bR)$, and let $G \to T_{-\delta} G$ be the map \eqref{eq:FTF}.
\begin{enumerate}
\item For $\delta > 0$ sufficiently small, the map $\Hom(F,G) \to \Hom(F,T_{-\delta} G)$ is an isomorphism.
\item If $L \to T^* M$ is an embedding, then for all $\delta > 0$, the map $\Hom(F,G) \to \Hom(F,T_{-\delta} G)$ is an isomorphism and $\Hom(T_{-\delta} F,G) = 0$.
\end{enumerate}
\end{prop}

\begin{proof}
By \eqref{eq:imm-small-delta}, whenever $\delta_1 < \delta_2$ is sufficiently small, the map
$
\Hom(F,T_{-\delta_1} G) \to \Hom(F,T_{-\delta_2} G)
$
induced by putting $\delta = \delta_2 - \delta_1$ in  \eqref{eq:FTF} is an isomorphism.  It follows that 
\[
\Hom(F,\varprojlim_{\delta > 0} T_{-\delta} G) \to \Hom(F,T_{-\delta} G)
\]
is an isomorphism for $\delta$ sufficiently small.  But the map $G \to \varprojlim_{\delta > 0} T_{-\delta} G$ is induced by isomorphism $\bk_{(-\infty, 0]}\overset{\sim}{\rightarrow} \varprojlim_{\delta > 0}  \bk_{\inftydelta}$--- this proves (1).

Together with \eqref{eq:imm-small-delta}, (1) implies that the map $\Hom(F,G) \to \Hom(F,T_{-\delta} G)$ is an isomorphism for every $\delta > 0$.  To prove the remaining half of (2), we may assume (again by \eqref{eq:imm-small-delta}) that $\delta \gg 0$.  Let $H = \underline{\Hom}(T_{-\delta} F,G)$ --- where in this presentable setting the internal Hom is formally defined to be the object representing the functor $\Hom(T_{-\delta} F \otimes_{\bk} -, G)$.  Then $H$ is quasi-constructible with respect to a stratification of $M \times \bR$ refining $\bF \cup T_{-\delta} \bF$.
For $\delta \gg C \gg 0, K\ll 0$, 
there is a horizontally compactly supported homogeneous Hamiltonian flow  on $T^{*, \circ}(M \times \bR )$ which carries $(T_{-\delta} \bL \cup \bL) \cap (T^* M \times (K,C))$ to $(T_{-\delta} \bL \cup (\Lambda \times \bR)) \cap ((T^* M) \times (K,C))$.  
Then $-dt$ (where $t$ is the projection $M \times \bR \to \bR$) is not in its singular support by \cite[Prop. 5.4.14]{KS}.  The microlocal Morse lemma (Corollary \ref{cor: MML}) completes the proof.
\end{proof}

\begin{theorem}
\label{thm:312}
Assume $L$ satisfies Assumption \ref{assump: lel}. If $\delta > 0$ is sufficiently small relative to $\bL$, then on $\Sh_{\bL}(M \times \bR)$ the functor $T_{\deltazero}$ is isomorphic to a composite functor
\[
\Sh_{\bL}(M \times \bR) \xrightarrow{\mumon} \Gamma(L,\Brane_L) \xrightarrow{\iota} \Sh_{\bL \cup T_{-\delta} \bL}(M \times \bR)
\]
where $\iota$ is a full embedding.
\end{theorem}

\begin{proof}
The fact that the primitive \S\ref{subsec:exact-lags} is proper and has only finitely many critical values implies that for $K \ll 0$, the front projection $\bL \to M \times \bR$ is topologically trivial above $M \times (-\infty,K)$.  In fact it is diffeomorphic to the product of $\Lambda \to M$ with $(-\infty,K)$.  As explained above, for any $\delta_0 > 0$ we may find a $K$ such that for all positive $\delta < \delta_0$, the restriction functors
\begin{equation}
\label{eq:first-reduction}
\begin{array}{c}
\Sh_{\bL \cup T_{-\delta} \bL}(M \times \bR) \to \Sh_{\bL \cup T_{-\delta} \bL}(M \times (K,\infty))
\\ \Sh_{\bL}(M \times \bR) \to \Sh_{\bL}(M \times (K,\infty))
\end{array}
\end{equation}
are equivalences.  The top functor of \eqref{eq:first-reduction} carries the localization by $\Loc(M\times \bR)$, denoted by $\cD_{\bL,\delta}(M \times \bR),$ to $\cD_{\bL,\delta}(M \times (K,\infty))$.  For the rest of the proof let us fix such a $\delta_0$ and $K$.  Thus, if we write $L_{>K}$ for the subset of $L$ on which the primitive takes values $> K$, to prove the theorem it suffices to construct a fully faithful functor (that we will also denote by $\iota$) from $\Gamma(L_{>K},\Brane_L)$ to $\Sh_{\bL \cup T_{-\delta} \bL}(M \times (K,\infty))$, such that the composite
\begin{equation}
\label{eq:first-reduction-prime}
\Sh_{\bL}(M \times (K,\infty)) \xrightarrow{\mumon} \Gamma(L_{>K},\Brane_L) \xrightarrow{\iota}\Sh_{\bL \cup T_{-\delta} \bL}(M \times (K,\infty))
\end{equation}
is isomorphic to $T_{\deltazero}$ (or more precisely to $T_{\deltazero}$ followed by restriction from $M \times (K-\delta,\infty)$ to $M \times (K,\infty)$ --- we will abuse notation and denote this simply by $T_{\deltazero}$).

Define the ``big cushion'' of a good cylinder $B \times (a,b)$ to be the supremum of those $\delta > 0$ such that $B \times (a+\delta,b+\delta)$ is a good cylinder.  When $\delta$ is smaller than the big cushion, the restriction functor
\[
\Sh_{\bL \cup T_{-\delta} \bL}(B \times (a-\delta,b)) \to \Sh_{\bL \cup T_{-\delta} \bL}(B \times (a,b))
\]
is an equivalence.  Because of this, $T_{\deltazero}$ induces a functor
\begin{equation}
\label{eq:Tdz-shl}
\Sh_{\bL}(B \times I) \to \Sh_{\bL \cup T_{-\delta} \bL}(B \times I)
\end{equation}
by first applying the correspondence
\[
B \times (a,b) \leftarrow B \times (a,b) \times [-\delta,0) \to B \times (a - \delta,b) \qquad \text{maps as in \eqref{eq:tamarkin-maps}}
\]
and then restricting to $B \times (a,b)$.  If $\Omega$ is a branch or union of branches of $\bL$ over $B \times I$, then \eqref{eq:Tdz-shl} carries $\Sh_{\Omega}(B \times I)$ to $\Shall(B \times I)$, and annihilates $\Loc(B \times I)$,

Let $\cV$ be a good cylindrical cover of $M \times (K,\infty)$, and suppose that the big cushion of every chart in $\cV$ is larger than $\delta$.  We may regard $\Loc$, $\Sh_{\bL}$, and $\Sh_{\bL \cup T_{-\delta} \bL}$ as objects of $\Fun(\cV^{\op},\St_{\bk})$, by restriction.  \eqref{eq:Tdz-shl} determines a morphism 
\begin{equation}
\label{eq:Tdz-with-V}
\Sh_{\bL} \to \Sh_{\bL \cup T_{-\delta} \bL} \qquad \text{in $\Fun(\cV^{\op},\St_{\bk})$}
\end{equation}
whose image under $\varprojlim_{\cV}:\Fun(\cV^{\op},\St_{\bk}) \to \St_{\bk}$ is $T_{\deltazero}$.  \eqref{eq:Tdz-with-V} kills $\Loc$ object-wise, so by the universal property of localization and \cite[Cor. 5.1.2.3]{higher-topoi}, it determines a commutative triangle $\Delta^2 \to \Fun(\cV^{\op},\St_{\bk})$ of the form
\begin{equation}
\label{eq:Delta2}
\xymatrix{
&  \Sh_{\bL}(B \times I)/\Loc(B \times I) \ar[dr] \\
\Sh_{\bL}(B \times I) \ar[ru] \ar[rr] && \Sh_{\bL \cup T_{-\delta} \bL}(B \times I)
}
\end{equation}
The image of \eqref{eq:Delta2} under $\varprojlim$ is
\[
\Sh_{\bL}(M \times (K,\infty)) \to \Gamma(L,\Brane_L) \overset{\iota}{\to} \Sh_{\bL \cup T_{-\delta} \bL}(M \times (K,\infty))
\]
where we draw only the top two arrows to save space and $\iota$ is the same one as in (\ref{eq:first-reduction-prime}).
We will show that when $\delta$ is sufficiently small, the first arrow in this is isomorphic to $\mumon$ and the second arrow $\iota$ is fully faithful , proving \eqref{eq:first-reduction-prime}.

Let us define the ``small cushion'' of a good cylinder to be the supremum of those $\delta$ such that the functor 
\begin{equation}
\label{eq:multi-branch-2}
\Sh_{\bL}(B \times I)/\Loc(B \times I) \to \Shall(B \times I)
\end{equation}
induced by \eqref{eq:Tdz-shl} is a full embedding. The small cushion of every good cylinder is positive, by \eqref{eq:3103}.  When the small cushion is larger than $\delta$, let $\cD_{\bL,\delta}(B \times I)$ be the essential image of \eqref{eq:multi-branch-2}. 

The functor $T_{(-\delta, 0]}$ is very local  in the sense that for any $F\in \Shall(B\times (a,b))$, $T_{(-\delta, 0]}F|_{B'\times (a',b')}$ only relies on $F|_{B'\times (a', b'+\delta)}$. Thus, if $B' \times I' \subset B \times I$ is a pair of good cylinders both of whose small cushions are larger than $\delta$, the restriction functor carries $\cD_{\bL,\delta}(B \times I)$ to $\cD_{\bL,\delta}(B' \times I')$.

Let $\cV$ be a good cylindrical cover of $M \times (K,\infty)$ \S\ref{subsec:good-cyl-cover}. Let $\cU \subset \Cylb$ be the set of triples $(B,I,\Omega)$ such that $B \times I \in \cV$.  As $\bF \cap [K,\infty)$ is compact, we may suppose that $\cU$ is finite.  Let $\delta>0$ be such that the small cushion of every good cylinder in $\cU$ is larger than $\delta$.  Now $\cD_{\bL,\delta}$ is a functor $\cV^{\op} \to \St_{\bk}$, and the full embeddings $\cD_{\bL,\delta}(B \times I) \subset \Shall(B \times I)$ give a morphism in $\Fun(\cV^{\op},\St_{\bk})$.  In particular they give a full embedding
\begin{equation}
\label{eq:proto-iota}
\varprojlim_{\cV} \cD_{\bL,\delta}(B \times I) \hookrightarrow \varprojlim_{\cV} \Shall(B \times I)
\end{equation}
The codomain of \eqref{eq:proto-iota} is naturally equivalent to $\Shall(M \times (K,\infty))$ \S\ref{subsec:descent}, and under this equivalence the essential image of \eqref{eq:proto-iota} lies in $\Sh_{\bL \cup T_{-\delta} \bL}(M \times (K,\infty))$.

Since $\cU$ is a \v Cech covering of $L_{>K}$, we have
\begin{equation}
\label{eq:second-reduction}
\Gamma(L_{>K}, \Brane_L) \cong \varprojlim_{\cU} \Brane_L(B \times I,\Omega)
\end{equation}
The limit on the right-hand side is equivalent to $\varprojlim_{\cV} \bigoplus_{\Omega} \Brane_L(B \times I,\Omega)$, where the sum is over the branches of $\bL$ over $B \times I$.  The two functors
\begin{equation}
\label{eq:sumomega-pre}
B \times I \mapsto \bigoplus_{\Omega} \Brane_L(B \times I,\Omega), \qquad B \times I \mapsto \Sh_{\bL}(B \times I)/\Loc(B \times I)
\end{equation}
are isomorphic in $\Fun(\cV^{\op},\St_{\bk})$ by the Proposition of \S\ref{subsec:multi-branch}.  The two functors
\begin{equation}
\label{eq:sumomega}
B \times I \mapsto \bigoplus_{\Omega} \Brane_L(B \times I,\Omega), \qquad B \times I \mapsto \cD_{\bL,\delta}(B \times I)
\end{equation}
are isomorphic by \eqref{eq:multi-branch-2}.  Now \eqref{eq:proto-iota} together with \eqref{eq:second-reduction} gives the fully faithful functor $\iota$ in \eqref{eq:first-reduction-prime}.
\end{proof}

\subsection{Nadler-Zaslow without Floer theory}\label{subsec:NZ without Fl}
We are now ready to construct the functor \eqref{eq:one}.  In this section we will prove that in \eqref{eq:two}, when $L \to T^* M$ is an embedding the left-hand arrow ($\mumon$) in \eqref{eq:two} is an equivalence, and that if $\Lambda$ is smooth the right-hand arrow ($\proj_{1,*}$) is fully faithful.  

By Theorem \ref{thm:312}, there is a $\delta > 0$ such that the functor $\mumon$ is isomorphic to the functor $T_{\deltazero}$.  By \S\ref{subsec:total-Tam}, if $\XY$ is an embedding then $T_{\deltazero}$ is isomorphic to $T_{\Czero}$ for any $C$ large or small, so to prove that $\mumon$ is an equivalence it suffices to find any $C$ such that $T_{\Czero}$ is an equivalence.  If $K$ is such that the restriction functor $\Sh_{\bL}^0(M \times \bR) \to \Sh_{\bL}^0(M \times (K,\infty))$ is an equivalence, then for any $C > |K|$ the composite
\[
\Sh_{\bL \cup T_{-C} \bL}^0(M \times \bR) \to \Sh_{\bL}^0(M \times (K,\infty)) \leftarrow \Sh_{\bL}^0(M \times \bR)
\]
is an inverse functor to $T_{\Czero}$.

Now let us prove that $\proj_{1,*}$ is fully faithful.  Fix $F$ and $G$ in $\Sh^0_{\bL}(M \times \bR;\bk)$.  We have an adjoint isomorphism
\[
\Hom(\proj_{1,*} F,\proj_{1,*} G) \cong \Hom(\proj_1^* \proj_{1,*} F,G)
\]
The sheaf $\proj_1^* \proj_{1,*}F$ is isomorphic to $\proj_{1,*} F \boxtimes \bk_{\bR}$, in particular it has singular support in $\Lambda \times \bR$.  As in the proof of Proposition \ref{prop:shift-down}, for $C_1 \gg C_2 \gg 0, K\ll 0$, there is a horizontally compactly supported homogeneous Hamiltonian flow  on $T^{*, \circ}(M \times \bR )$ which carries $(T_{C_1} \bL \cup \bL) \cap (T^* M \times (K,C_2))$ to $((\Lambda \times \bR) \cup \bL) \cap ((T^* M) \times (K,C_2))$, and this Hamiltonian flow induces an isomorphism
\[
\Hom(\proj_1^* \proj_{1,*} F,G) \cong \Hom(T_{C_1} F,G).
\]
But as $L \to T^* M$ is an embedding this is isomorphic to $\Hom(F,G)$ by Proposition \ref{prop:shift-down}(2).

\begin{example*}

If $L = \Gamma_{d\log(m)}$ is a standard Lagrangian, then $\bF$ is just the graph of $\log(m)$.  Every sheaf in $\Sh_{\bL}^0(M \times \bR)$ has the form $j_* F$, where $F$ is locally constant on $\{(x,t) \mid m(x) \geq t\}$.  The functor $\proj_{1,*}$ carries this to a standard sheaf on $M$, just as in the Nadler-Zaslow correspondence.  A more thorough comparison of \eqref{eq:two} and Nadler-Zaslow will appear elsewhere.
\end{example*}

\subsection{Hamiltonian invariance}
\label{subsec:ham-inv}
\newcommand{\bU}{\mathbf{U}}

Let $L_0$ and $L_1$ be two embedded lower exact Lagrangians whose wavefront maps $L_i \to M \times \bR$ are finite-to-one.  Let $\phi_t:T^* M \to T^* M$ be a Hamiltonian flow such that $\phi_1(L_0) = L_1$.  Then the suspension of $\phi_t$ is a Lagrangian concordance $b: \bR \times L \to T^* (\bR \times L)$.  We also assume that $b$ is lower exact and that its wavefront in $\bR \times M \times \bR$ is topologically a product over $\bR \times M \times (-\infty, K)$ for some $K \ll 0$.  For example, if $L_0$ and $L_1$ are the lower exact perturbations of two eventually conic exact Lagrangians $L'_0$ and $L'_1$, then we may find such a $b$ if $L'_0$ and $L'_1$ are Hamiltonian isotopic by a compactly supported isotopy.

Without any further loss of generality, we will also assume that $\phi_t = 1_{T^* M}$ for $t < 0$ and $\phi_t = \phi_1$ for $t > 1$.

By \eqref{eq:312}, the restriction functors
\[
\Sh^0_{\mathbf{b}}(\bR \times M \times \bR) \to \Sh^0_{\bL_i}(M \times \bR)
\]
are equivalences for $i = 0,1$.  The wavefront of $b$ may not be finite-to-one, but by a generic, horizontally compact supported Hamiltonian perturbation, we obtain another lower exact Lagrangian concordance $B:\bR \times L \to T^* (\bR \times L)$ that is finite to one --- we do not require that $B$ is the suspension of Hamiltonian isotopy.  Again by GKS, we have an equivalence
$\Sh^0_{\mathbf{b}}(\bR \times M \times \bR) \cong \Sh^0_{\mathbf{B}}(\bR \times M \times \bR)$ that is the identity outside a compact subset of $\bR \times M \times \bR$ --- we may assume it is the identity outside of $[0,1] \times M \times (K,\infty)$.  Let $\bU_0 = (-\infty,0) \times \bL_0$ and $\bU_1 = (1,\infty) \times \bL_1$.  Then the restriction maps
\begin{equation}
\label{eq:too-long}
\Sh^0_{\bU_0}((-\infty,0) \times M \times \bR) \leftarrow \Sh^0_{\mathbf{B}}(\bR \times M \times \bR) \to \Sh^0_{\bU_1}((1,\infty) \times M \times \bR)
\end{equation}
are equivalences.  By the topological triviality of $\mathbf{B}$ over $\bR \times M \times (-\infty,K)$, those equivalences are intertwined with
\begin{equation}
\Sh_{(-\infty,0) \times \Lambda}((-\infty,0) \times M) \leftarrow \Sh_{\bR \times \Lambda}(\bR \times M) \to \Sh_{(1,\infty) \times \Lambda}((1,\infty) \times M)
\end{equation}
by the pushforwards $\proj_{(-\infty,0) \times M,*}$, $\proj_{\bR \times M,*}$, $\proj_{(1,\infty) \times M,*}$ onto the first two factors.
As $\Brane$ is locally constant and the maps $L_i \to B$ are homotopy equivalences, the restriction functors 
\begin{equation}
\label{eq:too-long-two}
\Gamma((-\infty,0) \times L_0, \Brane_{B}) \leftarrow \Gamma(B, \Brane_B) \to \Gamma((1,\infty) \times L_1,\Brane_B)
\end{equation}
are also equivalences, and $\mumon$ intertwines \eqref{eq:too-long} and \eqref{eq:too-long-two}.

\subsection{Functoriality for pullbacks}
\label{subsec:pullback}
We deduced the Hamiltonian invariance of \eqref{eq:two} from the compatibilty of our constructions with restriction to open subsets of $M$.  There is a similar functoriality property for a more general class of maps.

Let $L, M$ continue to be as in \S\ref{subsec:exact-lags}, and $\bF, \bL$ as in \S\ref{subsec:bold-notation}.  Let $M'$ be another smooth manifold and let $f:M' \to M$ be smooth and proper.  
In that case put $L' = L \times_M M'$ --- it is a smooth manifold whose projection $L' \to f^{-1} T^* M \to T^* M'$ is a lower exact Lagrangian immersion.  The wavefront of $L'$ is the inverse image of $\bF \subset M \times \bR$ along  $f \times 1_{\bR}$ --- let us write this as $\bF' \subset M' \times \bR$ for short. 

When $\Lambda$ is smooth, there is a natural commutative diagram in $\St_{\bk}$
\[
\xymatrix{
\ar[d]  \Gamma(L,\Brane_L) & \ar[l] \Sh_{\bL}^0(M \times \bR) \ar[d]^{(f \times 1_{\bR})^*} \ar[r]^-{\proj_{1,*}} & \Sh_{\Lambda}(M) \ar[d]^{f^*} \\
\Gamma(L',\Brane_{L'}) & \ar[l] \Sh_{\bL'}^0(M' \times \bR) \ar[r]_-{\proj_{1,*}} & \Sh_{\Lambda'}(M')
}
\]
Or equivalently, after Theorem \ref{thm:312}, 
\[
\xymatrix{
\ar[d]  \Sh_{\bL \cup T_{-\delta} \bL}^0(M \times \bR) & \ar[l]_-{T_{\deltazero}} \Sh_{\bL}^0(M \times \bR) \ar[d]^{(f \times 1_{\bR})^*} \ar[r]^-{\proj_{1,*}} & \Sh_{\Lambda}(M) \ar[d]^{f^*} \\
\Sh_{\bL' \cup T_{-\delta} \bL'}^0(M' \times \bR) & \ar[l]^-{T_{\deltazero}} \Sh_{\bL'}^0(M' \times \bR) \ar[r]_-{\proj_{1,*}} & \Sh_{\Lambda'}(M')
}
\]
The commutativity of both squares follows from proper base change --- before applying it to the right square, compactify $M \times \bR$ to $M \times [-\infty,\infty)$ as in \eqref{eq:refine-this}.

\appendix
\section{More on contact transformations and proof of (\ref{eq:L-chiL})}\label{Appendix: contact}

Since the theory of contact transformations in \cite[\S 7.2]{KS} involves some general machineries that we are not trying to develop in the presentable setting here, we give independent proof of several results on contact transformations and analogue of microlocal cut-off lemmas using Tamarkin's category. A main consequence is a self-contained proof of (\ref{eq:L-chiL}). The key ingredients include the use of $\Omega$-lenses from \S\ref{subsec:MTONCS}, the theory of GKS \S\ref{subsec:GKS}, and calculations of stalks of microlocal sheaf categories using the contact Fourier transform. 

\subsection{Notations} \label{Appendix: Notations}

Let $V$ be an $r$-dimensional real vector space. Equip $V$ with a Euclidean inner product, and let $B_\epsilon\subset V$ be a standard ball of radius $\epsilon$ centered at $0$. We write $B$ for $B_1$.  
Let $T^{*, \geq 0}(V\times\bR)=\{(\widetilde{x};\widetilde{\xi}): \widetilde{\xi}(\partial_t)\geq 0\}$, and let $T^{*, <0}(V\times\bR)$ be the complement of $T^{*, \geq 0}(V\times\bR)$ in $T^*(V\times\bR)$. 

\subsubsection{Tamarkin's category}

From now on, we write $\Sh(X;\bk)$ for $\Shall(X;\bk)$. 
Let $\Sh_{\geq 0}(V\times\bR;\bk)$ be the full subcategory of $\Sh(V\times \bR;\bk)$ consisting of $F$ with $\SS(F)\subset T^{*,\geq 0}(V\times\bR)$. Let $\Sh^{<0}(V\times\bR;\bk)$ be the left orthogonal complement of $\Sh_{\geq 0}(V\times \bR;\bk)$, which is the essential image of the endofunctor $*\bk_{(-\infty, 0]}$ on $\Sh(V\times \bR;\bk)$. We refer this and the subsequent variant as Tamarkin's category. One can also take the right orthogonal complement of  $\Sh_{\geq 0}(V\times \bR;\bk)$, where both are canonically identified with the localization $\Sh(V\times \bR;\bk)/\Sh_{\geq 0}(V\times \bR;\bk)$. These definitions apply to any manifold $X$ in place of $V$. 

For any closed subset $C \subset T^{*,<0}(B\times \bR)/\bR_+\cong T^*B\times \bR$, let $\Sh^{<0}_{C}(B\times \bR;\bk)$ be the full subcategory of $\Sh^{<0}(B\times \bR;\bk)$ whose objects $F$ satisfy $\SS(F)\cap T^{*,<0}(B\times \bR)\subset \Cone(C)$. 

\subsubsection{A smooth Legendrian ball $\Omega$}\label{subsubsec: smooth Legendrian ball}
Let $\Omega$ be a smooth Legendrian ball in $T^*B\times \bR$ with a distinguished point $\Xi_0=(0,0;p_0,-dt)\in \Omega$ satisfying that 
\begin{itemize}
\item[(a)] The projection of $\Omega$ to $B\times \bR$ is finite-to-one, and the preimage of $\{0\}\times \bR$ is $\{\Xi_0\}$.

\item[(b)] The front projection $\bF$ of $\Omega$ is properly embedded in $B\times \bR$, and it is compatible with a Whitney stratification $\cS$ of $B\times \bR$ containing a 0-dimensional stratum $\{P_0=\XZ(\Xi_0)\}$. 

\item[(c)] We have $\bF\subset B\times I$ for a precompact interval $I$ in $\bR$. The distance squared function from $P_0$ (with respect to the fixed metric on $V$) has no $\cS$-critical value in $(0,\infty)$. 

\end{itemize}

\subsubsection{Standard contact transformations and kernels}\label{subsubsec: std contact tr}
Recall for any kernel $K\in \Sh(X\times Y;\bk)$, we have $K^{-1}\in \Sh(Y\times X;\bk)$ from (\ref{eq:K-inverse}), with a natural map $K^{-1}\circ K\to \Delta_{X,*}\bk$. Let $\chi: T^{*,<0}(V\times \bR)\to T^{*,<0}(V^*\times \bR)$ be any conic symplectomorphism. We say $\chi$ is \emph{standard} if $\chi$ is the negative conormal bundle (negative in $dt_1$) of the following hypersurface
\begin{align*}
H=\{t_1-t_0-p_1\cdot \bq_1+p_0\cdot \bq_0+\frac{1}{2}A_S(\bq_0)-\bq_0\cdot \bq_1=0\}, 
\end{align*}
where $\bq_0, \bq_1$ are vectors in $V$ and $V^*$, respectively, $p_0, p_1$ are given vectors in $V^*$ and $V$, and $A_S(\bq_0)$ is quadratic form in $\bq_0$. Given a conic Lagrangian $\bL$ containing $(0,0; p_0, -dt_0)$ in $T^{*, <0}(V\times\bR)$, if the above $\chi$ is a contact transformation for $(\bL, (0,0; p_0, -dt_0))$, then we say $\chi$ is a standard contact transformation for $(\bL, (0,0; p_0, -dt_0))$. Let $L$ be the Lagrangian projection of $\bL$, and let $A_{L_{(0;p_0)}}$ be the quadratic form on $\pi_*T_{(0;p_0)}L$ (which is naturally identified as a subspace in $V$) determined by $T_{(0;p_0)}L$. Then $\chi$ is a standard contact transformation for $(\bL, (0,0; p_0, -dt_0))$ if and only if $A_S|_{\pi_*T_{(0;p_0)}L}-A_{L_{(0;p_0)}}$ is nondegenerate.

By an affine transformation, $t_i\mapsto t_i-p_i\cdot \bq_i, i=0,1$, we may reduce to the case that $p_0=p_1=0$. In the following, for simplicity, we will always assume $p_0=p_1=0$. We say $\chi$ is a standard contact transformation for $\Omega$ if $\chi$ is one for every point in $\Cone(\Omega)$. By the assumptions on $\Omega$, if we have $\chi$ a standard contact transformation for $(\Cone(\Omega), \Xi_0)$, then for all $\epsilon>0$ sufficiently small, we have $\chi$ a standard contact transformation for $\Omega\cap (T^*B_\epsilon\times\bR)$. In the case that $A_S=0$, we denote the corresponding $\chi$ by $\chi_\sfF$, and call it the \emph{contact Fourier transform}, for the Lagrangian correspondence $\chi_\sfF$ represents the Fourier transform $T^*V\to T^*V^*$.

\subsubsection{Standard cone}\label{subsubsec: std cone}
Let $W$ be a vector space equipped with an Euclidean inner product. 
For any covector $\widetilde{\xi}_0\in T^{*,<0}(W\times\bR)|_{(0,0)}$ with $|\widetilde{\xi}_0|=1$, 
let $H_{\widetilde{\xi}_0}=\{(w, t)\in W\times \bR: \langle \widetilde{\xi}_0,(w, t)\rangle=-1\}$ and let $c\in H_{\widetilde{\xi}_0}$ be the orthogonal projection of the origin to $H_{\widetilde{\xi}_0}$. 
For any $\kappa>0$, let $\gamma_{\widetilde{\xi}_0, \kappa}$ be the proper closed cone whose intersection with $H_{\widetilde{\xi}_0}$ is the standard closed ball centered at $c$ of radius $1/\kappa$. We say $\gamma_{\widetilde{\xi}_0, \kappa}$ is the \emph{standard cone} associated to $\widetilde{\xi}_0$ with ratio $\kappa$.

\subsection{Basics on semi-orthogonal decompositions}
We collect some basic results on semi-orthogonal decompositions and apply to the relevant sheaf categories. 

Let $\cC\in\St_\bk$. Let $i':\cD\hookrightarrow \cC$ be the inclusion a cocomplete full subcategory (i.e. a full subcategory stable under (small) colimits in $\cC$). Let $j: \cC/\cD\hookrightarrow \cC$ be the inclusion of the right orthogonal complement of $\cD$. Then we have the adjoint functors
\begin{equation*}
\xymatrix{
\cD\ar@/^/[r]^{L'=i'}&\cC\ar@/^/[r]^L \ar@/^/[l]^{R'}&\cC/\cD\ar@/^/[l]^{R=j}
}
\end{equation*}
 and a fiber sequence in ${\bf End}(\cC)$ (the stable $\infty$-category of all exact endofunctors of $\cC$)
\begin{align*}
L'R'\to Id_\cC\to RL,
\end{align*}
where $R'$ (resp. $L$) is the right (resp. left) adjoint of $L'$ (resp. $R$). 

\begin{lemma}\label{lemma: L, L' limits}
If $L'$ preserves limits, then $L$ preserves limits as well. 
\end{lemma}
\begin{proof}
Since $L'$ preserves limits, so is $RL$. Now 
\begin{align*}
RL(\varprojlim_\alpha F_\alpha)\overset{\sim}{\to}\varprojlim_\alpha(RL(F_\alpha))\overset{\sim}{\leftarrow}R(\varprojlim_\alpha L(F_\alpha)). 
\end{align*}
Since $R$ is a fully faithful embedding, we have $L(\varprojlim_\alpha F_\alpha)\overset{\sim}{\to}\varprojlim_\alpha L(F_\alpha)$. 
\end{proof}

Now let $i_1':\cD_1\hookrightarrow \cC$ be the inclusion of a full subcategory that is closed under both colimits and limits in $\cC$. Let $j_1: \cD_1\backslash \cC\hookrightarrow \cC$ be the inclusion of the left orthogonal complement of $\cD_1$. Then we have the adjoint functors
\begin{equation*}
\xymatrix{
\cD_1\backslash \cC\ar@/^/[r]^{L_1=j_1}&\cC\ar@/^/[r]^{L_1'} \ar@/^/[l]^{R_1}&\cD_1\ar@/^/[l]^{R_1'=i_1'}
}
\end{equation*}
 and a fiber sequence of functors 
\begin{align*}
L_1R_1\to Id_\cC\to R_1'L_1',
\end{align*}
where $R_1$ (resp. $L_1'$) is the right (resp. left) adjoint of $L_1$ (resp. $R_1'$). 
In this setting, we also have the inclusion $R=j: \cC/\cD_1\hookrightarrow \cC$ and its left adjoint $L: \cC\to \cC/\cD_1$. 
We immediately have the following (see also \cite[Remark A.8.19, Proposition A. 8.20]{higher-algebra}): 

\begin{lemma}\label{lemma: left, right, orthogonal}
\begin{itemize}
\item[(i)] The composition $R_1\circ R: \cC/\cD_1\to \cD_1\backslash \cC$ gives the canonical equivalence between right and left orthogonal complements of $\cD_1$ in $\cC$. 

\item[(ii)] The inclusion $R=j:\cC/\cD_1\hookrightarrow\cC$ preserves limits and the inclusion $L_1=j_1: \cD_1\backslash \cC\hookrightarrow \cC$ preserves colimits. 

\item[(iii)] The projectors $R_1:\cC\to \cD_1\backslash \cC$ and $L:\cC\to \cC/\cD_1$ preserve both limits and colimits (in particular they are both $\bk$-linear), and they are identified through (i). 
\end{itemize}
\end{lemma}

For any open subset $\cU\subset T^*V\times \bR$, the inclusion $\Sh^{<0}_{(T^*V\times \bR)-\cU}(V\times \bR)\hookrightarrow \Sh^{<0}(V\times\bR)$ preserves both limits and colimits, so we can identify 
$$\Sh^{<0}(V\times\bR)/\Sh^{<0}_{T^*V\times \bR-\cU}(V\times \bR)$$ 
with either the left or the right orthogonal complement of $\Sh^{<0}_{T^*V\times \bR-\cU}(V\times \bR)$. 

\subsection{Notation}\label{notation: Sh<0cU}
For any open subset $\cU\subset T^*V\times \bR$, let $\Sh^{<0, \cU}(V\times \bR)$ denote the left orthogonal complement of $\Sh^{<0}_{T^*V\times \bR-\cU}(V\times \bR)$ in $\Sh^{<0}(V\times\bR)$. For any closed subset $C\subset \cU$, let $\Sh^{<0, \cU}_C(V\times \bR)$ denote the full subcategory of $\Sh^{<0, \cU}(V\times \bR)$ of sheaves $F$ such that $\SS(F)\cap \Cone(\cU)\subset \Cone(C)$. If $C$ is a closed subset of some open $\cU'\supset \cU$, we also write $\Sh^{<0, \cU}_C(V\times \bR)$ for $\Sh^{<0, \cU}_{C\cap \cU}(V\times \bR)$. 
Clearly, $\Sh^{<0, \cU}_C(V\times \bR)$ is closed under both colimits and limits in $\Sh^{<0, \cU}(V\times \bR)$. 

\begin{cor}\label{cor: proj cU_2, cU_1}
For any open $\cU_1\subset \cU_2\subset T^*V\times\bR$, the projector $P_{\cU_2,\cU_1}: \Sh^{<0, \cU_2}(V\times \bR)\to \Sh^{<0, \cU_1}(V\times \bR)$, i.e. the right adjoint of the inclusion $\Sh^{<0, \cU_1}(V\times \bR)\hookrightarrow \Sh^{<0, \cU_2}(V\times \bR)$ preserves both limits and colimits. 
\end{cor}
\begin{proof}
Since $\Sh^{<0, \cU_1}(V\times \bR)$ is identified with $\Sh^{<0, \cU_2}_{\cU_2-\cU_1}(V\times \bR)\backslash \Sh^{<0, \cU_2}(V\times \bR)$, the statement follows from Lemma \ref{lemma: left, right, orthogonal} (iii). 
\end{proof}

\subsection{Equivalences induced by $K_\chi\circ $}
In this subsection, we show several basic results about equivalences on (various versions of) Tamarkin's category induced by $K_\chi\circ$, for a standard conic symplectomorphism $\chi$. 

\begin{lemma}\label{lemma: KchiK-1chi}
Let $\chi:T^{*,<0}(V\times \bR)\to T^{*,<0}(V^*\times \bR)$ be a standard conic symplectomorphism. Let $K_\chi=\bk_H$. Then $K^{-1}_\chi=\mathrm{flip}(\Sigma^{r}\bk_{H})$, and the natural morphism $K_\chi^{-1}\circ K_\chi\to \Delta_{V\times \bR,*}\bk_{V\times\bR}$ fits into the fiber sequence 
\begin{align*}
j_!\cL\to \mathrm{flip}(\Sigma^r\bk_H)\circ \bk_H\to \Delta_{V\times\bR,*}\bk_{V\times\bR},
\end{align*}
for a local system $\cL$ on the open subset $\cU:=(V\times V-\Delta_V)\times \bR\times \bR$ and $j: \cU\hookrightarrow V\times \bR\times V\times \bR$ is the inclusion. In particular, $\SS(j_!\cL)^a\cap (T^{*,<0}(V\times \bR)\times T^{*}(V^*\times \bR))=\varnothing$. 
 
\end{lemma}
\begin{proof}
The first claim about $K^{-1}_\chi=\mathrm{flip}(\Sigma^{r}\bk_{H})$ is straightforward from definition. 

We calculate $K_\chi^{-1}\circ K_\chi$ as follows. Let $(\bq_i, t_i), i=0,1,2$, denote the elements in $V_i$, where $V_i=V$ if $i$ is even and $V_i=V^*$ otherwise. 
First, by a change of variables $t_j\mapsto t_j-\frac{1}{2}A_S(\bq_j), j=0, 2$, we may reduce to the case that $A_S=0$. 
Then 
\begin{align*}
&\mathrm{flip}(\Sigma^r\bk_H)\circ \bk_H=p_{02,!}\Sigma^r\bk_{Z}, \text{ for }\\
&Z=\{(\bq_0, t_0, \bq_1, t_1, \bq_2, t_2): t_1=t_0+\bq_0\cdot\bq_1=t_2+\bq_2\cdot \bq_1\}\\
&\subset V\times\bR\times V^*\times \bR\times V\times \bR. 
\end{align*}
and $p_{02}$ is the projection to the first and last copies of $V\times\bR$. 
By a direct calculation, we see the following:  
\begin{itemize}
\item[(i)] $ \Delta_{V\times\bR}^*p_{02,!}\Sigma^r\bk_{Z}\cong \bk_{V\times\bR}$, and the morphism $K_\chi^{-1}\circ K_\chi\to \Delta_{V\times \bR,*}\bk_{V\times\bR}$ is the standard one $p_{02,!}\Sigma^r\bk_{Z}\to  (\Delta_{V\times\bR})_*\Delta_{V\times\bR}^*(p_{02,!}\Sigma^r\bk_{Z})$.

\item[(ii)] Let  $G=\text{fiber}(K_\chi^{-1}\circ K_\chi\to \Delta_{V\times \bR,*}\bk_{V\times\bR})$, then $G|_{\{\bq_0=\bq_1; t_1,t_2\in \bR\}}\cong 0$. Since $G|_{\{\bq_0\neq \bq_1; t_1,t_2\in \bR\}}$ is a local system $\cL$, we have $G\cong j_!\cL$. 
\end{itemize}
The lemma now follows. 
\end{proof}

\begin{assumption}\label{assumption: chi, K_chi}
Let $\chi:T^{*,<0}(V\times \bR)\to T^{*,<0}(V^*\times \bR)$ be a standard conic symplectomorphism, and let $K_\chi=\bk_H$. 
\end{assumption}

\begin{lemma}\label{lemma: K_chi} 
Under Assumption \ref{assumption: chi, K_chi}, we have 
\begin{align*}
K_\chi \circ : \Sh(V\times \bR;\bk)/\Sh_{\geq 0}(V\times \bR;\bk)\overset{\sim}{\longrightarrow} \Sh(V^*\times\bR;\bk)/\Sh_{\geq 0}(V^*\times \bR;\bk): \circ K_{\chi}^{-1}
\end{align*}
give inverse equivalences of categories. 
\end{lemma}
\begin{proof}
We call any $T^{*,<0}(V\times\bR)$-lense a negative lense. Let $F\in \Sh^{<0}(V\times \bR;\bk)$ be the standard sheaf associated to a regular negative lense from \S \ref{subsubsec: std, Omega-lenses}. Since $\Supp(F)$ is compact, by the singular support estimates from \cite[Proposition 5.4.14, 5.4.4]{KS}, we know that $\SS(K_\chi\circ F)\subset \chi(\SS(F))\cup \zeta_{V^*\times \bR}$. Using Lemma \ref{lemma: KchiK-1chi}, we know that 
\begin{align*}
K_{\chi}^{-1}\circ K_\chi\in \End(\Sh_{SS(F)\cup \zeta_{V\times \bR}}(V\times \bR;\bk)/\Loc(V\times \bR;\bk))
\end{align*}
 is isomorphic to the identify functor through the canonical morphism $K^{-1}_{\chi}\circ K_\chi\to \Delta_{V\times\bR,*}\bk_{V\times\bR}$. Similarly, we have $K_{\chi}\circ K_\chi^{-1}$ isomorphic to the identity functor. Then the lemma follows from the generation property of $\Sh^{<0}(V\times \bR;\bk)$ by such $F$ under colimits. 
\end{proof}

\begin{lemma}\label{lemma: K_chi, <0, C}
Under Assumption \ref{assumption: chi, K_chi}, for any closed subset $C$ in $T^*V\times \bR$, we have 
\begin{align*}
K_\chi \circ : \Sh_{C}^{<0}(V\times \bR;\bk)\overset{\sim}{\longrightarrow} \Sh^{<0}_{\chi(C)}(V^*\times \bR;\bk)
\end{align*}
\end{lemma}
\begin{proof}
Using Lemma \ref{lemma: K_chi}, it suffices show that for any $F\in \Sh_{C}^{<0}(V\times \bR;\bk)$, $\SS(K_\chi\circ F)\cap T^{*, <0}(V^*\times \bR)$ is contained in $\Cone(\chi(C))$. Pick any regular subanalytic $(T^{*,<0}(V\times \bR)-\Cone(\chi(C)))$-lense $(f,\chi,a, \epsilon_0)$ with the standard sheaf $G$ associated to it, then 
\begin{align*}
\Hom_{\Sh^{<0}(V^*\times \bR;\bk)}(G, K_\chi\circ F)\cong \Hom_{\Sh^{<0}(V\times \bR;\bk)}(K_\chi^{-1}\circ G, F). 
\end{align*}
Let $G_t=\Cone(j_{t\epsilon_0,!}\bk_{H_{a, t\epsilon_0}^\dagger}\to j_{\epsilon_0,!}\bk_{H_{a, \epsilon_0}^\dagger})$, $t\in [0,1)$, then we have a compatible system $G_t\to G_{t'}$ for $t\leq  t'$ (induced from a non-negative contact Hamiltonian isotopy whose support is disjoint from $\chi(C)$) and $\varinjlim_{t\to 1} G_t\simeq 0$. We have 
\begin{align*}
&\Hom_{\Sh^{<0}(V^*\times \bR)}(G, K_\chi\circ F)\cong \Hom_{\Sh^{<0}(V\times \bR)}(K_\chi^{-1}\circ G, F)\\
&\overset{\mathrm{GKS}}{\cong} \varprojlim_{t\to 1} \Hom_{\Sh^{<0}(V\times \bR)}(K_\chi^{-1}\circ G_t, F)\cong \Hom_{\Sh^{<0}(V
\times \bR)}(K_\chi^{-1}\circ (\varinjlim_{t\to 1} G_t), F)\\
&=0. 
\end{align*}
The lemma then follows. 
\end{proof}

Using Notation \ref{notation: Sh<0cU}, Lemma \ref{lemma: K_chi, <0, C} is equivalent to saying that for any open subset $\cU\subset T^*V\times \bR$, $K_\chi \circ$ induces an equivalence
\begin{align}\label{eq: Sh<0,cU}
K_\chi\circ: \Sh^{<0, \cU}(V\times \bR)\overset{\sim}{\to} \Sh^{<0, \chi(\cU)}(V^*\times \bR). 
\end{align}
For $\cU=T^*B\times \bR$, we have an obvious equivalence $\Sh^{<0, T^*B\times \bR}(V\times \bR)\simeq \Sh^{<0}(B\times \bR)$.  Let $D^*V$ be the open unit co-disc bundle of $V$.  For any $\delta>0$, let $B_\delta^*\subset V^*$ denote the standard $\delta$-ball centered at $0$ in $V^*$. More generally, using the same proof as Lemma \ref{lemma: K_chi, <0, C}, we get the following.

\begin{prop}
Let $\cU\subset T^*V\times\bR$ be an open subset, and let $C\subset \cU$ be a closed subset, then under Assumption \ref{assumption: chi, K_chi}, we have 
\begin{align}\label{eq: Sh<0,cU}
K_\chi\circ: \Sh_{C}^{<0, \cU}(V\times \bR)\overset{\sim}{\to} \Sh_{\chi(C)}^{<0, \chi(\cU)}(V^*\times \bR). 
\end{align}
\end{prop}

\begin{lemma}\label{lemma: Sh_Omega, <0}
Let $X$ be a (possibly non-compact) smooth manifold. Let $\Omega$ be any properly embedded Legendrian in $T^*X\times \bR$ such that $\bF=\XZ(\Omega)\subset X\times I$ for an open interval $I\subset \bR$ that is bounded from above and $\XZ|_{\Omega}$ is proper. Then there is a natural equivalence 
\begin{align*}
\Sh_{\Cone(\Omega)\cup \zeta}(X\times \bR)/\Loc(X\times \bR)\simeq \Sh_{\Omega}^{<0}(X\times \bR).
\end{align*}
\end{lemma}
\begin{proof}
It suffices to show that for any $F\in \Sh_{\Omega}^{<0}(X\times \bR)$, $\SS(F)\cap T^{*, \geq 0, \circ}(X\times \bR)=\varnothing$. 
Let $U_{+,f}$ be the open subset given by $\{t>f(v)\}$ for a smooth function $f$ on $X$. If $U_{+,f}\cap \bF=\varnothing$, then the natural morphism $F\to j_{U_{+,f},*}j_{U_{+,f}}^*F$ must be the zero morphism, for $j_{U_{+,f},*}j_{U_{+,f}}^*F\in \Sh_{\geq 0}(X\times \bR)$. Hence $j_{U_{+,f}}^*F\cong 0$. Now for any $(v_0, t_0)\in X\times\bR$, choose $f$ such that $(v_0, t_0)\in U_{+,f}$ and $\overline{U_{+,f}\cap \bF}$ is compact (such $f$ clearly exists). 
It is clear that there exists a compactly supported contact Hamiltonian isotopy $\varphi_t$, with support contained in $T^{*, <0}(X\times \bR)/\bR_+$, such that $U_{+,f}\cap \XZ(\varphi_1(\Omega))=\varnothing$. Thus we have $\SS(F)\cap T^{*,\geq 0,\circ}(X\times \bR)|_{(v_0,t_0)}=\varnothing$, and the proof is complete.
\end{proof}

\begin{prop}\label{prop: K_chi, Modk}
Let $\Omega$ be an open Legendrian ball as in \S\ref{subsubsec: smooth Legendrian ball} and assume that  $p_0=0$. Let $\chi$ be a standard contact transformation for $\Omega$. We have 
\begin{align}
\label{eq: K_chi, B, D}&K_\chi\circ : \Sh^{<0}_{\Omega}(B\times \bR)\overset{\sim}{\longrightarrow} \Sh^{<0, D^*V^*\times \bR}_{\chi(\Omega)}(V^*\times \bR). 
\end{align}
and there are natural equivalences 
\begin{align}
\label{eq: chiOmega, Modk}&\Sh^{<0, D^*V^*\times \bR}_{\chi(\Omega)}(V^*\times \bR)\simeq \Mod(\bk)\simeq \Sh^{<0}_{\chi(\Omega)}(B_{\delta}^*\times \bR),
\end{align}
for $\delta>0$ sufficiently small. 
\end{prop}

\begin{remark*}
In the rest of the appendix, we will assume $p_0=0$ for $\Omega$ unless otherwise specified. In the case that $p_0\neq 0$, all results will hold, just with all $B_{\delta}^*$ replaced by $B_{\delta}^*(p_0)=B_\delta^*+p_0:=$the shift of $B_{\delta}^*$ by $p_0$ in $V^*$.  
\end{remark*}

The first equivalence in the proposition follows from (\ref{eq: Sh<0,cU}) and that a standard contact transformation sends $T^*B\times \bR$ to $D^*V^*\times \bR$. In the rest of the appendix, we develop several tools to prove the second part of the proposition. Let $D_\epsilon^*V^*$ be the open $\epsilon$-codisc bundle consisting of covectors of length $<\epsilon$. By assumption, the Lagrangian projection of $\chi(\Omega)$, denoted by $\chi(\Omega)_\pi$, is a (properly embedded) Lagrangian graph in $D^*V^*$ such that the function $|\xi|^2$ on $\chi(\Omega)_\pi$ is proper over $[0,1)$ and is regular except at $(0;0)$.

Let 
\begin{align}\label{eq: Phi', ep, delta, localization}
\Phi': \Sh^{<0, D_\epsilon^*V^*\times\bR}(V^*\times \bR)\to \Sh^{<0, D_{\epsilon}^*B_{\delta}^*\times \bR}(V^*\times \bR) 
\end{align}
be the localization functor, for any $(0, +\infty]\times (0,1)\ni(\delta,\epsilon)$ with $\delta\gg \epsilon$, $\epsilon\ll 1$, it clearly induces the functor 
\begin{align*}
\Phi: \Sh^{<0, D_\epsilon^*V^*\times\bR}_{\chi(\Omega)}(V^*\times \bR)\to \Sh^{<0, D_{\epsilon}^*B_{\delta}^*\times \bR}_{\chi(\Omega)}(V^*\times \bR). 
\end{align*}
Similarly,  by applying $K_{\chi}^{-1}\circ$, we have the following diagram (the square involving only $\widehat{\Phi}'$,  $\widehat{\Phi}$ as vertical functors is commutative):
\begin{equation}\label{diagram: widehatPhi}
\xymatrix{\Sh^{<0}(B_\epsilon\times\bR)\ar[dd]^{\widehat{\Phi}'}&\Sh^{<0}_{\Omega}(B_\epsilon\times\bR)\ar@{_{(}->}[l]_{\iota}\ar[dd]^{\widehat{\Phi}}\\
&&\\
\Sh^{<0, D_\delta^*B_\epsilon\times\bR}(B_\epsilon\times\bR)\ar@/^1pc/[uu]^{L'}\ar@/_1pc/[uu]_{R'}&\Sh^{<0,D_\delta^*B_\epsilon\times\bR}_{\Omega}(B_\epsilon\times\bR)
\ar@/^1pc/[uu]^{L}\ar@/_1pc/[uu]_{R}\ar@{_{(}->}[l]_{\iota_\delta}
}
\end{equation}
where
$\widehat{\Phi}'$ is the localization functor, $\widehat{\Phi}$ is the restriction of $\widehat{\Phi}'$, and $L', R', L, R$ are the respective left and right adjoint of $\widehat{\Phi}'$ and $\widehat{\Phi}$; $L'$ is just the obvious inclusion functor.

\subsection{Results on projectors}\label{subsec: results projectors}
We first prove some basic facts about the functor $\widehat{\Phi}'$ in diagram (\ref{diagram: widehatPhi}). We will work in a slightly more general setting.

Let $W$ be any finite dimensional vector space with a Euclidean inner product.
Consider the \emph{not} necessarily commutative diagram 
\begin{align*}
\xymatrix{
\Sh^{<0}(W\times\bR)\ar[r]^{P_{\epsilon}\ \ \ \ \ \ \ }\ar[d]_{j^!}& \Sh^{<0}_{(T^*W-D_\epsilon^*W)\times\bR}(W\times\bR)\ar[d]^{j^!}\ar@/^/[l]^{i_\epsilon}\\
\Sh^{<0}(B_\delta\times\bR)\ar[r]^{P_{\epsilon, \delta}\ \ \ \ \ }&\Sh^{<0}_{(T^*B_{\delta}-D_\epsilon^*B_\delta)\times\bR}(B_\delta\times\bR)\ar@/^/[l]^{i_{\epsilon,\delta}}
}
\end{align*}
where $i_\epsilon$ (reps. $i_{\epsilon, \delta}$) is the natural full embedding, and $P_\epsilon$ (resp. $P_{\epsilon, \delta}$) is the left adjoint, and $j: B_\delta\times \bR\hookrightarrow W\times\bR$ is the open inclusion. Since $j^!\circ i_\epsilon\cong i_{\epsilon, \delta}\circ j^!$, using adjunction we have a natural transformation 
\begin{align}\label{tau: P_ep,delta to P_ep}
\tau_P: P_{\epsilon, \delta}\circ j^!\to j^!\circ P_\epsilon 
\end{align}
(clearly, $P_{\epsilon, \delta}\circ j^!$ represents the localization functor that is left adjoint to $j_*\circ i_{\epsilon,\delta}$). Similarly, we have 
\begin{align}\label{diagram: rho_epsilon,delta}
\xymatrix{
\Sh^{<0}(W\times\bR)\ar[r]_{\rho_{\epsilon}\ \ \ \ \ \ \ }& \Sh^{<0, D_\epsilon^*W\times\bR}(W\times\bR)\ar@/_/[l]_{i'_\epsilon}\\
\Sh^{<0}(B_\delta\times\bR)\ar[r]_{\rho_{ \epsilon, \delta}\ \ \ \ \ }\ar[u]^{j_!}&\Sh^{<0, D_\epsilon^*B_\delta\times\bR}(B_\delta\times\bR)\ar@/_/[l]_{i'_{\epsilon, \delta}}\ar[u]_{j_!}
},
\end{align}
where $\rho_\epsilon$ (resp. $\rho_{\epsilon, \delta}$) is the right adjoint of the inclusion $i_\epsilon'$ (resp. $i_{\epsilon, \delta}'$) and there is a natural transformation $j_!\circ \rho_{\epsilon, \delta}\to \rho_\epsilon\circ j_!$. Moreover, applying the unit and co-unit of the adjoint pair $(j_!, j^!)$, we get a natural transformation $\tau_\rho: \rho_{\ep, \delta} j^!\to j^!\rho_\ep$. 

\begin{lemma}\label{lemma: fiber seq, rhoepdelta}
The following diagram naturally commutes
\begin{align}\label{lemma: diagram, unit, tau_P}
\xymatrix{ j^! \ar[r]^{\mathrm{unit}\ \ \ }\ar[drr]_{\mathrm{unit}} & i_{\ep,\delta} P_{\ep, \delta} j^! \ar[r]^{\tau_P} & i_{\ep, \delta}j^!P_\ep\ar[d]^{\sim}\\
& &j^! i_\ep P_\ep
}
\end{align}
where the two morphisms marked by $\mathrm{unit}$ are the maps (induced from) the unit map of the respective adjoint pairs, $\tau_P$ is as in (\ref{tau: P_ep,delta to P_ep}), and the vertical isomorphism is the obvious one as above. 

In particular, we get a natural commutative diagram with each row a fiber sequence: 
\begin{align}\label{lemma: diagram fiber seq, rhoepdelta}
\xymatrix{
\rho_{\ep,\delta} j^!\ar[r]\ar[d]_{\tau_\rho}& j^!\ar[r]\ar@{=}[d]& P_{\ep,\delta} j^!\ar[d]^{\tau_P}\\
j^!\rho_{\ep}\ar[r]& j^!\ar[r]& j^!P_{\ep}
}
\end{align}
where for simplicity, we write $P_{\ep,\delta}$ for $i_{\ep, \delta}P_{\ep,\delta}$ and similarly for others. 
\end{lemma}

\begin{proof}
To see that (\ref{lemma: diagram, unit, tau_P}) is commutative, it suffices to observe the following commutative diagram
\begin{align*}
\xymatrix{&i_{\ep, \delta}P_{\ep, \delta}j^!\ar[r]&i_{\ep, \delta}P_{\ep, \delta}j^! i_\ep P_\ep\ar[r]^\sim&i_{\ep, \delta}P_{\ep, \delta} i_{\ep,\delta} j^! P_\ep\ar[r]_\sim^{\ \ \ \counit}& i_{\ep, \delta} j^! P_\ep\ar[d]^{\sim}\\
j^!\ar[ur]\ar[r]&j^!i_\ep P_\ep\ar@{=}@/_1pc/[rrr]\ar[ur]^{\sim}_{\alpha}& i_{\ep, \delta} j^! P_\ep\ar[l]_{\sim}\ar@{=}[urr]\ar[ur]^{\unit} & &j^!i_\ep P_\ep
},\\
\end{align*} 
where (1) the left square has all edges the unit maps for the respective adjoint pairs, and the marked $\counit$ on the right-hand-side means (the map induced from) the co-unit map for the adjoint pair $(i_{\ep,\delta}, P_{\ep, \delta})$; (2) the unit map $\alpha$ is an isomorphism because $j^! i_\ep P_\ep$ has essential image lying in $\Sh^{<0}_{(T^*B_\delta-D_\ep^*B_\delta)\times\bR} (B_\delta\times \bR)$; (3) it is clear that all chambers are naturally commutative. 

The commutative diagram (\ref{lemma: diagram fiber seq, rhoepdelta}) is a direct consequence of the above (and a similar diagram as (\ref{lemma: diagram, unit, tau_P}) with $i_\ep, P_\ep, i_{\ep,\delta},P_{\ep, \delta}$ replaced by $i_\ep', \rho_\ep, i_{\ep,\delta}',\rho_{\ep, \delta}$).

\end{proof}

Let $(x,a)\in B_\delta\times\bR$. Let $C_{x, a, \epsilon}$ be the standard closed cone in $W\times\bR$ with vertex (or cone point) $(x, a)$ consisting of half rays emitting from $(x, a)$ in the directions $(v, -\partial_t)$ for $|v|\leq \epsilon^{-1}$. Equivalently, $C_{x,a,\epsilon}=(x,a)+\gamma_{(0,1),\epsilon}$ in terms of the standard cone $\gamma_{(0,1),\epsilon}$ from \S\ref{subsubsec: std cone} (see Figure \ref{figure: C_x,a,epsilon} below). Let $C'_{x, a, \epsilon}=C_{x, a, \epsilon}-(\{x\}\times (-\infty, a])$. 
\begin{figure}[htbp]
\begin{center}
\begin{tikzpicture}[scale=0.8]
\draw[thick] (0,0) node[above] {$(x,a)$}--(-2,-2);
\draw[thick] (0,0)--(2,-2);
\draw[thick, gray] (0,0)--(-6,-2);
\draw[thick, gray] (0,0)--(6,-2);
\draw[gray, rounded corners=3 pt] (-2.1,-0.7)--(-1.5,-0.5) to [out=40,in=140] (1.5,-0.5)--(2.1,-0.7);
\draw[gray, dashed] (-4.2,-1.4) to [out=15, in=165] (4.2,-1.4); 
\draw[gray, dashed] (-5.4,-1.8) to [out=15, in=165] (5.4,-1.8); 
\fill[gray, opacity=0.5] (-2,-2)--(0,0)--(2,-2);
\end{tikzpicture}
\caption{The filled cone is $C_{(x,a,\epsilon)}$. The outer edges emitting from $(x,a)$ form the boundary of $C_{x,a,\epsilon'}$ for $0<\epsilon'<\epsilon$. The portion of gray curve above $(x,a)$ is part of $\partial H_t^{\epsilon'}$ as a smoothing of $\partial C_{x,a,\epsilon'}$. 
}\label{figure: C_x,a,epsilon}
\end{center}
\end{figure}

\begin{lemma}\label{lemma: rho_delta, P_delta}
Let $(x,a)\in B_\delta\times\bR$, $U=B_\delta\times (a-\epsilon(\delta-|x|), \infty)$, and $j_U: U\hookrightarrow B_\delta\times\bR$ (resp. $j'_U: U\hookrightarrow W\times\bR$) be the open inclusion. Let $G=\bk_{\{x\}\times(-\infty, a]}$, then the following hold. 
\begin{itemize}
\item[(i)]  We have the fiber sequence
\begin{align*}
\rho_\epsilon(G)=\bk_{C_{x, a, \epsilon}}\to G \to \Sigma\bk_{C'_{x, a, \epsilon}}=P_\epsilon(G). 
\end{align*}

\item[(ii)] There are natural isomorphisms
\begin{align*}
&P_{\epsilon, \delta}(G)|_U\overset{\sim}{\to} (P_\epsilon(G))|_U,\\
&\rho_{\epsilon, \delta}(G)|_U\overset{\sim}{\to} \rho_\epsilon (G)|_U. 
\end{align*}

\end{itemize}

\end{lemma}

\begin{proof}
(i) It is clear that 
\begin{align*}
&\SS( \Sigma\bk_{C'_{x, a, \epsilon}})\cap \Cone(D_\epsilon^*W\times\bR)=\varnothing
\end{align*}
It suffices to show that 
\begin{align*}
&\bk_{C_{x, a, \epsilon}}\in \Sh^{<0,D^*_\epsilon W\times\bR}(W\times\bR).
\end{align*}
First, we have $\bk_{C_{x, a, \epsilon}}=\varinjlim_{0<\epsilon'<\epsilon}\bk_{C_{x,a,\epsilon'}}$, so it suffices to show that 
$$\bk_{C_{x,a,\epsilon'}}\in \Sh^{<0,D^*_\epsilon W\times\bR}(W\times\bR), \text{ for all }0<\epsilon'<\epsilon.$$ For each $\bk_{C_{x,a,\epsilon'}}$, we can do a smoothing of the boundary of the cone from the outside as indicated in Figure \ref{figure: C_x,a,epsilon}, so that we get a family $\bk_{H^{\epsilon'}_t}, t\in (0,1)$ satisfying 
\begin{itemize}
\item $H^{\epsilon'}_t$ is the closed region below a smooth hypersurface with the negative conormal vectors of $\partial H^{\epsilon'}_t$ contained in $\Cone(D^*_\epsilon W\times\bR)$;

\item $H^{\epsilon'}_t$ decreases as $t$ decreases, and $\varinjlim_{t\to 0}\bk_{H_t^{\epsilon'}}\cong \bk_{C_{x,a,\epsilon'}}$. 
\end{itemize}
This further reduces the problem to show that $\bk_{H_t^{\epsilon'}}\in \Sh^{<0,D^*_\epsilon W\times\bR}(W\times\bR)$. But this is straightforward, for example, we can realize $\bk_{H_t^{\epsilon'}}$ as the colimit of a family of standard sheaves for the $(D^*_\epsilon W\times\bR)$-lenses whose underling regions are cut out by the dashed smooth hypersurfaces and $\partial H_t^{\epsilon'}$ indicated in the same figure.

(ii) Let $I=(a-\epsilon(\delta-|x|), \infty)$. Note that $a-\epsilon(\delta-|x|)$ is the maximum of the $\bR$-coordinate in $\partial C_{x, a, \epsilon}\cap \partial B_\delta\times\bR$. 
Since 
\begin{align*}
&\SS( (j_{U}')^!\Sigma\bk_{C'_{x, a, \epsilon}})\cap \Cone(D_\epsilon^*B_\delta\times I)=\varnothing,\\
&(j_U')^!\bk_{C_{x, a, \epsilon}}\in \Sh^{<0,D^*_\epsilon B_\delta\times I}(U),
\end{align*}
we have 
\begin{align}\label{eq: G|U, Pepsilon|U}
\rho_{\epsilon,U}(G|_U):=\rho_\epsilon(G)|_U\to G|_U\to P_\epsilon(G)|_{U}=:P_{\epsilon,U}(G|_U)
\end{align}
is the semi-orthogonal decomposition of $G|_U\in \Sh^{<0}(U)$ with respect to $\Sh^{<0,D^*_\epsilon B_\delta\times I}(U)$ and its right orthogonal complement.
This implies that the composition
\begin{align*}
P_{\epsilon,U}(G|_U)\to P_{\epsilon,\delta}(G)|_U\to P_\epsilon(G)|_U
\end{align*}
is an isomorphism. 
By Lemma \ref{lemma: Q, I, cE} below, $P_{\epsilon,U}(G|_U)\to P_{\epsilon,\delta}(G)|_U$ is an isomorphism, hence  $P_{\epsilon,\delta}(G)|_U\overset{\sim}{\to} P_\epsilon(G)|_{U}$ and part (ii) of the lemma follows. 
\end{proof}

\begin{lemma}\label{lemma: Q, I, cE}
Let $Q\subset W$ be an open subset and let $(a,\infty)=I\subset \bR$. The functor $j_{Q\times I}^*$ sends $\Sh^{<0, D_\epsilon^*Q\times\bR}(Q\times \bR)$ to $\Sh^{<0, D_\epsilon^*Q\times I}(Q\times I)$. In particular, we have natural isomorphisms $P_{\ep, U}j_U^!\overset{\sim}{\to} j_U^!P_{\ep, \delta}$ and $\rho_{\ep, U}j_U^!\overset{\sim}{\to} j_U^!\rho_{\ep, \delta}$, under the assumption of Lemma \ref{lemma: rho_delta, P_delta}. 
\end{lemma}

\begin{proof}
Since $\Sh^{<0, D_\epsilon^*Q\times\bR}(Q\times \bR)$ is generated under colimits by the standard sheaves associated to $D_\ep^*Q\times\bR$-lenses in $\Sh(Q\times\bR)$, and the inclusion $\Sh^{<0, D_\epsilon^*Q\times\bR}(Q\times \bR)\hookrightarrow \Sh(Q\times\bR)$ preserves colimits, it suffices to show that $j_{Q\times I}^*$ sends every such standard sheaf to a sheaf in $\Sh^{<0, D_\epsilon^*Q\times I}(Q\times I)$. In the following we choose a slightly different collection of generators of $\Sh^{<0, D_\epsilon^*Q\times\bR}(Q\times \bR)$ under colimits. 

For any covector $(x_0,b_0;\widetilde{\xi}_0)\in D_\epsilon^*Q\times \bR$, let $\gamma^\circ$ be the interior of the standard cone $\gamma:=\gamma_{\widetilde{\xi}, \kappa}$ for some fixed $0<\kappa\ll \epsilon$ and $\widetilde{\xi}$ in a small open neighborhood  $\cV_0$ of $\widetilde{\xi}_0$. 
Let $U_1\subset \overline{U}_1\subset Q\times\bR$ be a small open neighborhood of $(x_0,b_0)$. For any $\delta>0$, let 
\begin{align*}
H_{\geq -\delta}=\{(x,t): \langle\widetilde{\xi},(x-x_0, t-b_0)\rangle\geq -\delta\},
\end{align*}
and let $H_{<-\delta}$ be its complement. We can choose $U_1$, $\delta$ and $\cV_0$, so that  $(U_1+\gamma)\cap H_{\geq -\delta}\subset Q\times \bR$, for all such choices of $\gamma$. Denote $((x,b)+\gamma^\circ)\cap H_{\geq -\delta}$ by $R^\gamma_{(x,b);\delta}$, for $(x, b)\in U_1$. 
Then by \cite[Prop. 5.1.1]{KS} and its proof, 
the collection of $\bk_{R^\gamma_{(x,b);\delta}}$, for such choices of $(x,b;\widetilde{\xi}), \kappa, \delta$, generates $\Sh^{<0, D_\epsilon^*Q\times\bR}(Q\times \bR)$ under colimits. Write $R$ for $R^\gamma_{(x,b);\delta}$.

There are the following cases for $(j_{Q\times I})^* \bk_{R}$:
\begin{itemize}
\item[(1)] If $R\subset (Q\times I)$, then $(j_{Q\times I})^* \bk_{R}=\bk_{R}$;

\item[(2)] If $R\cap (Q\times I)=\varnothing$, then  $(j_{Q\times I})^* \bk_{R}=0$;

\item[(3)] If $\varnothing\neq R\cap (Q\times I)\subsetneq R$, set $R^+:=R\cap (Q\times(a,\infty))$. 
Then we have the situation illustrated in Figure \ref{figure: Rplusminusx,t} (of course the figure doesn't illustrate all possible configurations), and $(j_{Q\times I})^* \bk_{R}= \bk_{R^+}$. Since $R^+$ (as a convex domain) has the non-solid boundary, i.e. $\overline{R^+}-R^+$, consists of piecewise smooth hypersurfaces whose outer conormals are contained in $\Cone(D_\ep^*Q\times I)$ (and its solid boundary is part of the solid boundary of $R$), we know that $\bk_{R^+}$ can be written as a colimit of standard sheaves associated to $D_\ep^*Q\times I$-lenses. Hence $\bk_{R^+}\in  \Sh^{<0, D_\epsilon^*Q\times I}(Q\times I)$. 
\end{itemize}

\begin{figure}[htbp]
\begin{center}
\begin{tikzpicture}
\draw[thick, gray] (-3,0)--(3,0) node[right] {$Q\times\{a\}$};
\draw[thick] (-2.4,0.8)--(2.4,-0.8);
\draw[thick, dashed] (-2.4, 0.8)--(-0.4,-1.2)--(2.4,-0.8);
\draw[dashed, thick] (-0.12,0.05)--(-1.65,0.05);
\draw[thick] (0.12,-0.05)--(-1.55,-0.05);
\fill[gray, opacity=0.5] (-2.4, 0.8)--(-0.12,0.05)--(-1.65,0.05);
\fill[gray, opacity=0.3] (-1.55,-0.05)--(0.12,-0.05)--(2.4,-0.8)--(-0.4,-1.2);
\end{tikzpicture}
\end{center}
\caption{The entire triangle illustrates $R$ in case (3). The upper triangle illustrates $R^+$.}\label{figure: Rplusminusx,t}
\end{figure}

The proof is then complete. 
\end{proof}

\subsection{A ``cut-off" lemma and applications}
We prove a ``cut-off" lemma that will play a similar role as the refined microlocal cut-off lemma in \cite{KS}, after applying the contact Fourier transform. We then deduce a couple of useful corollaries from it. For any open $U$ in a smooth manifold, we write $T^*\overline{U}$ for $T^*M|_{\overline{U}}$. 

\begin{lemma}[``Cut-off" lemma]\label{lemma: Q, limits, colimits}
Let $Q\subset W$ be an open subset and $I=(c,d)\subset \bR$ be an open interval. Let $F\in \Sh^{<0}(W\times\bR)$. Assume $\SS(F)\cap T^{*,<0}(W\times\bR)\subset T^{*}(\overline{Q\times I})$, then $F$ is generated by 
\begin{align}\label{eq: lemma M_x,a}
M_{\{x\}\times(-\infty, a]}=T_{(-\infty,0]}(M_{(x,a)}), (x,a)\in \overline{Q}\times \overline{I}, M\in \Mod(\bk), 
\end{align}
under colimits (resp. limits). 
\end{lemma}

\begin{proof}
In the proof, we assume that $(c,d)$ is a bounded open interval. The case when $c=-\infty$ and/or $d=\infty$ is easier and follows from the same argument. 

Let $\cC=\cC_{\overline{Q}\times \overline{I}}\subset \Sh^{<0}(W\times\bR)$ be the full subcategory that is generated by the sheaves in (\ref{eq: lemma M_x,a}) with $(x,a)\in \overline{Q}\times [c,d]$ under colimits. Let $\rho_{\cC}: \Sh^{<0}(W\times\bR)\to \cC$ be the right adjoint of the inclusion. By the singular support assumption on the sheaf $F$, we have 
\begin{align}\label{eq: lemma Hom M_x, l, u}
\Hom(M_{\{x\}\times (\ell, u]}, F)=0, \forall (\{x\}\times (\ell, u])\cap \overline{Q\times I}=\varnothing, M\in \Mod(\bk), 
\end{align}
where $u\in\bR, \ell\in [-\infty, u)$. 

Conversely, the equation (\ref{eq: lemma Hom M_x, l, u}) for any sheaf $F'\in \Sh^{<0}(W\times\bR)$ (in place of $F$) implies that 
\begin{align}\label{eq: SSF', Q', I'}
\SS(F')\cap T^{*,<0}(W\times\bR)\subset T^*(\overline{Q\times I}).
\end{align}
Indeed, for any $(x_0,a)\not\in \overline{Q\times I}$, any covector $\widetilde{\xi}_0\in T^{*,<0}(W\times\bR)|_{(0,0)}$ with $|\widetilde{\xi}_0|=1$, and $\epsilon>0$ small, let 
\begin{align*}
H'_{\geq -\epsilon}=\{(x, t): \langle \widetilde{\xi}_0,(x-x_0, t-a)\rangle \geq -\epsilon\},\  H'_{-\epsilon}=\{(x, t): \langle \widetilde{\xi}_0, (x-x_0, t-a)\rangle=-\epsilon\}
\end{align*}
Let $\gamma^\circ$ be the interior of the standard cone $\gamma:=\gamma_{\widetilde{\xi}_0, \kappa}$ (as in \S\ref{subsubsec: std cone}) for some $\kappa>0$. Choose any $\epsilon>0$ and $\kappa>0$, so that
\begin{itemize}
\item $(((x_0,a)+\gamma)\cap H_{\geq -\epsilon})\cap \overline{Q\times I}=\varnothing$;

\item The fiber of $\pi_W|_{((x_0,a)+\gamma^\circ)\cap H_{\geq -\epsilon}}\to W$ over $x\in W$ is either $\varnothing$ or $\{x\}\times (\ell, u]$ for some $\ell<u$, where $\pi_W: W\times \bR\to W$ is the obvious projection.
\end{itemize}
Let $i: ((x_0,a)+\gamma^\circ)\cap H_{\geq -\epsilon}\hookrightarrow W\times\bR$ be the inclusion. By base change and using (\ref{eq: lemma Hom M_x, l, u}), we have  $(\pi_W)_*i^!F'= 0$. Therefore, $\Hom(\bk_{((x_0,a)+\gamma^\circ)\cap H_{\geq -\epsilon}}, F')= 0$. By a similar argument as in the proof of \cite[Prop. 5.1.1]{KS}, we see that (\ref{eq: SSF', Q', I'}) holds. 

It is clear that for any $F'\in \cC$, it satisfies (\ref{eq: SSF', Q', I'}). In particular, that holds for $F'=\rho_\cC(F)$. 
Let $G:=\Cone(\rho_{\cC}F\to F)$. Since $G$ is right orthogonal to $\cC$ and $G$ satisfies (\ref{eq: SSF', Q', I'}) with $F'$ replaced by $G$, we see that $\Hom(\bk_{(a,b]}, G)=0$ for any $-\infty\leq a<b< \infty$, hence $G=0$ in $\Sh^{<0}(W\times\bR)$, and $F\in \cC$.

For the ``dual" version about generation under limits, take the obvious equivalence $\Psi: \Sh^{<0}(W\times\bR)\overset{\sim}{\to} \Sh(W\times\bR)/\Sh_{\geq 0}(W\times\bR)$, where the RHS is understood as the right orthogonal complement of $\Sh_{\geq 0}(W\times\bR)$. Then one just needs to run a dual version of the above argument with $\cC$ replaced by the full subcategory of $\Sh(W\times\bR)/\Sh_{\geq 0}(W\times\bR)$ generated by  
\begin{align*}
\Psi(M_{\{x\}\times(-\infty, a]})=M_{\{x\}\times(a, \infty)},
\end{align*} 
for $(x,a)\in \overline{Q}\times [c, d], M\in \Mod(\bk)$, under limits. In this case, one replaces (\ref{eq: lemma Hom M_x, l, u}) by 
\begin{align*}
\Hom(F, M_{\{x\}\times (\ell, u]})=0, \forall (\{x\}\times [\ell, u))\cap \overline{Q\times I}=\varnothing, M\in \Mod(\bk), 
\end{align*}
where $\ell\in\bR, u\in (\ell,\infty]$. We omit the details for the rest of the steps since they are completely similar as before. 
\end{proof}

Recall the notations from \S\ref{subsec: results projectors}.

\begin{figure}[htbp]
\begin{center}
\begin{tikzpicture}[scale=1.5]
\filldraw[fill=gray] (-0.5, 0.3)--(0.5,0.3)--(0.5,-0.3)--(-0.5,-0.3)--(-0.5, 0.3);
\draw[thick, dashed] (-2,1)--(-2,-2);
\draw[thick, dashed] (2,1)--(2,-2);
\draw[dashed] (-2, -0.3)--(2,-0.3);
\draw[dashed] (-2, 0.3)--(2,0.3);
\draw (-0.5, 0.3)--(-2, -0.5);
\draw (0.5, 0.3)--(2, -0.5);
\draw(0.5,0.3) -- (0.7,0.3) node[right, xshift=0.1cm, yshift=-0.1cm] {\tiny{$\arctan \epsilon$}} arc (0:-29:0.2) -- cycle;
\draw [decorate,decoration={brace,amplitude=2pt},xshift=-3pt,yshift=0pt]
(-2,-0.3) -- (-2,0.3) node [black,midway,xshift=-1cm] 
{\footnotesize $I_1=(a_1,b_1)$};
\draw [decorate,decoration={brace,amplitude=2pt, mirror},xshift=0pt,yshift=-2pt]
(-0.5,-0.3) -- (0.5,-0.3) node [black,midway,yshift=-0.3cm] 
{\footnotesize $B_{\eta_1}$};
\draw [decorate,decoration={brace,amplitude=2pt, mirror},xshift=0pt,yshift=-3pt]
(-2,-2) -- (2,-2) node [black,midway,yshift=-0.3cm] 
{\footnotesize $B_{\delta}$};
\end{tikzpicture}
\caption{An illustration of the assumptions in Corollary \ref{cor: F,G,a1b1,eta1eta2}. }\label{figure: cor: F,G,a1b1,eta1eta2}
\end{center}
\end{figure}

\begin{cor}\label{cor: F,G,a1b1,eta1eta2}
Let $F, G\in \Sh^{<0}(W\times\bR)$. Let $0<\eta_1<\delta$ and $I_1=(a_1,b_1)$ be a bounded interval. Assume $\SS(F)\cap T^{*,<0}(W\times\bR)\subset T^{*}(B_{\eta_1}\times I_1)$ and the same for $G$. Let 
\begin{align}\label{eq: cor epsilon b1a1}
\epsilon>\frac{b_1-a_1}{\delta-\eta_1}
\end{align}
(see Figure \ref{figure: cor: F,G,a1b1,eta1eta2}). 
Then the projector $\rho^\epsilon_{\epsilon,\delta}: \Sh^{<0, D^*_\epsilon W\times\bR}(W\times\bR)\to \Sh^{<0, D^*_\epsilon B_\delta\times\bR}(B_\delta\times\bR)$ (the right adjoint of $j_!$ in the right column of (\ref{diagram: rho_epsilon,delta})) induces a natural isomorphism
\begin{align*}
\Hom(\rho_\epsilon(F), \rho_\epsilon(G))\overset{\sim}{\to} \Hom(\rho_{\epsilon,\delta}(F), \rho_{\epsilon,\delta}(G))
\end{align*}
\end{cor}

\begin{proof}
By adjunction, it is equivalent to prove 
\begin{align*}
\Hom(\rho_\epsilon(F), G)\overset{\sim}{\to} \Hom(\rho_{\epsilon,\delta}(F), G),
\end{align*}
where we simply write $\rho_\epsilon(F)$ for $i_\epsilon\rho_\epsilon(F)$ and $\rho_{\epsilon,\delta}(F)$ for $i_{\epsilon, \delta}\rho_{\epsilon,\delta}(F)$. 
Using Lemma \ref{lemma: Q, limits, colimits}, and that all of $\rho_\epsilon$, $\rho_{\epsilon,\delta}$ and $\rho^\epsilon_{\epsilon,\delta}$ preserve colimits and limits (Corollary \ref{cor: proj cU_2, cU_1}), it suffices to prove the case for $F=M_{\{x\}\times (-\infty, u]}$ and $G=T_{(-\infty,0]}(M_{(y,v)})$, where $x,y\in \overline{B}_{\delta_1}$ and $u, v\in I_1$. Using the adjunction of $T_{(-\infty, 0]}$ (as a right adjoint of the inclusion $\Sh^{<0}(W\times\bR)\hookrightarrow \Sh(W\times\bR)$), we can replace $G$ by $M_{(y,v)}$. 

Let $U=W\times (a_1,\infty)$ and $j_U: U\hookrightarrow W\times\bR$ be the open inclusion. 
Applying Lemma \ref{lemma: rho_delta, P_delta} and the condition on $\epsilon$ (\ref{eq: cor epsilon b1a1}), we get
\begin{align*}
\Hom(\rho_{\epsilon, \delta}(F), G)\cong \Hom(j_{U}^!\rho_{\epsilon, \delta}(F), G)\cong \Hom(j_U^!\rho_\epsilon(F), G)\cong \Hom(\rho_\epsilon(F),G). 
\end{align*}
\end{proof}

\begin{figure}[htbp]
\begin{center}
\begin{tikzpicture}
\filldraw[fill=gray] (-0.5, 0.3)--(0.5,0.3)--(0.5,-0.3)--(-0.5,-0.3) --(-0.5, 0.3);
\draw [decorate,decoration={brace,amplitude=2pt},xshift=0pt,yshift=3pt]
(-0.5,0.3) -- (0.5,0.3) node [black,midway,yshift=0.3cm] 
{\footnotesize $Q_2$};
\draw [decorate,decoration={brace,amplitude=2pt,mirror,raise=9pt},xshift=-4pt,yshift=0pt]
(0.5,-0.3) -- (0.5,0.3) node [black,midway,xshift=1.4cm] 
{\footnotesize $I_2=(a_2,b_2)$};
\filldraw[fill=gray] (-4, -0.5)--(-3,-0.5)--(-3,-1.1)--(-4,-1.1) --(-4, -0.5);
\draw [decorate,decoration={brace,amplitude=2pt},xshift=0pt,yshift=3pt]
(-4,-0.5) -- (-3,-0.5) node [black,midway,yshift=0.3cm] 
{\footnotesize $Q_1$};
\draw [decorate,decoration={brace,amplitude=2pt},xshift=-3pt,yshift=0pt]
(-4,-1.1) -- (-4,-0.5) node [black,midway,xshift=-1cm] 
{\footnotesize $I_1=(a_1,b_1)$};
\draw (-0.5, 0.3)--(-3, -1.3);
\draw (-3, -0.5)--(-0.5, -2.1);
\draw(-0.5,0.3) -- (-0.7,0.3) node[right, above] {\tiny{$\arctan \epsilon\ \ \ \ \ \ \ \ \ \ \ \ \ $}} arc (180:180+29:0.2) -- cycle;
\draw[dashed] (-0.7,0.3)--(-1.2,0.3);
\draw (-3, -0.5) -- (-2.8,-0.5) node[right, above] {\tiny{$\ \ \ \ \ \ \ \ \ \ \ \arctan \epsilon$}} arc (0:-29:0.2) -- cycle;
\draw[dashed] (-2.8,-0.5)--(-2.3,-0.5);
\end{tikzpicture}
\end{center}
\caption{An illustration of the assumptions in Corollary \ref{lemma: disjoint support}.}\label{figure: lemma: disjoint support}
\end{figure}

\begin{cor}\label{lemma: disjoint support}
Let $Q_1, Q_2$ be two disjoint pre-compact open subsets of $W$ such that 
\begin{align*}
d=d(\overline{Q_1}, \overline{Q_2}):=\min\{|x-y|: x\in\overline{Q_1}, y\in \overline{Q_2}\}>0. 
\end{align*}
Let $I_i=(a_i,b_i), i=1,2,$ be two bounded open intervals in $\bR$ and let $U_i=Q_i\times I_i, i=1,2$. Let 
\begin{align*}
\epsilon>\frac{\max\{|b_i-a_{i+1}|, i\in \bZ/2\bZ\}}{d}
\end{align*}
 (see Figure \ref{figure: lemma: disjoint support}). Let $\Sh^{<0}(W\times\bR)_{\overline{U}_i}$ be the (cocomplete) full subcategory of $\Sh^{<0}(W\times\bR)$ consisting of sheaves $F$ satisfying 
\begin{align*}
\SS(F)\cap T^{*,<0}(W\times\bR)\subset T^*\overline{U_i}. 
\end{align*}
 Let $\widetilde{G}_i\in \Sh^{<0}(W\times\bR)_{\overline{U}_i}$ and $G_i$ be the projection of $\widetilde{G}_i$ to $\Sh^{<0,D_\epsilon^*W\times\bR}(W\times\bR)$, for $i=1,2$. Then $\Hom(G_1, G_2)= 0$. 
\end{cor}

\begin{proof}
First, the condition on $\epsilon$ is exactly to make the following hold:
\begin{align}\label{eq: condition, delta, Cx,a,delta}
(\bigcup_{(x,a)\in \overline{U_i}}C_{x, a, \epsilon})\cap \overline{U_{i+1}}=\varnothing, i\in \bZ/2\bZ.
\end{align}

By Lemma \ref{lemma: Q, limits, colimits}, $\Sh^{<0}(W\times\bR)_{\overline{U}_i}$ is generated by sheaves of the form 
\begin{align*}
\widetilde{F}_{i,M_i}=(M_i)_{\{x_i\}\times (-\infty, a_i]}, x_i\in \overline{Q_i}, a_i\in \overline{I}_i, M_i\in \Mod(\bk)
\end{align*}
both under colimits and limits. Let $F_{i,M_i}$ be the projection of $\widetilde{F}_{i,M_i}$ in $\Sh^{<0,D_\epsilon^*W\times\bR}(W\times\bR)$. Since the projector $\Sh^{<0}(W\times\bR)\to \Sh^{<0,D_\epsilon^*W\times\bR}(W\times\bR)$ preserves both limits and colimits (Corollary \ref{cor: proj cU_2, cU_1}), 
 It suffices to show that for any such $\widetilde{F}_{1,M_1}$ and $\widetilde{F}_{2, M_2}$, we have $\Hom(F_{1,M_1}, F_{2,M_2})=0$. 
 
 By Lemma \ref{lemma: rho_delta, P_delta} (i), we have the fiber sequence
 \begin{align*}
F_{i,M_i}\cong (M_i)_{C_{x_i, a_i, \epsilon}}\to \widetilde{F}_{i,M_i} \to \Sigma (M_i)_{C'_{x_i, a_i, \epsilon}}. 
\end{align*}
By the assumption (\ref{eq: condition, delta, Cx,a,delta}), we directly see that $\Hom(F_{1,M_1}, F_{2,M_2})=0$. 
\end{proof}

\begin{cor}\label{cor: C, Phi_C, ep,delta}
Let $C\subset D_\ep^*W\times\bR$ be a closed subset. Let $\ep, \delta>0$. Assume that $\XZ(C)\subset \overline{B}_{\eta_1}\times [a_1,b_1]$ such that (\ref{eq: cor epsilon b1a1}) holds. Then the natural functor from restriction of localization (i.e. the analogue of (\ref{eq: Phi', ep, delta, localization}) for $W$ replacing $V^*$) 
\begin{align*}
\Phi_C: \Sh^{<0, D_\ep^*W\times\bR}_{C}(W\times\bR)\to \Sh^{<0, D_\ep^*B_\delta\times\bR}_{C}(W\times\bR)
\end{align*}
is an equivalence. 
\end{cor}

\begin{proof}
Let $\cQ=T^*\overline{B}_{\eta_1}\times[a_1,b_1]\subset T^*W\times\bR$. Let $\cQ_\epsilon=\cQ\cap (D_\epsilon^*W\times \bR)$. By assumption, $\cQ_{\epsilon}\subset D_\epsilon^*B_\delta\times \bR$. 
 
First, we claim that both 
\begin{align}
\label{eq: cor: C, Phi_C, ep,delta}&\rho_{\epsilon;\cQ}:=\rho_\epsilon|_{\Sh_{\cQ}^{<0}(W\times\bR)}: \Sh_{\cQ}^{<0}(W\times\bR)\to \Sh_{\cQ_\ep}^{<0, D_\ep^*W\times\bR}(W\times\bR)\\
\nonumber&\rho_{\epsilon,\delta; \cQ}:=\rho_{\epsilon,\delta}|_{\Sh_{\cQ}^{<0}(W\times\bR)}: \Sh_{\cQ}^{<0}(W\times\bR)\to \Sh_{\cQ_{\ep}}^{<0, D_\ep^*B_\delta\times\bR}(W\times\bR)
\end{align}
are essentially surjective. For any $F\in \Sh_{\cQ_\ep}^{<0, D_\ep^*W\times\bR}(W\times\bR)$, if 
\begin{align*}
\Hom(F, \rho_{\ep} (M_{\{x\}\times (a,\infty)}))=0,\ \forall (x,a)\in \overline{B}_{\eta_1}\times [a_1,b_1],\ M\in \Mod(\bk), 
\end{align*}
then by adjunction, we have $\Hom(F, M_{\{x\}\times (a,\infty)})=0$ for such $(x,a)$ and $M$. This implies that $\SS(F)\cap \Cone(\cQ)=\varnothing$, and so $F=0$. The same argument applies to $F\in  \Sh_{\cQ_{\ep}}^{<0, D_{\ep}^*B_\delta\times\bR}(W\times\bR)$. Hence the claim follows. 

Second, we observe that 
\begin{align*}
\Phi_{\cQ_\ep}: \Sh^{<0, D_\ep^*W\times\bR}_{\cQ_\epsilon}(W\times\bR)\to \Sh^{<0, D_\ep^*B_\delta\times\bR}_{\cQ_\ep}(W\times\bR)
\end{align*}
is an equivalence. This is a direct consequence of the above claim and  Corollary \ref{cor: F,G,a1b1,eta1eta2}. 

Lastly, since $\SS(\Phi_{\cQ_\ep}(F))\cap \Cone(\cQ_\ep)=\SS(F)\cap \Cone(\cQ_\ep)$ for any $F\in  \Sh^{<0, D_\ep^*W\times\bR}_{\cQ_\epsilon}(W\times\bR)$, we see that $\Phi_C$ is an equivalence. 

\end{proof}

\subsection{$\widehat{\Phi}$ from (\ref{diagram: widehatPhi}) is an equivalence for $0<\epsilon\ll \delta$}
Now we are under the assumptions in Proposition \ref{prop: K_chi, Modk}. 
Consider the following diagram: 
\begin{equation}\label{diagram: chi, Omega'}
\xymatrix{ 
\Sh^{<0, D^*_\epsilon V^*\times\bR}(V^*\times\bR)\ar[d]^{\Phi'}&\Sh^{<0, D_\epsilon^*V^*\times\bR}_{\chi_\sfF(\Omega)}(V^*\times\bR)\ar@{_{(}->}[l]_{\iota_1}\ar[d]^{\Phi}\\
\Sh^{<0, D_\epsilon^*B^*_\delta\times\bR}(B^*_\delta\times\bR)\ar@/^1pc/[u]^{L_1'}\ar@/_1pc/[u]_{R_1'}&\Sh^{<0,D^*_\epsilon B^*_\delta\times\bR}_{\chi_\sfF(\Omega)}(B^*_\delta\times\bR)
\ar@/^1pc/[u]^{L_1}\ar@/_1pc/[u]_{R_1}\ar@{_{(}->}[l]_{\iota_{\delta,1}}
}
\end{equation}
that is related to (\ref{diagram: widehatPhi}) under $K_{\chi_\sfF}\circ$.

\begin{lemma}\label{lemma: Phi equivalence}
For any fixed $\delta>0$, $\Phi$ is an equivalence for $0<\epsilon\ll \delta$, and so is $\widehat{\Phi}$. 
\end{lemma}
\begin{proof}
We will apply Corollary \ref{cor: C, Phi_C, ep,delta}. For this, we look at $\XZ(\chi_\sfF(\Omega))$ in $V^*\times\bR$. It is clear that there exist functions $h(\ep)>0$ and $\nu(\epsilon)>0$ such that 
\begin{itemize}
\item[(1)] $h(\ep)\to 0$ and $\nu(\ep)\to 0$ as $\ep\to 0$;
\item[(2)] $\XZ(\chi_\sfF(\Omega))\subset B_{\nu(\ep)}\times (-\ep h(\ep), \ep h(\ep))$. 
\end{itemize}
This follows from the inequality, for any smooth $\gamma: [0,1]\to \Omega_\pi$, 
\begin{align*}
|\int_\gamma \bq d\bp|<\epsilon \int_\gamma |d\bp|
\end{align*}
and that $\int_\gamma \bq d\bp$ only depends on the endpoints of $\gamma$. Note that for any fixed $\delta>0$, we have 
\begin{align*}
\epsilon>\frac{2\ep h(\ep)}{\delta-\nu(\ep)}\text{ as }\ep\to 0. 
\end{align*} 
Now Corollary \ref{cor: C, Phi_C, ep,delta} applies for $C=\chi_\sfF(\Omega)\cap (D_\epsilon^*V^*\times\bR)$, $\eta_1=\nu(\ep)$,  $(a_1,b_1)=(-\ep h(\ep), \ep h(\ep))$, and $0<\ep\ll \delta$. 

\end{proof}

\subsection{Proof of Proposition \ref{prop: K_chi, Modk}}\label{subsec: proof of prop: K_chi, Modk}

We will use Lemma \ref{lemma: Phi equivalence} to perform two ways of calculating 
\begin{align}\label{eq: cM_0;0}
\xymatrix@R=.4pc{
\cM_{(0;0)}^{\pi}:\ar@{=}[r]&\varinjlim_{(0;0)\in\cU} \Sh^{<0,\cU\times \bR}_{\Omega}(V\times\bR)\\
\ar[r]^{\sim\ \ \ \ \ \ \ \ \ \ \ \ \ \ \ \ \ \ \ \ \ \ \ \ \ \ \ \ \ }_{K_{\chi_\sfF}\circ\ \ \ \ \ \ \ \ \ \ \ \ \ \ \ \ \ \ \ \ \ \ \ \ \ \ \ \ \ } &\varinjlim_{(0;0)\in\cU'} \Sh^{<0,\cU'\times \bR}_{\chi_\sfF(\Omega)}(V^*\times\bR)\\
\ar[r]^{\sim\ \ \ \ \ \ \ \ \ \ \ \ \ \ \ \ \ \ \ \ \ \ \ \ \ \ \ \ \ }_{K_{\chi}\circ K_{\chi_\sfF}^{-1}\circ\ \ \ \ \ \ \ \ \ \ \ \ \ \ \ \ \ \ \ \ \ \ \ \ \ \ \ \ \ }  & \varinjlim_{(0;0)\in\cU'} \Sh^{<0,\cU'\times \bR}_{\chi(\Omega)}(V^*\times\bR)
}
\end{align} 
where $\cU\subset T^*V$ and $\cU'\subset T^*V^*$ respectively runs over all open subsets that contain $(0;0)$. One may interpret $\cM_{(0;0)}^{\pi}$ as the stalk of the microlocal sheaf category on $\Omega_\pi$ at $(0;0)$. Note that this stalk, by definition, is different from the stalk of $\text{MSh}_{\Cone(\Omega)\cup\zeta_{V\times\bR}}^{\mathrm{p}}$ at $(0,0;0,-dt)$ from \S\ref{subsubsec: MSh_Z}. However, see Corollary \ref{cor: cM0;p0toMShp} below. \\

\noindent\emph{Approach 1}: Using  Lemma  \ref{lemma: Phi equivalence}, we get 
\begin{align*}
\cM_{(0;0)}^\pi\simeq \varinjlim_{\delta\gg \epsilon>0} \Sh^{<0, D_\epsilon^*B_\delta^*\times \bR}_{\chi_{\sfF}(\Omega)}(V^*\times\bR)\overset{\sim}{\leftarrow}  \Sh_{\Omega}^{<0}(B_\epsilon\times\bR).
\end{align*}

\noindent\emph{Approach 2}: $\chi(\Omega)_\pi\cap D_\epsilon^*B^*_\delta$ has wavefront a closed smooth hypersurface $H_\delta$ in $B^*_{\delta}\times \bR$ projecting to $B^*_\delta$ isomorphically, where $H_\delta$ is the graph $t=f(v^*)$ for some function $f$ on $B_\delta^*$. By a change of variables $t\mapsto t-f(v^*)$, which doesn't affect the third line of (\ref{eq: cM_0;0}) above, we may replace $\Cone(\chi(\Omega))\cap T^{*,<0}(B_\delta^*\times \bR)$ by the negative conormal bundle of $B_\delta^*\times \{0\}\subset B_\delta^*\times \bR$. Denote the Legendrian boundary of the negative conormal bundle of $V^*\times\{0\}$ in $V^*\times\bR$ by $\Omega_0'$. 

Let $0<\varrho\ll \ep$. Then applying $K_{\chi_\sfF}^{-1}\circ$ (and later $K_{\chi_\sfF}\circ$) and Corollary \ref{cor: C, Phi_C, ep,delta}, we get equivalences 
\begin{align*}
&\Sh^{<0, D_\epsilon^*B_\varrho^*\times \bR}_{\Omega'_0}(V^*\times \bR)\overset{\sim}{\to }\Sh^{<0, D^*_{\varrho}B_\epsilon\times \bR}_{\chi_\sfF^{-1}(\Omega_0')}(V\times \bR)\simeq \Sh^{<0, D^*_\varrho V\times\bR}_{\chi_\sfF^{-1}(\Omega_0')}(V\times \bR)\\
&\overset{\sim}{\to} \Sh^{<0}_{\Omega'_0}(B_\varrho^*\times\bR)\simeq \Mod(\bk),
\end{align*}
where $\chi_\sfF^{-1}(\Omega_0')_\pi$ is the cotangent fiber at $0\in V$, and the front projection of $\chi_\sfF^{-1}(\Omega_0')$ is just the point $(0,0)\in V\times \bR$. 

Hence 
\begin{align*}
\cM^\pi_{(0;0)}&\simeq \varinjlim_{0<\varrho\ll \epsilon} \Sh^{<0, D_\epsilon^*B_\varrho^*\times \bR}_{\Omega'_0}(V^*\times \bR)\simeq  \varinjlim_{0<\varrho}\Sh^{<0}_{\Omega'_0}(B_\varrho^*\times\bR)\\
\simeq&\Mod(\bk).
\end{align*}

Combining \emph{Approach 1} and \emph{Approach 2}, we directly get the second equivalence in the proposition. Hence the proof is complete.

\subsection{Proof of (\ref{eq:L-chiL})}\label{subsec: Appendix Proof of eq:L-chiL}
Now (\ref{eq:L-chiL}) follows immediately from Lemma \ref{lemma: Sh_Omega, <0} and Proposition \ref{prop: K_chi, Modk}.

\begin{cor}\label{cor: cM0;p0toMShp}
Let $\Omega$ be a smooth Legendrian ball as \S\ref{subsubsec: smooth Legendrian ball}. Then the natural functor $\cM^\pi_{(0;p_0)}\to (\MSh_{\Cone(\Omega)\cup \zeta_{V\times\bR}}^{\mathrm{p}})_{(0,0;p_0, -dt)}$, induced from localization, is an equivalence. 
\end{cor}

\begin{proof}
Without loss of generality, we may assume $p_0=0$. The only difference between $\cM_{(0;0)}^\pi$ (\ref{eq: cM_0;0}) and the stalk of $\MSh_{\Cone(\Omega)\cup \zeta_{V\times\bR}}^{\mathrm{p}}$ at $(0,0;0, -dt)$ is that we need to cut the open subsets $\cU$ (resp. $\cU'$) in the 1-jet bundle by small values of $|t|$ in the expression(s) of the colimit(s). 

Under the same setting of \S\ref{subsec: proof of prop: K_chi, Modk}, choose $1\gg \eta\gg \delta\gg \epsilon$, it suffices to show that the restriction of the localization functor $P_{D_\epsilon^*B_\delta^*\times\bR, D_\epsilon^*B_\delta^*\times (-\eta, \eta)}$ (Corollary \ref{cor: proj cU_2, cU_1})
\begin{align*}
\Phi_{(-\eta,\eta)}: &\Sh_{\chi_\sfF(\Omega)}^{<0,D_\epsilon^*B_\delta^*\times\bR}(B_\delta^*\times\bR)\to \Sh_{\chi_\sfF(\Omega)}^{<0,D_\epsilon^*B_\delta^*\times (-\eta, \eta)}(B_\delta^*\times (-\eta,\eta)), 
\end{align*}
is an equivalence. Using $\XZ(\SS(\chi_\sfF(\Omega)))\subset B^*_{\nu(\ep)}\times (-\ep h(\ep), \ep h(\ep))$ (notation from the proof of Lemma \ref{lemma: Phi equivalence}) and Lemma \ref{lemma: Q, I, cE}, we know that $\Phi_{(-\eta,\eta)}$ is the same as the pullback along $B_\delta^*\times(-\eta, \eta)\hookrightarrow B_\delta^*\times\bR$, for $\eta\gg \ep$. 
The fully faithfulness then follows from the same argument as in the proof of Corollary \ref{cor: F,G,a1b1,eta1eta2}. 
The essential surjectivity of $\Phi_{(-\eta,\eta)}$ follows from the same proof for Corollary \ref{cor: C, Phi_C, ep,delta}. 
\end{proof}

\begin{remark}\label{remark: cM0;p0toMShp}
There is a straightforward generalization of Corollary \ref{cor: cM0;p0toMShp} in the following setting. One can replace $\Omega$ by any compact subset $C$ in $T^*V\times\bR$ such that 
\begin{itemize}
\item[(1)] $C\cap (T^*_0V\times\bR)=\{(0,0;p_0=0)\}$;

\item[(2)] there exist sequences $\ep_n>0$ and $h_n>0$ with $\ep_n\to 0$ and $h_n\to 0$ such that $C\cap (T^*B_{\ep_n}\times\bR)\subset T^*B_{\ep_n}\times (-\ep_n h_n, \ep_n h_n)$, for $n\gg 1$.
\end{itemize}
It is clear that there exists a function $\nu(r)>0$ with $\lim_{r\to 0^+}\nu(r)=0$, such that $C\cap (T^*B_{\ep_n}\times\bR)\subset D^*_{\nu(\ep_n)}B_{\ep_n}\times (-\ep_n h_n, \ep_n h_n)$, for $n\gg1$.  
Then under the contact Fourier transform, we have 
\begin{align*}
\XZ\big(\chi_\sfF(C)\cap (D_{\ep_n}^*B_\delta^*\times\bR)\big)\subset B_{\nu(\ep_n)}^*\times (-\ep_n (h_n+\nu(\ep_n)), \ep_n (h_n+\nu(\ep_n))), 
\end{align*}
for $n\gg 1$. Then the argument in Corollary \ref{cor: cM0;p0toMShp} certainly applies. This says that 
\begin{align*}
(\MSh_{\Cone(C)\cup \zeta_{V\times\bR}}^{\mathrm{p}})_{(0,0;0, -dt)}\simeq \varinjlim_{\ep\to 0}\Sh^{<0}_C(B_\ep\times \bR)\simeq \varinjlim_{\ep\to 0}\Sh_C(B_\ep\times\bR)/\Loc(B_\ep\times\bR). 
\end{align*}
This recovers a special case of \cite[Lemma 7.7]{Nadler-Shende} (when one takes colimits in $\widehat{\Cat}_\infty$), which is sufficient for many applications. For example, assuming $C$ is a compact subanalytic Legendrian in $T^*\overline{B_r}\times\bR$, for some $r>0$,  satisfying condition (1), then condition (2) is automatic. 

\end{remark}

\section{Proof of descent (\ref{eq: Sh descent})}\label{appendix B, descent}

In this appendix, we give a proof of (\ref{eq: Sh descent}) on the sheaf property of the functor $\Sh(-;\bk)$\footnote{We thank Peter Haine for telling us this proof and for sharing \cite{HPT}.}.
Let $\widehat{\Cat}_\infty$ denote the $\infty$-category of (not necessarily small) $\infty$-categories. Recall $\PrL$ and $\PrR$ from \cite[Definition 5.5.3.1]{higher-topoi}. 
Let $\cC$ be an $\infty$-category with colimits and pullbacks. 
Consider the natural functor 
\begin{align}\label{eq: cCopCunderc}
\cC^{op}&\to \Cat_\infty\\
\nonumber c&\mapsto \cC_{/c}\\
\nonumber (f: c\to c')\text{ in }\cC&\mapsto (f^*: \cC_{/c'}\to \cC_{/c}, (x\to c')\mapsto (x\underset{c'}{\times}c\to c)). 
\end{align}
Note that $f^*$ is the right adjoint to the obvious functor $\cC_{/c}\to \cC_{/c'}$ given by composition with $f$. 

\begin{definition}
Let $\cC$ be an $\infty$-category with colimits and pullbacks. We say that colimits in $\cC$ are \emph{van-Kampen} if for any diagram $I\to \cC, i\mapsto c_i$, the natural functor induced from (\ref{eq: cCopCunderc})
\begin{align*}
\cC_{/\varinjlim\limits_{i\in I}c_i}\to \varprojlim_{i\in I} \cC_{/c_i}
\end{align*}
is an equivalence. 
\end{definition}

\begin{lemma}\label{lemma: colimits van-Kampen}
A presentable $\infty$-category $\cX$ is an $\infty$-topos if and only if colimits in $\cX$ are van-Kampen. 
\end{lemma}

\begin{proof}

First, by \cite[Theorem 6.1.3.9 (3)$\Leftrightarrow$(4), Theorem 6.1.0.6]{higher-topoi}, a presentable $\infty$-category $\cX$ is an $\infty$-topos if and only if $\cX^{op}\to \PrL$ (as in (\ref{eq: cCopCunderc}) but with target changed to $\PrL$) is limit preserving. Since $\PrL\to \widehat{\Cat}_\infty$ preserves limits (\cite[Proposition 5.5.3.13]{higher-topoi}), the latter is equivalent to that colimits in $\cX$ are van-Kampen. 
\end{proof}

In the following, let $X$ be a topological space. Recall $\Shv(X)$, the $\infty$-category of sheaves of spaces, from \cite[Definition 6.2.2.6]{higher-topoi} (see also \cite[\S 1.3.1]{SAG}). 
\begin{cor}\label{cor: Shv, descent}
Let $\{U_i\}_{i\in I}$ be a covering sieve of $X$. Then the natural functor induced from restrictions $\Shv(X)\to \Shv(U_i)$
\begin{align}\label{eq: cor ShvX, ShvU_i}
\Shv(X)\to \varprojlim_{i\in I}\Shv(U_i)
\end{align}
is an equivalence. 
\end{cor}

\begin{proof}
Let $j: \Open(X)\to \Shv(X)$ be the Yoneda embedding. 
For any open $U\subset X$, there is a canonical equivalence $\Shv(X)_{/j(U)}\simeq \Shv(U)$. For any covering sieve $\{U_i: i\in I\}$ of $X$, we have $\varinjlim_{i\in I}j(U)\overset{\sim}{\to} j(X)$ in $\Shv(X)$. Since $\Shv(X)$ is an $\infty$-topos \cite[Proposition 6.2.2.7]{higher-topoi}, by Lemma \ref{lemma: colimits van-Kampen}, we have an equivalence
\begin{align*}
\Shv(X)\simeq \Shv(X)_{/\varinjlim_{i\in I}j(U_i)}\overset{\sim}{\to} \varprojlim_{i\in I}\Shv(X)_{/j(U_i)}\simeq \varprojlim_{i\in I}\Shv(U_i),
\end{align*}
canonically isomorphic to (\ref{eq: cor ShvX, ShvU_i}). 
\end{proof}

For any presentable $\infty$-category $\cE$, let $\Sh(X;\cE)$ denote the presentable $\infty$-category of $\cE$-valued sheaves on $X$. By \cite[\S 1.3.1]{SAG}, there is a natural identification $\Sh(X;\cE)\simeq \Shv(X)\otimes \cE$, where the tensor product is taken in $\PrL$. 

We have the following lemma (see also \cite[Lemma 4.2.2]{HPT}). 
\begin{lemma}\label{lemma: cC_j, cE, limit}
Let $\{\cC_j\}_{i\in J}$ be a diagram in $\PrL$ such that the functors involved in the diagram have both left and right adjoints. Then the natural functor 
\begin{align}\label{eq: cC_j, cE, limit}
(\varprojlim_{j\in J^{op}}\cC_j)\otimes \cE\to \varprojlim_{j\in J^{op}}(\cC_j\otimes\cE)
\end{align} 
is an equivalence. 
\end{lemma}

\begin{proof}
Since both forgetful functors $\PrL\to \widehat{\Cat}_\infty$, $\PrR\to \widehat{\Cat}_\infty$ preserve limits \cite[Proposition 5.5.3.13, Theorem 5.5.3.18]{higher-topoi} and  the functors involved in the diagram have both left and right adjoints, the functor (\ref{eq: cC_j, cE, limit}), lying in both the 1-full subcategory $\PrL$ and $\PrR$ of $\widehat{\Cat}_\infty$, is equivalently induced from the family of tautological functors $(\varprojlim_{j\in J^{op}}\cC_j)\otimes\cE\to \cC_j\otimes\cE, j\in J$, in $\PrR$. Under the equivalence $(\PrL)^{op}\simeq \PrR$, the functor (\ref{eq: cC_j, cE, limit}) viewed in $\PrR$ is corresponding to $\varinjlim_{j\in J^{op}}(\cC_j\otimes\cE)\to (\varinjlim_{j\in J^{op}}\cC_j)\otimes\cE$ in $\PrL$. By \cite[Remark 4.8.1.24]{higher-algebra}, $(-)\otimes \cE$ preserves colimits in $\PrL$, hence (\ref{eq: cC_j, cE, limit}) is an equivalence. 

\end{proof}

\begin{cor}\label{cor: ShX; cE, descent}
Let  $\cE$ be any presentable $\infty$-category. 
Let $\{U_i\}_{i\in I}$ be a covering sieve of $X$. Then the natural functor induced from restrictions $\Sh(X;\cE)\to \Sh(U_i;\cE)$
\begin{align}\label{eq: cor ShX, ShU_i, cE}
\Sh(X;\cE)\to \varprojlim_{i\in I}\Sh(U_i;\cE)
\end{align}
is an equivalence. 
\end{cor}

\begin{proof}
Using Corollary \ref{cor: Shv, descent} and Lemma \ref{lemma: cC_j, cE, limit}, we have 
\begin{align*}
\Sh(X;\cE)\simeq  \Shv(X)\otimes \cE\overset{\sim}{\to} (\varprojlim_{i\in I}\Shv(U_i))\otimes \cE\overset{\sim}{\to }\varprojlim_{i\in I}(\Shv(U_i)\otimes \cE)\simeq \varprojlim_{i\in I}\Sh(U_i;\cE). 
\end{align*}
Clearly the composition is naturally identified with (\ref{eq: cor ShX, ShU_i, cE}). 
\end{proof}

\subsection{Proof of (\ref{eq: Sh descent}).} Now (\ref{eq: Sh descent}) is the special case of Corollary \ref{cor: ShX; cE, descent}, where $\cE=\Mod(\bk)$.

\end{document}